\pdfoutput=1

\documentclass{amsart}
\usepackage{amsmath,amssymb}

\usepackage[T1]{fontenc}
\usepackage{newtxtext}
\usepackage{newtxmath}

\usepackage[margin=1in]{geometry}
\usepackage{tikz-cd}
\usepackage[hidelinks,linktoc=all]{hyperref}

\hypersetup{
  pdftitle={Meet obstructions and saturation for the constant window convolution on graded posets},
  pdfauthor={Shinobu Yokoyama},
  pdfsubject={Algebraic topology (math.AT); MSC2020 Primary 55N31; Secondary 18G80, 18A40, 06A11, 05E45},
  pdfkeywords={cellular sheaf; face poset; window convolution; thickening kernel; incidence algebra; homotopy colimit; bar construction; homotopy finality; interleaving distance; saturation}
}

\newtheorem{theorem}{Theorem}[subsection]
\newtheorem{lemma}[theorem]{Lemma}
\newtheorem{proposition}[theorem]{Proposition}
\newtheorem{corollary}[theorem]{Corollary}
\newtheorem{conjecture}[theorem]{Conjecture}
\theoremstyle{definition}
\newtheorem{definition}[theorem]{Definition}
\newtheorem{example}[theorem]{Example}
\newtheorem{problem}[theorem]{Problem}
\theoremstyle{remark}
\newtheorem{remark}[theorem]{Remark}
\theoremstyle{plain}

\numberwithin{subsection}{section}

\newcommand{\Shv}{\mathrm{Shv}}
\newcommand{\Pos}{\mathsf{P}}
\newcommand{\Ch}{\mathrm{Ch}}
\newcommand{\FinVec}{\mathrm{FinVec}}
\newcommand{\colim}{\operatorname*{colim}}
\newcommand{\id}{\mathrm{id}}
\newcommand{\len}{\ell}
\newcommand{\st}{\operatorname{st}}
\newcommand{\cl}{\operatorname{cl}}
\newcommand{\ct}{T^{*}\Pos}
\newcommand{\LC}{\mathbb{C}}
\newcommand{\RC}{\mathbb{C}}
\newcommand{\sky}{\operatorname{sky}}
\newcommand{\Wt}{W}
\newcommand{\hocolim}{\operatorname*{hocolim}}
\newcommand{\meet}{\wedge}
\newcommand{\up}{\mathbin{\uparrow}}
\newcommand{\Ka}{\mathcal{K}}
\DeclareMathOperator{\rk}{rk}
\DeclareMathOperator{\Chb}{Ch^{b}}
\DeclareMathOperator{\Db}{D^{b}}
\DeclareMathOperator{\Lan}{Lan}
\DeclareMathOperator{\Ran}{Ran}
\DeclareMathOperator{\Hom}{Hom}
\DeclareMathOperator{\srep}{srep}

\begin{document}

\title[Meet obstructions and saturation for the constant window convolution]{Meet obstructions and saturation for the constant window convolution on graded posets}

\author{Shinobu Yokoyama}
\address{Independent}
\email{yokoyama.shinobu.46c@kyoto-u.jp}

\date{August 31, 2026}


\begin{abstract}
Let $\mathsf{P}$ be a finite graded poset and $\Delta_a^{\mathsf{P}}$ the height-$a$ thickening of its diagonal, with projections $q_1,q_2$ to $\mathsf{P}$. We study the \emph{window convolution} $C_a=\operatorname{Lan}_{q_1}\circ q_2^\ast$ on $\mathrm{Shv}(\mathsf{P};k)$, a discrete analogue of convolution against a thickening kernel.
An interleaving distance needs the homotopy window convolution $\mathbb{C}_a$ to compose as a flow, $\mathbb{C}_a\mathbb{C}_b\simeq\mathbb{C}_{a+b}$; where the meet assignment $\Phi$ is total it is a functor and carries a comparison map.
Finality is sufficient, and necessary at every minimal apex and wherever the finality defect of $\Phi$ is essential; where $\Phi$ is total at a minimal apex with unit windows, it is the failure of a length-two interval above the apex to have a single interior element. The flow fails at every branching length-two interval with minimal bottom element, and with it on the face poset of every finite regular cell complex of dimension $\ge2$. It survives on tame posets, where $\mathrm{id}\Rightarrow\mathbb{C}_a$ gives a canonical extended interleaving pseudometric on $\operatorname{D^{b}}(\mathrm{Shv}(\mathsf{P};k))$; in the saturation cases computed here it takes no finite value above the length of $\mathsf{P}$, and is finite if and only if the derived colimits agree.
\end{abstract}

\maketitle

% Sections only: with the lead subsections that the three-level numbering
% requires, a subsection-level table of contents no longer fits on page one.
\setcounter{tocdepth}{1}
% The requirement is that the contents end on page one; any excess is absorbed
% here rather than by shortening the abstract or an appendix title. At eleven
% contents lines this was worth 2pt of overfull title page; the 2026-08-12
% appendix merge brought the count to ten, and the setting is kept as margin.
\enlargethispage{\baselineskip}
\tableofcontents

\section{Introduction}

Fix a finite graded poset $\Pos$; write $\ct=\{(s,t)\in \Pos\times \Pos: s\le t\}$ with the product order, $q_1,q_2\colon\ct\to \Pos$ for the two projections, and, for $a\ge0$, $\Delta_a^{\Pos}\subseteq\ct$ for the height-$a$ window \eqref{eq:window}, a discrete analogue of the thickened diagonal of Petit--Schapira \cite{PS}. The motivating case is $\Pos=\Pos(K)$, the face poset of a finite regular cell complex $K$ (Section~\ref{sec:conventions}); the standing hypothesis is only that $\Pos$ is finite and graded, and each result that uses the cell structure names it. The \emph{window convolution} is the endofunctor of $\Shv(\Pos;k)$
\[
C_a\;=\;\Lan_{q_1^\Delta}\circ(q_2^\Delta)^\ast,
\qquad
q_i^\Delta=q_i|_{\Delta_a^{\Pos}},
\]
the left Kan extension along the first projection of the window of the restriction along the second, and the discrete analogue of convolution against a thickening kernel (Definition~\ref{def:Ca}, Remark~\ref{rem:kernel}); it is right exact, not exact, one half of an adjunction $C_a\dashv C^a$. We write $\LC_a$ for the homotopy window convolution, the homotopy left Kan extension along the same projection (Definition~\ref{def:LCa}). Interleaving-type distances related to Aoki, Petit--Schapira and Berkouk--Ginot \cite{Aoki,PS,BerkoukGinot} require the family $\{\LC_a\}$ to compose like a flow, $\LC_a\LC_b\simeq\LC_{a+b}$. This paper bounds the domain of that flow for this kernel, isolates the obstruction between the bounds, and delimits what the distance measures where the flow survives.

\emph{Degree zero.} For every finite connected regular cell complex of dimension $\ge1$ and every $a\ge1$, no assignment $D$ on isomorphism classes of vector spaces computes $H^0(\Pos;C_aF)$ from $H^0(\Pos;F)$ (Theorem~\ref{thm:noshadow}). Section \ref{sec:shadow} proves this and computes $C_a$ on the indecomposable injectives, which is what the derived work below is built on.

\emph{The derived flow and its obstruction.} $\LC_a$ has a finite stalkwise bar model, and the strong flow holds on every rooted forest, chains included (Propositions~\ref{prop:bar}, \ref{prop:chain}, \ref{prop:rooted}). Where the meet assignment $\Phi$ of Definition~\ref{def:phi} is total it carries the meet comparison $\gamma_\sigma\colon\LC_a\LC_bF\to\LC_{a+b}F$, coherent in both windows (Propositions~\ref{prop:final}, \ref{prop:meet-coherence}). The obstruction to $\gamma_\sigma$ being a quasi-isomorphism is the finality defect of $\Phi$, the reduced homology of its comma fibres assembled over the fibres of $q_2^J$ into a defect module and built from the order of $\Pos$ alone (Remark~\ref{rem:weight-independent}). Vanishing is sufficient (Proposition~\ref{prop:final}) and is necessary at every minimal apex (Theorem~\ref{thm:minimal-sharp}); away from one it is necessary wherever that module is essential (Corollary~\ref{cor:sharp-essential}), as it is at every unit first window. Where $\Phi$ is total and its defect sits at the top of the window at a minimal apex, as it does throughout a unit first window, the two functors are exhibited with different values and the law itself fails, not merely the comparison (Proposition~\ref{prop:sat-nogo}). What separates the two directions is that the available test sheaves impose a $k$-homology equivalence for each principal up-set of $\Pos$, where finality asks for one for each principal up-set of the index poset (Proposition~\ref{prop:representable-test}).

\emph{The distance on the tame domain.} On tame posets (Definition~\ref{def:tame}), the transition maps $\id\Rightarrow\LC_a$ and the flow law give an extended pseudometric on $\Db(\Shv(\Pos;k))$, independent of the isomorphisms witnessing tameness (Definition~\ref{def:interleaving}, Propositions~\ref{prop:metric}, \ref{prop:choice}). The tame class contains the closed edge (Example~\ref{ex:closed-edge}) and every rooted forest, chains included (Proposition~\ref{prop:rooted}); on a rooted forest with a branch $\LC_a$ need not be concentrated in degree zero, so the distance there is not a height shift. It is bounded by the window all the same: in each saturation case computed here it takes no finite value above the length of $\Pos$, and is bounded by that length if and only if the derived colimits agree (Lemma~\ref{lem:constant-flow}). The missing-meet and branching length-two examples are why no distance is asserted on arbitrary $\Pos$.

\subsection{Related work}

The Alexandrov and cellular sheaf background follows Curry \cite{Curry}, and the module presentation $\Shv(\Pos;k)\simeq\mathrm{Mod}\text{-}k\Pos$ over the incidence algebra follows Ladkani \cite{Ladkani}. The abstract frame in which a thickening kernel is a monoidal presheaf on the parameter monoid and convolution is its multiplication is that of Kashiwara--Schapira \cite[Def.~2.2]{KS} and Petit--Schapira \cite{PS}; the derived isometry theory for such convolution distances is Berkouk--Ginot \cite[Thm.~5.10]{BerkoukGinot}, and the flow-metric formalism within which the distance of Section~\ref{sec:metric} is constructed is de Silva--Munch--Stefanou \cite{dSMS}.

Height interleavings on finite posets are Aoki \cite{Aoki}; his height window is the object our near-window (Definition~\ref{def:near-window}) inverts, and Example~\ref{ex:aoki-chain-table} carries the comparison on chains. Guzeev \cite[Thm.~6]{Guzeev} and Guzeev--Tanaka \cite[Thm.~3.8]{GT} bound an interleaving distance along a map of persistence posets whose fibres are $\varepsilon$-acyclic, respectively weakly $\varepsilon$-contractible; there the poset is the variable object and the bound carries a multiple of the target's cardinality, whereas here the poset is fixed, finality is applied to the meet functor comparing two windows upon it, and the criterion is exact with $k$-linear coefficients. Where those results are stability theorems, the analysis here is a delimitation. The (homotopy) colimits underlying $C_a$ are read throughout as derived functors of $\colim$, in the sense of Riehl \cite{Riehl} and Shulman \cite[\S5]{Shulman}; derived functors via deformations follow \cite[Ch.~2]{Riehl}, \cite{DHKS}, and the finality and Grothendieck apparatus follows \cite[Ch.~5, Ch.~8]{Riehl}, \cite[\S1, Thm.~A]{Quillen}, \cite{Thomason}, with the poset-indexed counterparts of Thomason's theorem and of Quillen's Theorem~A in Fern\'andez--Minian \cite{FM}. The diamond property of the face poset of a regular cell complex is treated by Bj\"orner \cite{Bjorner}.

\section{Conventions and the window convolution}\label{sec:conventions}

Throughout, $\Pos$ is a \emph{finite graded poset}. Write $[s,t]=\{x\in\Pos:s\le x\le t\}$ for the closed interval and $(s,t)=\{x\in\Pos:s<x<t\}$ for the open one, and $s\lessdot t$ to say that $t$ \emph{covers} $s$, that is $s<t$ with $[s,t]=\{s,t\}$. Gradedness is the existence of a rank function $\rk\colon\Pos\to\mathbb Z$ with $\rk(t)=\rk(s)+1$ whenever $s\lessdot t$, and for $s\le t$ the \emph{height}
\[
   \len(s,t)=\rk(t)-\rk(s)
\]
is the length of every maximal chain in $[s,t]$, so that $\len$ is additive, $\len(s,u)=\len(s,t)+\len(t,u)$ for $s\le t\le u$, and $\len(s,t)=0$ if and only if $s=t$. Set
\begin{equation}\label{eq:window}
\ct=\{(s,t)\in \Pos\times \Pos: s\le t\},\qquad
\Delta_a^{\Pos}=\{(s,t)\in \ct:\len(s,t)\le a\}\quad (a\in\mathbb{Z}_{\ge0}),
\end{equation}
with the product order, and $q_1,q_2\colon\ct\to \Pos$ the two projections.

The \emph{face poset} of a finite regular cell complex $K$ in Curry's sense \cite[Def.~4.1.1]{Curry} is its poset $\Pos(K)$ of nonempty cells ordered by the face relation, graded by $\rk=\dim+1$; gradedness is strictly weaker.
Appendix \ref{app:combinatorics} fixes the passages between posets, simplicial complexes and simplicial sets, and Appendix \ref{ssec:face-posets} carries Bj\"orner's characterisation of them, the witnesses separating them from the graded posets, and the diamond property that Section \ref{sec:derived} uses.

Let $k$ be a field.
Curry's Alexandrov equivalence \cite[Thm.~4.2.10]{Curry}, stated there for an arbitrary complete and cocomplete target, identifies sheaves of $k$-vector spaces on $\Pos$ (with its Alexandrov topology) with functors on $\Pos$,
\[
\Shv(\Pos;k)\;\simeq\;\mathrm{Fun}(\Pos,\mathrm{Vec}_k),
\]
under which the stalk of $F$ at $\sigma$ is $F(\sigma)$, the corestriction for $\sigma\le\tau$ is $F(\sigma)\to F(\tau)$, and sections over an open set are the limit of the stalks it contains (Appendix \ref{ssec:chains-alexandrov}). We work in the functor presentation and write $\Shv(\Pos;k)$ for it; Remark \ref{rem:kernel} sets the construction beside the sheaf-theoretic convolution of \cite{PS}. The same equivalence gives $\Shv(\ct;k)\simeq\mathrm{Fun}(\ct,\mathrm{Vec}_k)$. Write
\[
\Shv(\Pos;\FinVec_k)\;\simeq\;\mathrm{Fun}(\Pos,\FinVec_k)
\]
for the full subcategory of sheaves with finite-dimensional stalks.

For an element $x$ of a finite poset $X$, write
\[
\st(x)=\{y\in X:x\le y\},
\qquad
\cl(x)=\{y\in X:y\le x\}
\]
for the principal up-set and the principal down-set at $x$: the smallest open set containing $x$ and the smallest closed one, that is the closure of $\{x\}$ (Appendix \ref{ssec:chains-alexandrov}). Both are used for $\Pos$ and for the auxiliary index posets of Sections \ref{sec:derived}--\ref{sec:metric} alike.

\subsection{The near-window and the window convolution}

\begin{definition}[The near-window]\label{def:near-window}
For $a\ge0$ and $x\in \Pos$, the \emph{$a$-near-window} (or \emph{window-star}) at $x$ is the fibre of $\Delta_a^{\Pos}$ over $x$ under the first projection,
\[
\st_a(x):=\{u\in \Pos:(x,u)\in\Delta_a^{\Pos}\}=\{u\in\st(x):\len(x,u)\le a\}.
\]
This is the discrete thickened-diagonal object of \cite[Def.~1.3.1(a)]{PS} read at $x$, and the superlevel form of the window attached to Aoki's height-difference function \cite[Def.~3.2]{Aoki}. On a face poset, where $\rk=\dim+1$, one has $\rk(u)-\rk(v)=\dim u$ for a vertex $v$, so $\st_a(v)$ is the set of cells of dimension $\le a$ incident to $v$.
\end{definition}

The near-window is the motivating object, but it is \emph{not} functorial in $x$: for $x\le x'$ an element $u\in\st_a(x)$ need not admit any $u'\in\st_a(x')$ with $u\le u'$, so the naive superlevel colimit $\sigma\mapsto\colim_{u\in\st_a(\sigma)}F(u)$ acquires no corestrictions for $a\ge1$. On $\{v<e_1,\ v<e_2\}$ one has $\st_1(v)=\{v,e_1,e_2\}$ and $\st_1(e_1)=\{e_1\}$, the latter a single point because $e_1$ is maximal; so $e_2\in\st_1(v)$ has nowhere to go along $v\le e_1$. (This inverts Aoki's $a$-latching functor $F\mapsto L^\ell_aF$ for the height function $\ell$, which colimits over the sublevel window below $\sigma$ \cite[Def.~3.4]{Aoki}, into the superlevel cone above $\sigma$.) The covariant repair is to index instead over the comma object below.

\begin{definition}[The window-comma poset]\label{def:window-index}
For $a\ge0$ and $\sigma\in \Pos$, set
\[
J_a(\sigma):=\{(s,t)\in\Delta_a^{\Pos}:s\le\sigma\}
=\textstyle\bigsqcup_{s\le\sigma}\{s\}\times\st_a(s),
\]
viewed as the full subposet of $\Delta_a^{\Pos}$.
Call $a$ the \emph{window scale} and $\sigma$ the \emph{apex} of $J_a(\sigma)$.
Note that $\sigma\mapsto J_a(\sigma)$ is monotone in $\sigma$.
\end{definition}

\begin{definition}[The window convolution]\label{def:Ca}
For $a\ge 0$ write $q_i^\Delta:=q_i|_{\Delta_a^{\Pos}}\colon\Delta_a^{\Pos}\to \Pos$ ($i=1,2$) for the restrictions of the two projections to the window. The \emph{window convolution} is the endofunctor of $\Shv(\Pos;k)$
\[
   C_a\;:=\;\Lan_{q_1^\Delta}\circ(q_2^\Delta)^\ast ,
\]
the left Kan extension along the first projection of the restriction along the second.
\end{definition}

\begin{remark}[The word ``kernel'']\label{rem:kernel}
Petit--Schapira call a \emph{thickening kernel} on a space $X$ a monoidal presheaf on $(\mathbb R_{\ge0},+)$ with values in the monoidal category $(\Db(k_{X\times X}),\circ)$ of kernels on $X$ \cite[Def.~1.3.1(a)]{PS}.
Of those two ingredients Definition \ref{def:Ca} carries the first, the window standing in for the kernel object, and Remark \ref{rem:weighted} records the weights it may be convolved against. The ambient shorthand $q_{1\sharp}(k_{\Delta_a^{\Pos}}\otimes^{\mathbb L}q_2^\ast F)$ is not used. Extension by zero to $\ct$ forces the transition of a weight to vanish along every relation that factors through a term outside the window, and the window is not convex: on the chain $0<1<2<3$ the relation $(0,1)\le(2,3)$ between terms of $\Delta_1^{\Pos}$ factors through $(0,3)$, which is not one. The extension is therefore a functor on $\ct$ only for a weight whose transitions already vanish there, and never for the constant weight, whose transitions are identities.
\end{remark}

\begin{lemma}[Stalk formula]\label{lem:support-presentation}
For every $a\ge0$ and $F\in\Shv(\Pos;k)$, $C_aF$ is the functor
\[
   \sigma\;\longmapsto\;(C_aF)(\sigma)=\colim_{(s,t)\in J_a(\sigma)} F(t),
\]
the colimit of $(s,t)\mapsto F(t)$ with transition maps the restrictions of $F$, the corestriction along $\sigma\le\sigma'$ being induced by the full inclusion $J_a(\sigma)\subseteq J_a(\sigma')$.
\end{lemma}

\begin{proof}
For $\sigma\in \Pos$ the comma poset $(q_1^\Delta\downarrow\sigma)$ is $\{(s,t)\in\Delta_a^{\Pos}:s\le\sigma\}=J_a(\sigma)$, so the pointwise formula for left Kan extensions \cite[Thm.~6.2.1]{RiehlCTIC}, \cite[Ch.~X]{ML} gives the displayed colimit, naturally in $\sigma$ and $F$, with the corestrictions induced by the inclusions of comma posets.
\end{proof}

\begin{lemma}[Unit]\label{lem:unit}
$C_0\simeq\id_{\Shv(\Pos;k)}$.
\end{lemma}

\begin{proof}
By Lemma \ref{lem:support-presentation}, for $a=0$ the index poset at $\sigma$ is $\{(s,s):s\le\sigma\}$, which has terminal object $(\sigma,\sigma)$; the colimit is $F(\sigma)$, functorially in $\sigma$ and $F$.
\end{proof}

\subsection{Exactness, the right adjoint, and the derived functors}

The Kan-extension form of Definition \ref{def:Ca} has a matching right adjoint, from which the conventional categorical properties of $C_a$ follow at once.

\begin{definition}[The co-window limit]\label{def:Cupper}
For $G\in\Shv(\Pos;k)$ and $a\ge0$, set
\[
C^aG:=\Ran_{q_2^\Delta}(q_1^\Delta)^\ast G,
\qquad\text{pointwise}\qquad
(C^aG)(\tau)=\lim_{\substack{(s,t)\in\Delta_a^{\Pos}\\ \tau\le t}}G(s),
\]
the limit of the diagram $(s,t)\mapsto G(s)$ over the full subposet $\{(s,t)\in\Delta_a^{\Pos}:\tau\le t\}\subseteq\Delta_a^{\Pos}$; for $\tau\le\tau'$ the reverse inclusion of index posets restricts compatible families and gives the corestriction $(C^aG)(\tau)\to(C^aG)(\tau')$. We call $C^a$ the \emph{co-window} (or \emph{matching}) limit.
\end{definition}

\begin{proposition}[Adjunction]\label{prop:adjunction}
$C_a\dashv C^a$ for every $a\ge0$.
\end{proposition}

\begin{proof}
By Definition \ref{def:Ca} and the two Kan-extension adjunctions $\Lan_{q_1^\Delta}\dashv(q_1^\Delta)^\ast$ and $(q_2^\Delta)^\ast\dashv\Ran_{q_2^\Delta}$, which are the defining universal properties \cite[Def.~6.1.1, Cor.~6.2.7]{RiehlCTIC}:
\[
\Hom(C_aF,G)
\cong\Hom\bigl(\Lan_{q_1^\Delta}(q_2^\Delta)^\ast F,\,G\bigr)
\cong\Hom\bigl((q_2^\Delta)^\ast F,\,(q_1^\Delta)^\ast G\bigr)
\cong\Hom\bigl(F,\,\Ran_{q_2^\Delta}(q_1^\Delta)^\ast G\bigr),
\]
naturally in $F$ and $G$.
\end{proof}

\begin{corollary}[Right exactness and the dual unit]\label{cor:exactness}
$C_a$ preserves all colimits and is right exact; $C^a$ preserves all limits and is left exact \cite[Thm.~4.5.2, Cor.~4.5.9]{RiehlCTIC}. Moreover $C^0\simeq\id$, by Lemma \ref{lem:unit} and uniqueness of adjoints \cite[Prop.~4.3.1]{RiehlCTIC}.
\end{corollary}

Right exactness is strict in general but not on every $\Pos$: a minimal element with two distinct covers already makes $C_1$ inexact, while on a chain $C_a$ is exact for every $a$ (Proposition \ref{prop:inexact}). The induced endofunctor of complexes is therefore not homotopical in general, which is why Sections \ref{sec:derived}--\ref{sec:metric} are written for $\LC_a$ and not for $C_a$.

\begin{lemma}[Enough projectives]\label{lem:enough-projectives}
For $x\in \Pos$ and a $k$-vector space $M$, write $M_{\st(x)}\in\Shv(\Pos;k)$ for the constant sheaf on $\st(x)$ taking value equal to $M$, extended by zero to $\Pos$ elsewhere. This is the left Kan extension $x_!M$ of $M:\bullet\to \mathrm{Vec}_k$ along $x:\bullet\to \Pos$.
Then
\[
\Hom(x_!M,F)\;\cong\;\Hom_k(M,F(x))
\]
naturally in $M$ and $F$, each $M_{\st(x)}$ is projective, and the counit
\[
\bigoplus_{x\in \Pos}F(x)_{\st(x)}\longrightarrow F
\]
is an epimorphism from a projective. So $\Shv(\Pos;k)$ has enough projectives.
\end{lemma}

\begin{proof}
Every $k$-vector space has a basis, hence is projective.
The displayed adjunction is \cite[\S2.4]{Ladkani}, and it is natural in both variables.
The covariant hom-functor $\Hom(M_{\st(x)},-)$ preserves epimorphisms because $M$ is projective and the stalk functor is exact. The direct sum of projectives is projective, and the claim follows by definition.
\end{proof}

\begin{remark}[The module presentation]\label{rem:module-presentation}
$\Shv(\Pos;k)$ is equivalent to the category $\mathrm{Mod}\text{-}k\Pos$ of all right modules over the incidence algebra $k\Pos$.
The algebra is at Definition \ref{def:incidence-algebra} and the equivalence at Lemma \ref{lem:module-diagram}, where the finiteness of stalks is not assumed and the only finiteness used is that of $\Pos$.
$\Shv(\Pos;k)$ is abelian, has enough projectives by Lemma \ref{lem:enough-projectives}, and has enough injectives by the dual statement \cite[Cor.~2.2]{Ladkani}, with $M_{\cl(x)}$ in place of $M_{\st(x)}$.
\end{remark}

\begin{remark}[Homotopical conventions]\label{rem:homotopical}
Write $\mathcal A=\Shv(\Pos;k)\simeq\mathrm{Mod}\text{-}k\Pos$. Throughout, $\Ch(\mathcal A)$ carries the quasi-isomorphisms as its weak equivalences and nothing else; the class is saturated (every map inverted in the homotopy category is already a weak equivalence) and satisfies two-out-of-six, so $(\Ch(\mathcal A),\text{q-iso})$ is a homotopical category \cite[Def.~2.1.1, Rem.~2.1.9, Lem.~2.1.10]{Riehl}, \cite{DHKS}. We write $\Chb(\mathcal A)$ for bounded complexes with the induced weak equivalences, and $\Db(\mathcal A)$ for the localization of $\Chb(\mathcal A)$ at quasi-isomorphisms \cite[Def.~2.1.6]{Riehl}.

Of the two ingredients of $C_a=\Lan_{q_1^\Delta}\circ(q_2^\Delta)^\ast$ (Definition \ref{def:Ca}) the restriction is exact and the Kan extension is not, and it is the Kan extension alone that is derived here. Lemma \ref{lem:enough-projectives} holds verbatim on any finite poset, so $\Shv(\Delta_a^{\Pos};k)$ has enough projectives, and enough injectives by the dual statement of Remark \ref{rem:module-presentation}. Appendix \ref{app:bar} records the deformation and bar conventions used here.

Complexes are graded homologically throughout, with differential lowering degree by one. We adopt the following shift and cone conventions:
\[
   (C[p])_n=C_{n-p},
   \qquad
   d_{C[p]}=(-1)^pd_C,
\]
so that $C[1]$ is the suspension and $C[-1]$ the desuspension, and, for a chain map $f\colon X\to Y$,
\[
   \operatorname{Cone}(f)_n=X_{n-1}\oplus Y_n,
   \qquad
   d(x,y)=\bigl(-d_Xx,\ f(x)+d_Yy\bigr).
\]
Thus $\operatorname{Cone}(f)[-1]$ is the shifted cone occurring in Lemma \ref{lem:defect-reduced} and in the comparison of Section \ref{sec:derived}, and a distinguished triangle closes with a $[1]$. Weibel's translation is the opposite one, $(C[p])_n=C_{n+p}$ \cite[1.2.8]{Weibel}; the results cited from \cite{Weibel} below involve no shift, so no transcription is needed.
\end{remark}

\begin{definition}[The homotopy window convolution and its right adjoint]\label{def:LCa}
For $a\ge0$ the \emph{homotopy window convolution} and its right adjoint are
\[
   \LC_a:=\mathbb L\Lan_{q_1^\Delta}\circ(q_2^\Delta)^\ast,
   \qquad
   \RC^a:=\mathbb R\Ran_{q_2^\Delta}\circ(q_1^\Delta)^\ast ,
\]
both endofunctors of $\Db(\Shv(\Pos;k))$ (Proposition \ref{prop:derived-exist}): the homotopy left Kan extension along the first projection of the restriction along the second, and the homotopy right Kan extension along the second of the restriction along the first. We write $H_i\LC_aF$ for the homology of the one and $H^i\RC^aG$ for the cohomology of the other.
\end{definition}

A left derived functor composed with an exact functor, $\LC_a$ is triangulated; it preserves boundedness by the vanishing range of Proposition \ref{prop:bar}; and $H_0\LC_aF\cong C_aF$ for a sheaf $F$, since $\Lan_{q_1^\Delta}$ is right exact and its zeroth left derived functor is itself.

\begin{remark}[Not a derived functor of $C_a$]\label{rem:not-derived}
The derived operation is the Kan extension alone, so in general
\[
\LC_a=\mathbb L\Lan_{q_1^\Delta}\circ(q_2^\Delta)^\ast\;\not\simeq\;\mathbb L\bigl(\Lan_{q_1^\Delta}\circ(q_2^\Delta)^\ast\bigr)=\mathbb LC_a ,
\]
and no comparison between the two is claimed. The two agree if and only if $(q_2^\Delta)^\ast$ carries projectives to $\Lan_{q_1^\Delta}$-acyclic objects, equivalently if and only if for every projective $Q$ and every $\sigma$ the restriction $\left.(q_2^\Delta)^\ast Q\right\vert_{J_a(\sigma)}$ is $\colim_{J_a(\sigma)}$-acyclic. Necessity is immediate, a projective $Q$ having $H_i(\mathbb LC_aQ)=0$ for $i\ge1$. For sufficiency, take a projective resolution $Q_\bullet\to F$, which exists because $\Pos$ is finite (Lemma \ref{lem:enough-projectives}); exactness of $(q_2^\Delta)^\ast$ makes $(q_2^\Delta)^\ast Q_\bullet\to(q_2^\Delta)^\ast F$ a resolution, under the hypothesis one by $\Lan_{q_1^\Delta}$-acyclic objects, and dimension shifting along that resolution (Remark \ref{rem:dimension-shifting}) identifies the homology of $\Lan_{q_1^\Delta}(q_2^\Delta)^\ast Q_\bullet$, which computes $\mathbb LC_aF$, with that of $\mathbb L\Lan_{q_1^\Delta}\bigl((q_2^\Delta)^\ast F\bigr)$.

The hypothesis fails already at $a=2$. Take the bigon decomposition of $S^2$, with two vertices, two edges and two $2$-cells, and let $\sigma$ and $u$ be the two $2$-cells. Then $u$ is maximal, so $\st(u)=\{u\}$ and the skyscraper at $u$ is the projective $k_{\st(u)}$ of Lemma \ref{lem:enough-projectives}, whence $H_i(\mathbb LC_2k_{\st(u)})=0$ for $i\ge1$; while $(q_2^\Delta)^\ast k_{\st(u)}$ is carried on the up-set of those $(s,t)\in J_2(\sigma)$ with $t=u$, which is the face poset of the circle the two $2$-cells share, its order complex a $4$-cycle, so the bar model of Proposition \ref{prop:bar} gives $\bigl(H_1\LC_2k_{\st(u)}\bigr)(\sigma)\cong k$.
\end{remark}

\begin{proposition}[Existence of the derived functors and transition maps]\label{prop:derived-exist}
Let $a\ge0$.
\begin{enumerate}
\item[\textup{(i)}] the total left derived functor $\mathbb L\Lan_{q_1^\Delta}$ exists, and with it the homotopy window convolution $\LC_a$ of Definition \ref{def:LCa}, an endofunctor of $\Db(\Shv(\Pos;k))$; boundedness holds because $H_i\LC_a$ vanishes above the maximal chain length of the window-comma posets (Proposition \ref{prop:bar});
\item[\textup{(ii)}] dually, $\mathbb R\Ran_{q_2^\Delta}$ exists, and with it $\RC^a\colon\Db(\Shv(\Pos;k))\to\Db(\Shv(\Pos;k))$;
\item[\textup{(iii)}] the adjunction descends: $\LC_a\dashv\RC^a$ on $\Db(\Shv(\Pos;k))$;
\item[\textup{(iv)}] for $a\le b$ there are natural transformations
\[
   \tau_{a,b}\colon C_a\Rightarrow C_b,\qquad
   \eta_{a,b}\colon \LC_a\Rightarrow\LC_b,
\]
with
\[
   \tau_{a,a}=\id,\quad \tau_{b,c}\circ\tau_{a,b}=\tau_{a,c},
   \qquad
   \eta_{a,a}=\id,\quad \eta_{b,c}\circ\eta_{a,b}=\eta_{a,c}.
\]
The transformations $\eta_{a,b}$ are the \emph{derived transition maps}. After the identification $\LC_0\simeq\id$, write $\eta_a:=\eta_{0,a}\colon\id\Rightarrow\LC_a$; then $\eta_b=\eta_{a,b}\circ\eta_a$ for $a\le b$.
\end{enumerate}
\end{proposition}

\begin{proof}
Lemma \ref{lem:enough-projectives} and Remark \ref{rem:module-presentation} hold verbatim over the finite poset $\Delta_a^{\Pos}$, so $\Shv(\Delta_a^{\Pos};k)$ has enough projectives and enough injectives; existence of the derived functors of the right exact $\Lan_{q_1^\Delta}$ and the left exact $\Ran_{q_2^\Delta}$ over a module category is then standard homological algebra \cite[Thm.~10.5.6, Cor.~10.5.7]{Weibel}, and is the deformation construction recalled in Remark \ref{rem:deformation} \cite[\S2.3]{Riehl}. For (iii), the restrictions $(q_1^\Delta)^\ast$ and $(q_2^\Delta)^\ast$ are exact and so homotopical, and the deformable pairs $\Lan_{q_1^\Delta}\dashv(q_1^\Delta)^\ast$ and $(q_2^\Delta)^\ast\dashv\Ran_{q_2^\Delta}$ descend to $\mathbb L\Lan_{q_1^\Delta}\dashv(q_1^\Delta)^\ast$ and $(q_2^\Delta)^\ast\dashv\mathbb R\Ran_{q_2^\Delta}$ \cite[Thm.~2.2.11]{Riehl}, and composing those two adjunctions gives $\LC_a\dashv\RC^a$; it is obtained that way and not by deriving the adjoint pair $(C_a,C^a)$ of Proposition \ref{prop:adjunction}. The vanishing range in (i) is proved independently by the bar model of Proposition \ref{prop:bar}.

For (iv), write $\mathcal A=\Shv(\Pos;k)$ and write $q_i^c\colon\Delta_c^{\Pos}\to\Pos$ for the two projections at window scale $c$.  For $F\in\mathcal A$, let
\[
   \alpha^a_F\colon (q_2^a)^\ast F\longrightarrow (q_1^a)^\ast C_aF
\]
be the adjunction unit for $\Lan_{q_1^a}\dashv(q_1^a)^\ast$. Equivalently, $(C_aF,\alpha^a_F)$ is the initial object of the comma category
\[
   (q_2^a)^\ast F\downarrow(q_1^a)^\ast .
\]
If $a\le b$, the inclusion $i:\Delta_a^{\Pos}\subseteq\Delta_b^{\Pos}$ restricts the $b$-unit to
\[
(q_2^a)^\ast F=i^\ast(q_2^b)^\ast F\longrightarrow
i^\ast(q_1^b)^\ast C_bF=(q_1^a)^\ast C_bF.
\]
Thus $C_bF$ is also an object of $(q_2^a)^\ast F\downarrow(q_1^a)^\ast$, and initiality gives the unique morphism
\[
   \tau_{a,b,F}\colon C_aF\longrightarrow C_bF
\]
whose image under $(q_1^a)^\ast$ makes the following triangle commute:
\[
\begin{tikzcd}
& \arrow[ld] (q_2^a)^\ast F\arrow[rd] & \\
(q_1^a)^\ast C_a F \arrow[rr,dotted,"(q_1^a)^\ast\tau_{a,b,F}"'] && (q_1^a)^\ast C_b F.
\end{tikzcd}
\]
The uniqueness is compatible with morphisms in $F$ and gives a natural transformation $\tau_{a,b}\colon C_a\Rightarrow C_b$. The same initiality argument gives $\tau_{a,a}=\id$ and $\tau_{b,c}\circ\tau_{a,b}=\tau_{a,c}$ for $a\le b\le c$.

For the derived transition maps, fix $a\le b$. The inclusion $\Delta_a^{\Pos}\subseteq\Delta_b^{\Pos}$ satisfies $q_1^b\circ i=q_1^a$ and $i^\ast(q_2^b)^\ast=(q_2^a)^\ast$, so at each $\sigma$ the two homotopy colimits of Proposition \ref{prop:bar} are taken over $J_a(\sigma)\subseteq J_b(\sigma)$ with the same coefficient diagram. Define $\eta_{a,b}$ stalkwise as the bar map that inclusion induces with identity coefficients (Definition \ref{def:bar-hocolim}); the identification of Proposition \ref{prop:bar} is natural in $\sigma$ and in $F$, so this is a map of functors on $\Db(\Shv(\Pos;k))$, it inherits the two identities because the inclusions compose, and it is $\tau_{a,b}$ on $H_0$. Finally $\len(s,t)=0$ with $s\le t$ forces $s=t$, so $q_1^\Delta$ is an isomorphism $\Delta_0^{\Pos}\to\Pos$, whence $\LC_0\simeq\id$, and $\eta_a:=\eta_{0,a}$ satisfies $\eta_b=\eta_{a,b}\circ\eta_a$ for $a\le b$.
\end{proof}

\begin{example}[Chains: the two height shifts]\label{ex:chain-adjoint}
On a chain $x_0<x_1<\cdots<x_n$ the co-window index poset at $x_i$ has least element $(x_{\max(i-a,0)},x_i)$, so
\[
(C^aG)(x_i)=G\bigl(x_{\max(i-a,0)}\bigr):
\]
$C^a$ is the downward height shift, adjoint to the upward shift $C_a$ of Proposition \ref{prop:chain}.
\end{example}

\begin{example}[The branching poset: annihilation transposed]\label{ex:lambda}
On $\Lambda=\{v<e_1,\ v<e_2\}$ one has $C_1\sky_v=0$ (the annihilation mechanism of Theorem \ref{thm:noshadow}), and the adjunction transposes this to $\Hom(\sky_v,C^1G)=0$ for every $G$. Directly, $C^1\sky_v\cong k_\Lambda$, the constant sheaf: at $\tau=v$ the index poset is all of $\Delta_1^\Lambda$ and a compatible family is determined by its value over $(v,v)$; at $\tau=e_i$ the index is $\{(v,e_i),(e_i,e_i)\}$ and the limit is again $k$. Consistently, $\Hom(\sky_v,k_\Lambda)=\ker\bigl(k\to k^2\bigr)=0$.
\end{example}

\begin{remark}[The window convolution is one weight]\label{rem:weighted}
A \emph{weight} on the window is a functor $\Wt\colon\Delta_a^{\Pos}\to\mathrm{Vec}_k$, convolution against it is $F\mapsto\Lan_{q_1^\Delta}\bigl(\Wt\otimes_k(q_2^\Delta)^\ast F\bigr)$, and Definition \ref{def:Ca} is the constant weight $\Wt=k$; this is the discrete analogue of the kernel-convolution calculus of Kashiwara--Schapira and Petit--Schapira \cite{KS,PS}, with the delimitation of Remark \ref{rem:kernel}. Over a field a weight twists the same diagrams by an exact tensor factor. Nevertheless, no weighted flow statement is used below; the classification of flow-carrying weights is deferred.
\end{remark}

On a finite poset, $H^r(\Pos;F)$ denotes the right derived functors of $F\mapsto\lim_{\Pos}F$, so that $H^0(\Pos;F)=\lim_{\Pos}F$. For a face poset $\Pos=\Pos(K)$ this is cellular sheaf cohomology, computed by the cellular cochain complex $C^r(\Pos;F)=\bigoplus_{\dim\sigma=r}F(\sigma)$ with incidence signs and the restriction maps of $F$ \cite[Def.~6.2.1]{Curry}. For the constant sheaf $k_{\Pos}$ this is the cellular cochain complex of $K$, so $H^*(\Pos;k_{\Pos})\cong H^*(K;k)$. For any $\sigma\in\Pos$, $\sky_\sigma$ denotes the skyscraper sheaf with stalk $k$ at $\sigma$ and $0$ elsewhere.

\section{No degree-zero shadow}\label{sec:shadow}

\subsection{The no-shadow theorem and its minimal witness}

\begin{theorem}[No-shadow theorem]\label{thm:noshadow}
Let $K$ be a finite connected regular cell complex with at least one $1$-cell, and let $\Pos$ be its face poset; the same assertion holds for any finite connected poset of length $\ge1$. Fix $a\ge1$. There is no assignment $D$ on isomorphism classes of $k$-vector spaces, the class of $V$ written $[V]$, such that
\[
H^0(\Pos;C_aF)\;\cong\;D\bigl(H^0(\Pos;F)\bigr)\qquad\text{for all }F\in\Shv(\Pos;k);
\]
so there is no endofunctor $D_a$ of $\mathrm{Vec}_k$ making the diagram commute:
$$
\begin{tikzcd}
{\rm Shv}(\Pos;k) \arrow[d,"C_a"]\arrow[r,"H^0(\Pos;-)"] & \mathrm{Vec}_k \arrow[d,dotted,"D_a"] \\
{\rm Shv}(\Pos;k) \arrow[r,"H^0(\Pos;-)"] & \mathrm{Vec}_k.
\end{tikzcd}
$$
\end{theorem}

\begin{proof}
The assertion is a non-existence, so it suffices to exhibit two sheaves with the same $H^0(\Pos;-)$ and different $H^0(\Pos;C_a-)$.

\emph{The constant sheaf.} We claim $C_ak_{\Pos}\cong k_{\Pos}$. Fix $\sigma$. Every value of the diagram over $J_a(\sigma)$ is $k$ and every transition map is the identity, so it suffices to show $J_a(\sigma)$ is connected; then the colimit is $k$, with every structure map the identity. Connectedness: any object $(s,t)$ receives $(s,s)\le(s,t)$, and $(s,s)\le(\sigma,\sigma)$ since $s\le\sigma$, so every object is connected to $(\sigma,\sigma)$ in at most two steps. The restriction maps of $C_ak_{\Pos}$ are induced by the full inclusions $J_a(\sigma)\subseteq J_a(\sigma')$ and send the class of any object to the class of the same object, and are the identity of $k$. Thus $C_ak_{\Pos}\cong k_{\Pos}$ and, since $K$ is connected,
\[
H^0(\Pos;C_ak_{\Pos})\cong H^0(K;k)\cong k,\qquad H^0(\Pos;k_{\Pos})\cong k.
\]

\emph{A vertex skyscraper.} Since $K$ has a $1$-cell, choose a $1$-cell $f$ and a vertex $v\le f$; then $\len(v,f)=1\le a$. We claim $C_a\sky_v=0$. Fix $\sigma$. The only object of $J_a(\sigma)$ on which the diagram $(s,t)\mapsto \sky_v(t)$ is nonzero is $(v,v)$ (a pair $(s,t)$ has $\sky_v(t)\ne0$ only for $t=v$, forcing $s=v$), and $(v,v)\in J_a(\sigma)$ only when $v\le\sigma$. In that case $(v,v)\le(v,f)\in J_a(\sigma)$, and the transition map sends the generator of $\sky_v(v)=k$ to $0=\sky_v(f)$. The class of the generator vanishes in the colimit and $(C_a\sky_v)(\sigma)=0$ for every $\sigma$, so
\[
H^0(\Pos;C_a\sky_v)=0.
\]
On the other hand $C^0(\Pos;\sky_v)=k$ concentrated at $v$, and $d^0=0$ because all edge stalks of $\sky_v$ vanish; so $H^0(\Pos;\sky_v)\cong k$.

If $D$ existed, then
\[
D([k])\cong H^0(\Pos;C_ak_{\Pos})\cong k
\quad\text{and}\quad
D([k])\cong H^0(\Pos;C_a\sky_v)=0,
\]
a contradiction.

\emph{General posets.} Each use of the cell structure has a poset substitute. $C_ak_{\Pos}\cong k_{\Pos}$ was proved above for an arbitrary finite poset, and $H^0(\Pos;k_{\Pos})=\lim_{\Pos}k_{\Pos}=k$ because $\Pos$ is connected and all restrictions are identities. For the skyscraper, take a comparable pair $x<y$, choose a minimal $v\le x$, and choose a cover $f$ above $v$; then $\len(v,f)=1\le a$. Since $v$ is minimal, $\lim_{\Pos}\sky_v=k$, no $s<v$ forcing the value there to vanish, and the same arrow $(v,v)\le(v,f)$ kills $C_a\sky_v$ at every stalk.
\end{proof}

If $\dim K=0$ then $\Delta_a^{\Pos}$ is the diagonal for every $a$, so $C_a\simeq\id$ by Lemma \ref{lem:unit} and the shadow exists. Both witnesses have finite-dimensional stalks and finite-dimensional cohomology, so the theorem is unchanged if $F$ ranges over $\Shv(\Pos;\FinVec_k)$ and $D$ is defined only on isomorphism classes of finite-dimensional spaces.

The next computation shows the opposite failure mode: $C_1$ retains stalk-level incidence data that cohomology has forgotten.

\begin{proposition}[What $C_1$ retains: the minimal witness]\label{prop:edge}
Let $\Pos$ be the face poset of the $2$-simplex $\Delta^2$, with vertices $v_0,v_1,v_2$, edges $e_{ij}$, and top cell $T$. Then
\[
H^0(\Pos;C_1\sky_{e_{01}})\cong k^2,\qquad H^0(\Pos;C_1\sky_T)=0,
\]
while $H^0(\Pos;\sky_{e_{01}})=0=H^0(\Pos;\sky_T)$.
\end{proposition}

\begin{proof}
The last statement is immediate: skyscrapers on positive-dimensional cells have zero degree-zero cochains.

\emph{Stalks of $C_1\sky_{e}$, $e=e_{01}$.} The objects carrying a nonzero value are the $(s,e)$ with $s\le e$, i.e.\ $s\in\{v_0,v_1,e\}$. The outcome is the following, each cell carrying the stalk of $C_1\sky_{e_{01}}$ there in parentheses:
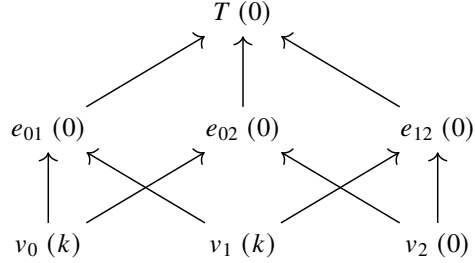
\[
\begin{tikzcd}[row sep=large,column sep=large]
& T\ (0) & \\
e_{01}\ (0)\arrow[ur] & e_{02}\ (0)\arrow[u] & e_{12}\ (0)\arrow[ul]\\
v_0\ (k)\arrow[u]\arrow[ur] & v_1\ (k)\arrow[ul]\arrow[ur] & v_2\ (0)\arrow[ul]\arrow[u]
\end{tikzcd}
\]
The diagram $(s,t)\mapsto\sky_e(t)$ on $J_1(\sigma)$ is nonzero on the objects $(s,e)$ with $s\le\sigma$ and on no others, and all of these carry the same class, since $(s,e)\le(e,e)$ whenever $(e,e)$ is present. The stalk is therefore $k$ unless the index poset supplies an arrow out of one of them into an object of value $0$, and the table lists such an arrow wherever there is one:
\[
\begin{array}{c|c|c|c}
\sigma & \text{nonzero objects of }J_1(\sigma) & \text{killing arrow} & \text{stalk}\\
\hline
v_i\ (i=0,1) & (v_i,e) & \text{none} & k\\
v_2 & \text{none} & \text{--} & 0\\
e & (v_0,e),(v_1,e),(e,e) & (e,e)\le(e,T) & 0\\
e_{02} & (v_0,e) & (v_0,e)\le(e_{02},T) & 0\\
e_{12} & (v_1,e) & (v_1,e)\le(e_{12},T) & 0\\
T & (v_0,e),(v_1,e),(e,e) & (e,e)\le(e,T) & 0
\end{array}
\]
Two verifications suffice for the whole table. That no arrow leaves $(v_i,e)$ inside $J_1(v_i)$ for $i=0,1$: such an arrow would need $t'\ge e$ with $\len(v_i,t')\le1$, forcing $t'=e$, and $s'$ with $v_i\le s'\le v_i$. That the listed arrows are arrows of the index poset and land on the value $0$: each has second coordinate $T$, where $\sky_e$ vanishes, and each satisfies the two window conditions (for instance $(v_0,e)\le(e_{02},T)$ because $v_0\le e_{02}\le T$, $e\le T$ and $\len(e_{02},T)=1$). The table gives $C^0(\Pos;C_1\sky_e)\cong k^2$ (at $v_0,v_1$); all edge stalks vanish, so $d^0=0$ and $H^0(\Pos;C_1\sky_e)\cong k^2$.

\emph{Stalks of $C_1\sky_T$ at vertices.} At $\sigma=v_i$ the window objects are $(v_i,t)$ with $t$ of dimension $\le1$, and $\sky_T$ vanishes on all of them, so the stalk is $0$; then $C^0(\Pos;C_1\sky_T)=0$ and $H^0(\Pos;C_1\sky_T)=0$.
\end{proof}

\subsection{Which sheaves \texorpdfstring{$C_a$}{C\_a} annihilates}

The module $\bigoplus_{x\in \Pos}F(x)$ attached to $F$ by Lemma \ref{lem:module-diagram} is finite-dimensional if and only if every stalk is, $\Pos$ being finite, so that equivalence restricts to an identification of $\Shv(\Pos;\FinVec_k)$ with the category of finite-dimensional right $k\Pos$-modules \cite[Lem.~2.7]{Ladkani}. In the sheaf notation used here, its indecomposable injectives are the down-set sheaves
\[
I_\sigma:=k_{\cl(\sigma)},
\]
constant equal to $k$ on the principal down-set $\cl(\sigma)$ with identity restrictions, and $0$ elsewhere. Curry proves this for a face poset, with $I_\sigma$ his elementary injective supported on the closure of $\sigma$ \cite[Def.~7.1.3, Lem.~7.1.6]{Curry}, attributing the harder direction to Shepard; over a finite poset it is the standard indecomposable-injective classification for the finite-dimensional algebra $k\Pos$, read through Remark \ref{rem:module-presentation} \cite[\S2]{Ladkani}. The classification is the one statement of the paper that needs finite dimensionality; Proposition \ref{prop:inj-annih} itself is a computation of $C_ak_{\cl(\sigma)}$ over $\mathrm{Vec}_k$.

\begin{proposition}[Exact annihilation locus of the injectives]\label{prop:inj-annih}
Let $\Pos$ be a finite poset, $\sigma\in \Pos$, and $a\ge1$. Then the indecomposable injective $I_\sigma$ survives $C_a$ if and only if some minimal element below $\sigma$ has its entire near-window among the elements below $\sigma$:
\[
C_aI_\sigma\ne 0
\quad\Longleftrightarrow\quad
\exists\ v\le\sigma \text{ minimal in }\Pos \ \text{ s.t. }\ \st_a(v)\subseteq \cl(\sigma).
\]
On a face poset the minimal elements are the vertices, and the criterion reads: some vertex of $\sigma$ has its entire near-window among the faces of $\sigma$.
\end{proposition}

\begin{proof}
Since $I_\sigma(u)=k$ for $u\le\sigma$ and $0$ otherwise, with identity restrictions on $\cl(\sigma)$, the colimit of Lemma \ref{lem:support-presentation} at a cell $x$ is presented as a quotient of the direct sum of its nonzero terms,
\[
(C_aI_\sigma)(x)\;=\;\Big(\ \bigoplus_{\substack{(s,u)\in J_a(x)\\ u\le\sigma}} k_{(s,u)}\ \Big)\Big/\!\sim,
\]
one copy $k_{(s,u)}$ of $k$ per pair $(s,u)\in J_a(x)$ with $u\le\sigma$; the relations identify $k_{(s,u)}\xrightarrow{\ \id\ }k_{(s',u')}$ along each $(s,u)\le(s',u')$ with $u'\le\sigma$, and send $k_{(s,u)}\mapsto0$ along each $(s,u)\le(s',u')$ with $u'\not\le\sigma$.

\emph{Stalk formula at a minimal element.} Let $v$ be minimal in $\Pos$. Minimality forces $s=v$ for every $(s,u)\in J_a(v)$, so $J_a(v)=\{(v,u):u\in\st_a(v)\}$ is isomorphic (via the second coordinate) to $\st_a(v)$, with least object $(v,v)$; the diagram value is $k$ on $\st_a(v)\cap \cl(\sigma)$ and $0$ elsewhere.
\begin{itemize}
\item If $\st_a(v)\subseteq \cl(\sigma)$, all values are $k$, all transition maps are identities, and the index is connected (it has a least object); the colimit is $k$.
\item If $\st_a(v)\not\subseteq \cl(\sigma)$, then either $v\not\le\sigma$, whence $\st_a(v)\cap \cl(\sigma)=\varnothing$ (because $u\in \cl(\sigma)$ and $v\le u$ would give $v\in \cl(\sigma)$) and the diagram is identically $0$; or $v\le\sigma$ and some $u_0\in\st_a(v)$ has $u_0\not\le\sigma$, so the arrow $(v,v)\le(v,u_0)$ sends $k_{(v,v)}\mapsto0$, and every surviving generator $k_{(v,u)}$ ($u\in\st_a(v)\cap \cl(\sigma)$) equals $k_{(v,v)}$ through $(v,v)\le(v,u)$ and vanishes with it. Either way the colimit is $0$.
\end{itemize}
Collecting the two cases gives, for every minimal $v$, the dichotomy
\[
(C_aI_\sigma)(v)=
\begin{cases}
k, & \st_a(v)\subseteq \cl(\sigma),\\[2pt]
0, & \text{otherwise,}
\end{cases}
\tag{$\ast$}
\]
which drives both directions of the criterion.

\emph{The nonvanishing criterion.} ``$\Leftarrow$'' is immediate: if a minimal $v\le\sigma$ satisfies $\st_a(v)\subseteq \cl(\sigma)$, then $(\ast)$ gives $(C_aI_\sigma)(v)=k\ne0$, so $C_aI_\sigma\ne0$.

For ``$\Rightarrow$'' we prove the contrapositive: assume every minimal $v\le\sigma$ has $\st_a(v)\not\subseteq \cl(\sigma)$, fix an arbitrary $x\in\Pos$, and show $(C_aI_\sigma)(x)=0$. Let $(s,u)\in J_a(x)$ be a generator, so $u\le\sigma$. Pick a minimal $v\le s$, which exists because $\Pos$ is finite; then $v\le s\le u\le\sigma$, so $v\le\sigma$, and by hypothesis some $u_0\in\st_a(v)$ has $u_0\not\le\sigma$. As $u_0\in\st_a(v)$, i.e.\ $(v,u_0)\in\Delta_a^{\Pos}$, and $v\le s\le x$, the pair $(v,u_0)$ lies in $J_a(x)$ with value $0$, and $(v,v)\le(v,u_0)$ forces $k_{(v,v)}\mapsto0$. Finally $(v,v)\le(s,u)$ (since $v\le s\le u$) with $u\le\sigma$, so $k_{(s,u)}=k_{(v,v)}=0$ in the colimit. Every generator vanishes, so $(C_aI_\sigma)(x)=0$; as $x$ was arbitrary, $C_aI_\sigma=0$.
\end{proof}

\begin{corollary}[Skyscraper stalks at a minimal element]\label{cor:sky-stalk}
Let $\Pos$ be a finite poset, let $a\ge1$, let $v$ be minimal in $\Pos$, and let $\sigma\in \Pos$. Then
\[
(C_a\sky_\sigma)(v)=
\begin{cases}
k, & (v,\sigma)\in\Delta_a^{\Pos} \ \text{ and }\ \sigma \text{ is maximal in } \st_a(v),\\[2pt]
0, & \text{otherwise.}
\end{cases}
\]
\end{corollary}

\begin{proof}
Because $v$ is minimal, every $(s,t)\in J_a(v)$ has $s\le v$ forcing $s=v$, so $J_a(v)=\{(v,u):u\in\st_a(v)\}$ is isomorphic to the poset $\st_a(v)$, which has least element $v$. By Lemma~\ref{lem:support-presentation} the stalk is $\colim_{u\in\st_a(v)}\sky_\sigma(u)$. This diagram equals $k$ at $u=\sigma$ if and only if $\sigma\in\st_a(v)$, i.e.\ when $(v,\sigma)\in\Delta_a^{\Pos}$, and $0$ at every other $u$. If $(v,\sigma)\notin\Delta_a^{\Pos}$ the diagram vanishes, and so does the colimit. If $(v,\sigma)\in\Delta_a^{\Pos}$, the single generator at $u=\sigma$ survives the colimit iff no arrow carries it to a zero object, that is iff there is no $u'\in\st_a(v)$ with $\sigma<u'$; equivalently $\sigma$ is maximal in $\st_a(v)$. When $\sigma$ is not maximal, the restriction $\sky_\sigma(\sigma)\to \sky_\sigma(u')=0$ along some $\sigma<u'$ kills the class.
\end{proof}

\begin{example}[The minimality hypothesis is essential]\label{ex:lateral}
At a non-minimal $v$ the dichotomy fails: some $s<v$ can feed into the window an element that is \emph{not} above $v$, so the comma index $J_a(v)$ sees elements laterally and not only through $\st_a(v)$. On the poset $\{0<p,\ 0<q\}$ one has $(C_1\sky_p)(q)=k$ although $q\not\le p$, and on the bigon (two edges $s,v$ on the same pair of vertices) $(C_1\sky_s)(v)=k^{2}$. In general $(C_a\sky_\sigma)(v)\cong k^{\,c}$, where $c$ is the number of connected components of $\{s\le v:(s,\sigma)\in\Delta_a^{\Pos}\}$ that admit no window-coface strictly above $\sigma$ (the degree-zero component count of Corollary \ref{cor:h0-count}). At a minimal $v$ one has $c\in\{0,1\}$ and the count reduces to the stated condition.
\end{example}

The obstruction is not confined to degree zero: the two computations in the proof of Theorem \ref{thm:noshadow} give $H^i(\Pos;C_ak_{\Pos})\cong H^i(K;k)$ and $H^i(\Pos;C_a\sky_v)=0$ in every degree $i$. Nor is it tied to the value $[k]$: by Proposition \ref{prop:edge} any $D$ as in Theorem \ref{thm:noshadow} would have to send $[0]$ to both $k^2$ and $0$.

\section{The derived flow and the meet obstruction}\label{sec:derived}

The object of study is the finality defect of the meet functor $\Phi$ (Remark \ref{rem:weight-independent}). All quasi-isomorphisms are in the sense of Remark \ref{rem:homotopical}.

\subsection{Stalkwise models, and the strict flow on chains}

\begin{definition}[Flows and their strength]\label{def:flow}
Let $(M,\le,+,0)$ be a commutative monoid carrying a partial order for which $+$ is monotone in each variable (here $M=\mathbb{Z}_{\ge0}$, or $\mathbb{Z},\mathbb{R}$ with their usual orders), read as a monoidal category with one arrow $a\to b$ when $a\le b$. A \emph{flow indexed by $M$} on a category $\mathcal C$ is a lax monoidal functor $X_\bullet\colon(M,\le,+,0)\to(\operatorname{End}(\mathcal C),\circ,\mathrm{Id})$: a family $\{X_a\}$ with $X_0=\mathrm{Id}$, transitions $X_a\Rightarrow X_b$ composing along $\le$, and a comparison $\mu_{a,b}$ between $X_a\circ X_b$ and $X_{a+b}$, natural in $a$ and $b$ as well as in the object and subject to the unit and associativity coherence. Forgetting the order leaves a coherent action of the bare monoid, which is the notion of de Silva--Munch--Stefanou \cite[Def.~2.3, Def.~2.4]{dSMS}, there a coherent action of $([0,\infty),+)$ carrying a natural transformation $u\colon\mathrm{Id}\to X_0$ that is not required to be invertible; it is normalized here to the strictly unital variant, in which $X_0=\mathrm{Id}$.
The parameter is $(\mathbb Z_{\ge0},\le,+)$ throughout, the order being what the derived transition maps of Proposition \ref{prop:derived-exist}\textup{(iv)} and every interleaving-type distance use. The \emph{strength} of a flow is the nature of $\mu_{a,b}$: \emph{strict} (an equality $X_aX_b=X_{a+b}$), \emph{strong} (a coherent natural isomorphism), \emph{lax} (a coherent natural transformation $X_aX_b\Rightarrow X_{a+b}$, not required invertible), or \emph{oplax} (one in the reverse direction). We are concerned with the family $\{\LC_a\}$ on $\Db(\Shv(\Pos;k))$: strict when $\Pos$ is a chain, but obstructed already by missing meets and by non-final diamond fibres.
\end{definition}

\begin{definition}[Simplicial replacement and finite bar model]\label{def:bar-hocolim}
Let $I$ be a finite poset and let $D\colon I\to\Chb(\mathrm{Vec}_k)$ be a covariant diagram. The \emph{simplicial replacement} $\srep(D)$ is the simplicial object of $\Chb(\mathrm{Vec}_k)$ with
\[
   \srep(D)_n=
   \bigoplus_{i_0\le\cdots\le i_n}D(i_0).
\]
The face $d_0$ applies the structure map $D(i_0)\to D(i_1)$ and then deletes $i_0$; the other faces delete the indicated index and act by the identity on the coefficient. Degeneracies repeat an index. The \emph{bar complex}
\[
   \operatorname{Bar}(I,D)
\]
is the normalized total complex of this simplicial replacement (Definition \ref{def:normalized-complex}), equivalently the one-sided complex $B(k,I,D)$, the case of a constant first variable in the two-sided bar complex of Definition \ref{def:two-sided-bar}. In bar degree $n$ it is
\[
   \operatorname{Bar}_n(I,D)=
   \bigoplus_{i_0<\cdots<i_n}D(i_0),
\]
with the faces of the simplicial replacement above. Write
\[
   \partial_{\mathrm{bar}}=\sum_{j=0}^n(-1)^jd_j
\]
for the resulting alternating bar differential. If $D$ takes values in $\mathrm{Vec}_k$, $\operatorname{Bar}(I,D)$ is the complex just displayed with differential $\partial_{\mathrm{bar}}$. For chain-complex-valued $D$ there are two differential directions, and $\operatorname{Bar}(I,D)$ is the total complex of the double complex
\[
   \operatorname{Bar}_{n,m}(I,D)=\bigoplus_{i_0<\cdots<i_n}D(i_0)_m ,
\]
placed in total degree $n+m$: on a homogeneous $x\in D(i_0)_m$ indexed by $i_0<\cdots<i_n$,
\[
   d_{\mathrm{tot}}(x)=d_{D(i_0)}(x)+(-1)^m\,\partial_{\mathrm{bar}}(x).
\]
The sign is forced. Every face map is built from chain maps, so $d_D$ and $\partial_{\mathrm{bar}}$ commute; the Koszul factor $(-1)^m$, inserted because $\partial_{\mathrm{bar}}$ has degree $-1$ in the bar direction, converts that into the anticommutation required for $d_{\mathrm{tot}}^2=0$. This is the case $R=k$ of Definition \ref{def:two-sided-bar}.

A monotone map $\Psi\colon I'\to I$, together with a map of diagrams $\alpha\colon D'\to\Psi^\ast D$, gives a map of simplicial replacements $\srep(D')\to\srep(D)$: the summand at $i_0\le\cdots\le i_n$ goes by $\alpha_{i_0}\colon D'(i_0)\to D(\Psi(i_0))$ to the summand at $\Psi(i_0)\le\cdots\le\Psi(i_n)$. It commutes with every face, the $d_0$ case being that $\alpha$ is a map of diagrams, and with every degeneracy, so it is a chain map of Moore complexes and descends to the normalized quotients (Definition \ref{def:normalized-complex}). The \emph{induced bar map} is that quotient map,
\[
   \Psi_\#\colon\operatorname{Bar}(I',D')\longrightarrow\operatorname{Bar}(I,D),
\]
written after the monotone map, with $\alpha$ the identity of $\Psi^\ast D$ unless another is named. It carries the summand at $i_0<\cdots<i_n$ to the summand at $\Psi(i_0)<\cdots<\Psi(i_n)$, and to zero when that image is not strict, the second clause being the passage to the normalized quotient.

The passage from the simplicial object $\srep(D)$ to the complex $\operatorname{Bar}(I,D)$ is an \emph{algebraic realization}.
It is the algebraic analogue of Dugger's realization formula $\hocolim_ID=|\srep(D)|=\int^{[n]\in\mathbf\Delta}\srep(D)_n\times\Delta^n$ \cite[\S4, Def.~4.5]{DuggerHocolim}: with $\mathcal N_\ast(\Delta^n;k)$ the normalized simplicial chain complex of the standard $n$-simplex,
\[
   \operatorname{Bar}(I,D)
   \;\cong\;
   \int^{[n]\in\mathbf\Delta}\mathcal N_\ast(\Delta^n;k)\otimes_k\srep(D)_n ,
\]
the coend being formed in $\Ch(\mathrm{Vec}_k)$ and $\otimes_k$ the tensor product of complexes, whose sign rule is the $(-1)^m$ above. This is Lemma \ref{lem:algebraic-realization}.

In the present finite $k$-linear setting $\operatorname{Bar}(I,D)$ computes $\hocolim_I D$, i.e.\ the left derived colimit $\mathbb L\!\colim_I D$:
\begin{equation}\label{eq:colim-square}
\begin{tikzcd}[column sep=huge,row sep=large]
\Chb(\Shv(I;k))
  \arrow[r,"\colim_I"]
  \arrow[d,"\gamma"']
&
\Chb(\mathrm{Vec}_k)
  \arrow[d,"\gamma"]
\\
\Db(\Shv(I;k))
  \arrow[r,"\mathbb L\!\colim_I\cong \operatorname{Bar}(I{,}-)"']
&
\Db(\mathrm{Vec}_k)
  \arrow[Rightarrow,from=2-1,to=1-2,shorten <=12pt,shorten >=12pt,"\theta"]
\end{tikzcd}
\end{equation}
Here $\operatorname{Bar}(I,-)$ represents the lower row, and $\theta$, the comparison in the sense of Remark \ref{rem:deformation}, is induced by the augmentation of the bar complex. The Bousfield--Kan formula and its identification with the derived colimit \cite[Thm.~5.1.1, Cor.~5.1.3]{Riehl}, \cite[\S\S~4, 8, 10]{DuggerHocolim} are stated for spaces or simplicial model categories; for diagrams of chain complexes; specifically over a finite category $I$, and quasi-isomorphisms it is \cite[Thm.~3.4--(1)]{Arakawa}.

Over the incidence algebra $kI$ (Definition \ref{def:incidence-algebra}), $\operatorname{Bar}(I,D)$ computes $D\otimes^{\mathbb L}_{kI}k=\mathbb L\!\colim_ID$ by resolving the trivial module rather than $D$; this is Lemma \ref{lem:bar-derived-colim}, proved from Lemma \ref{lem:incidence-trivial}\textup{(ii)--(iii)} and the deformation of Remark \ref{rem:deformation}.
\end{definition}

\begin{proposition}[Stalkwise bar model]\label{prop:bar}
Let $F\in\Shv(\Pos;k)$, $\sigma\in \Pos$, $a\ge0$. Then
\[
   (\LC_aF)(\sigma)\;\simeq\;\hocolim_{J_a(\sigma)}\ q_2^\ast F ,
\]
computed by $\operatorname{Bar}(J_a(\sigma),q_2^\ast F)$, namely
\[
\operatorname{Bar}_n\bigl(J_a(\sigma),q_2^\ast F\bigr)=\bigoplus_{o_0<\cdots<o_n}F\bigl(t(o_0)\bigr),
\qquad o_i=(s_i,t_i)\in J_a(\sigma),
\]
where $d_0$ deletes $o_0$ and applies $F(t_0\le t_1)$, and $d_i$ deletes $o_i$ (identity on the coefficient) for $i\ge1$. The complex has $H_0\bigl(\operatorname{Bar}(J_a(\sigma),q_2^\ast F)\bigr)=(C_aF)(\sigma)$, and $H_i\LC_aF=0$ above the longest chain in $\Delta_a^{\Pos}$.

The statement holds as well when $F$ is a bounded complex of sheaves, that is $F\in\Chb(\Shv(\Pos;k))$; then $\operatorname{Bar}(J_a(\sigma),q_2^\ast F)$ is read as the totalization of Definition \ref{def:bar-hocolim}: the displayed bar object acquires the internal differential of $F$ with the Koszul sign, and the vanishing bound applies to the bar direction.
\end{proposition}

\begin{proof}
Write $G=(q_2^\Delta)^\ast F$, so that $\LC_aF=\mathbb L\Lan_{q_1^\Delta}G$ (Definition \ref{def:LCa}). The pointwise formula for an ordinary left Kan extension \cite[Thm.~6.2.1]{RiehlCTIC} evaluates $\Lan_{q_1^\Delta}$ at $\sigma$ as the colimit over the comma poset $q_1^\Delta\downarrow\sigma=J_a(\sigma)$, and the stalk functor at $\sigma$ is exact, so
\[
   \bigl(H_i\,\mathbb L\Lan_{q_1^\Delta}G\bigr)(\sigma)
   =H_i\Bigl(\colim_{J_a(\sigma)}P_\bullet|_{J_a(\sigma)}\Bigr)
\]
for a projective resolution $P_\bullet\to G$ in $\Shv(\Delta_a^{\Pos};k)$. Resolve by the epimorphisms of Lemma \ref{lem:enough-projectives}, so that each $P_n$ is a direct sum of objects $M_{\st(o)}$ with $o\in\Delta_a^{\Pos}$ and $M$ a $k$-vector space. Such an object restricts on $J_a(\sigma)$ to $M$ on $\{o'\in J_a(\sigma):o'\ge o\}$ and to $0$ elsewhere, that set being empty when $o\notin J_a(\sigma)$, since $(s,t)\ge o=(s_0,t_0)$ and $s\le\sigma$ force $s_0\le\sigma$. It is up-closed in $J_a(\sigma)$ with least element $o$, so Lemma \ref{lem:support-restriction}\textup{(i)} presents the bar complex of the restriction as $M\otimes_kC_\ast(\Delta(\{o'\in J_a(\sigma):o'\ge o\});k)$, and that order complex is contractible (Lemma \ref{lem:cone-point}); by Lemma \ref{lem:bar-derived-colim}\textup{(ii)} the restriction is therefore acyclic for $\colim_{J_a(\sigma)}$. A resolution by such objects computes the derived colimit, which is $\hocolim_{J_a(\sigma)}G$. Definition \ref{def:bar-hocolim} presents it by the displayed finite bar complex, the chain-level identification being Lemma \ref{lem:bar-derived-colim} and \cite[Thm.~3.4--(1)]{Arakawa}; \cite[Cor.~5.1.3]{Riehl} states the bar model for a simplicial model category, which is not assumed here (Remark \ref{rem:homotopical}). Its $H_0$ is the cokernel of $d_0-d_1$, which is the colimit presentation of Lemma \ref{lem:colim-coeq} at $I=J_a(\sigma)$ and so $(C_aF)(\sigma)$; vanishing above the nerve dimension of $J_a(\sigma)$ is immediate from the bar grading.

For a bounded complex $F$ of sheaves nothing changes in the argument, the construction acting degreewise; the stalk at $\sigma$ is again the homotopy colimit over $J_a(\sigma)$, now of the $\Chb(\mathrm{Vec}_k)$-valued diagram $q_2^\ast F$, which Definition \ref{def:bar-hocolim} computes by the totalization with differential $d_{\mathrm{tot}}$.
\end{proof}

\begin{definition}[Grothendieck construction of a $\mathbf{Cat}$-valued diagram]\label{def:grothendieck}
Let $I$ be a small category and $X\colon I\to\mathbf{Cat}$ a functor. Following \cite[Def.~1.1]{Thomason}, the \emph{Grothendieck construction} $\int_IX$ is the category whose objects are pairs $(i,x)$ with $i\in I$ and $x\in X(i)$, and whose morphisms $(i,x)\to(i',x')$ are pairs $(\alpha,f)$ with $\alpha\colon i\to i'$ in $I$ and $f\colon X(\alpha)(x)\to x'$ in $X(i')$; composition is $(\alpha',f')\circ(\alpha,f)=(\alpha'\alpha,\ f'\circ X(\alpha')(f))$. The assignment $(i,x)\mapsto i$ is a functor $\mathrm{pr}\colon\int_IX\to I$, the \emph{Grothendieck projection}, with fibre $X(i)$ over $i$.

We use only the case in which $I$ and $L$ are finite posets, every $X(i)$ is a full subposet of $L$, and every $X(i\le i')$ is the inclusion $X(i)\subseteq X(i')$. Then $\int_IX$ is a finite poset, there is at most one morphism between any two objects, and the order is componentwise,
\[
   (i,x)\le(i',x')
   \iff
   i\le i'\ \text{ and }\ x\le x' ,
\]
the second relation read in $L$, equivalently in $X(i')$, which is full and contains both $x$ and $x'$. Alongside $\mathrm{pr}$ it carries the \emph{second projection}
\[
   \pi\colon\textstyle\int_IX\longrightarrow L,\qquad\pi(i,x)=x ,
\]
monotone by that display.
\end{definition}

\begin{proposition}[Grothendieck single-hocolim model of the iterate]\label{prop:groth}
For $a,b\ge0$ the assignment $(s,t)\mapsto J_b(t)\subseteq \Delta_b^{\Pos}$ is a covariant functor $D_{a,b}(\sigma):J_a(\sigma)\to\mathbf{Cat}$. The Grothendieck construction (Definition \ref{def:grothendieck}) of the diagram $D_{a,b}(\sigma)$ is the finite poset
\[
   \Ka_{a,b}(\sigma)=\textstyle\int_{(s,t)\in J_a(\sigma)}J_b(t)
   =\{(s,t,r,u):s\le\sigma,\ (s,t)\in\Delta_a^{\Pos},\ r\le t,\ (r,u)\in\Delta_b^{\Pos}\},
\]
ordered componentwise, and carrying the Grothendieck projection $\mathrm{pr}(s,t,r,u)=(s,t)$ of Definition \ref{def:grothendieck}, together with the coefficient projection $p(s,t,r,u)=u$ of Proposition \ref{prop:final} that fit in the diagram:
\[
\begin{tikzcd}[column sep=large,row sep=large]
J_b(t)\arrow[r,hook]\arrow[d] & \Ka_{a,b}(\sigma)\arrow[d,"\mathrm{pr}"]\arrow[r,"p"] & \Pos\\
\{(s,t)\}\arrow[r,hook] & J_a(\sigma) &.
\end{tikzcd}
\]
The left-hand square exhibits the inner window-comma poset $J_b(t)$ as the fibre of $\mathrm{pr}$ over $(s,t)$, and the coefficient diagram of the model below is $p^\ast F$. Then
\[
   (\LC_a\LC_bF)(\sigma)\;\simeq\;\hocolim_{\Ka_{a,b}(\sigma)}\ F(u),
\]
a single homotopy colimit; in particular $H_0$ recovers the underived iterate $(C_aC_bF)(\sigma)$.
\end{proposition}

\begin{proof}
Write $D:=D_{a,b}(\sigma)\colon J_a(\sigma)\to\mathbf{Cat}$. A morphism $(\alpha,h)\colon(s,t,r,u)\to(s',t',r',u')$ of $\Ka_{a,b}(\sigma)$ consists of a relation $\alpha\colon(s,t)\le(s',t')$ in $J_a(\sigma)$ together with $D(\alpha)(r,u)\le(r',u')$ in $J_b(t')$.
Because each $J_b(t)$ is a full subposet of $\Delta_b^{\Pos}$ (Definition \ref{def:window-index}) and the transition maps $D(\alpha)$ are the inclusions $J_b(t)\hookrightarrow J_b(t')$ for $t\le t'$ (if $r\le t\le t'$ then $(r,u)\in J_b(t')$), the assignment is functorial; by Definition \ref{def:grothendieck} the Grothendieck construction is the displayed poset with the componentwise order. Applying Proposition \ref{prop:bar} twice gives
\[
   (\LC_a\LC_bF)(\sigma)\simeq
   \hocolim_{(s,t)\in J_a(\sigma)}\hocolim_{(r,u)\in J_b(t)}F(u),
\]
the inner homotopy colimit being $\operatorname{Bar}(J_b(t),q_2^\ast F)$, functorial in $t$ through the inclusions $J_b(t)\subseteq J_b(t')$ and so a chain-level model of $\LC_bF$ over $J_a(\sigma)$, which the outer bar complex may use by Lemma \ref{lem:bar-invariance}. Lemma \ref{lem:groth-hocolim}, applied with $I=J_a(\sigma)$, ambient poset $L=\Delta_b^{\Pos}$, $X=D$ and $G=q_2^\ast F$, so that $E(s,t,r,u)=F(u)$, identifies the iterated homotopy colimit with the single one over $\Ka_{a,b}(\sigma)=\int_ID$; it is the $k$-linear coefficient counterpart of Thomason's theorem \cite{Thomason}, \cite[Cor.~3.3]{FM}. For the last assertion, only the Kan extension in $C_a=\Lan_{q_1^\Delta}\circ(q_2^\Delta)^\ast$ is derived (Definition \ref{def:LCa}), and a left Kan extension is right exact; so $H_0\LC_a=C_a$, and more generally $H_0(\LC_aG)\cong C_a(H_0G)$ for a complex $G$ of sheaves concentrated in degrees $\ge0$.
Taking $G=\LC_bF$, whose $H_0$ is $C_bF$, gives the $H_0$ statement.
\end{proof}

\begin{proposition}[Skyscraper stalks of the derived convolution]\label{prop:derived-sky}
Let $\tau,\sigma\in \Pos$ and $a\ge0$. Set
\[
A_{\sigma,\tau}=\{(s,t)\in J_a(\sigma):t\ge\tau\},
\qquad
B_{\sigma,\tau}=\{(s,t)\in J_a(\sigma):t>\tau\},
\]
both up-closed subposets of $J_a(\sigma)$. Then for every $i\ge0$
\[
H_i\bigl((\LC_a\sky_\tau)(\sigma)\bigr)\;\cong\;H_i\bigl(\Delta(A_{\sigma,\tau}),\Delta(B_{\sigma,\tau});k\bigr),
\]
the relative simplicial homology of the pair of order complexes.
\end{proposition}

\begin{proof}
Write $J:=J_a(\sigma)$ and $D$ for the diagram $(s,t)\mapsto\sky_\tau(t)$ on $J$; it is $k$ on $A_{\sigma,\tau}\setminus B_{\sigma,\tau}=\{(s,\tau)\in J\}$ and $0$ elsewhere, with identity structure maps between nonzero values, so $D=k_{A_{\sigma,\tau}/B_{\sigma,\tau}}$ in the notation of Lemma \ref{lem:bar-pair}. The stalk is $\operatorname{Bar}(J,D)$ by Proposition \ref{prop:bar}, and Lemma \ref{lem:bar-pair}\textup{(i)} identifies that with $C_\ast(\Delta(A_{\sigma,\tau});k)/C_\ast(\Delta(B_{\sigma,\tau});k)$, whose homology is the stated relative homology.
\end{proof}

\begin{corollary}[Degree zero: the component count]\label{cor:h0-count}
$\dim_k(C_a\sky_\tau)(\sigma)$ equals the number of connected components of $\Delta(A_{\sigma,\tau})$ containing no vertex of $\Delta(B_{\sigma,\tau})$. At a vertex $\sigma=v$ this recovers Corollary \ref{cor:sky-stalk}, and at a general cell it is the component count of Example \ref{ex:lateral}.
\end{corollary}

\begin{proof}
$H_0$ of a pair is the cokernel of $H_0(\Delta(B_{\sigma,\tau}))\to H_0(\Delta(A_{\sigma,\tau}))$.
\end{proof}

\begin{example}[Two zeros of $C_1\sky_v$ told apart by $\LC_1$]\label{ex:overglue}
Both are computed by Proposition \ref{prop:derived-sky} at the stalk $\sigma=v$. In both the underived stalk vanishes because the colimit defining $C_1$ identifies and then cancels the surviving generators; whether that cancellation is free is visible only after deriving.
\begin{enumerate}
\item[\textup{(i)}] \emph{Branch.} On $\Lambda=\{v<e_1,\ v<e_2\}$ the window-comma poset at $v$, with the subposet $B_{v,v}$ boxed, is
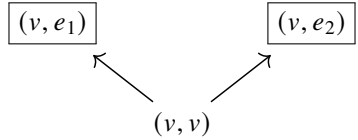
\[
\begin{tikzcd}[row sep=1.8em,column sep=1.2em]
\boxed{(v,e_1)} & & \boxed{(v,e_2)}\\
& (v,v)\arrow[ul]\arrow[ur] &
\end{tikzcd}
\]
Thus $A_{v,v}=J_1(v)$ has least element $(v,v)$, so $\Delta(A_{v,v})$ is contractible, while $\Delta(B_{v,v})=\{(v,e_1),(v,e_2)\}$ is two points; the pair sequence gives
\[
H_0\bigl((\LC_1\sky_v)(v)\bigr)=0,
\qquad
H_1\bigl((\LC_1\sky_v)(v)\bigr)\cong k .
\]
\item[\textup{(ii)}] \emph{Chain.} On $\{v<e\}$, at every cell the pair $(\Delta(A),\Delta(B))$ consists of a contractible complex and a nonempty contractible subcomplex, so $H_\ast(\Delta(A),\Delta(B))=0$ and $\LC_1\sky_v=0$: the vanishing is genuine and deriving does not recover it.
\end{enumerate}
Thus the passage $C_a\rightsquigarrow\LC_a$ cures the branching collapse (present already in dimension one) but not the composition failure detected by Proposition \ref{prop:diamond-nogo} at diamond intervals.
\end{example}

\begin{proposition}[Exactness threshold]\label{prop:inexact}
Let $\Pos$ be a finite poset and $a\ge1$. The following are equivalent:
\begin{enumerate}
\item[\textup{(i)}] $H_i\bigl(\Delta(A_{\sigma,\tau}),\Delta(B_{\sigma,\tau});k\bigr)=0$ for all $i\ge1$ and all $\sigma,\tau\in \Pos$;
\item[\textup{(ii)}] $H_i\LC_aF=0$ for all $i\ge1$ and all $F\in\Shv(\Pos;k)$.
\end{enumerate}
Either implies that $C_a$ is exact. For $v$ minimal in $\Pos$,
\[
H_i\LC_a\sky_v(v)\;\cong\;\widetilde H_{i-1}\bigl(\Delta(\st_a(v)\setminus\{v\});k\bigr),\qquad i\ge1 ,
\]
which for $a=1$ reduces to
\[
H_1\LC_1\sky_v(v)\;\cong\;k^{\,c(v)-1}\qquad(c(v)\ge1),
\]
$c(v)$ being the number of elements covering $v$; on a face poset the minimal elements are the vertices and $c(v)=\deg v$, the number of $1$-cells containing $v$. Inexactness of $C_1$ itself is visible without deriving: $C_1$ fails to be exact at any minimal element with two distinct covers; on a face poset, at any vertex incident to at least two $1$-cells, already in dimension one and with no diamond interval involved. Conversely $C_a$ is exact on chains.
\end{proposition}

\begin{proof}
That (ii) implies (i) is Proposition \ref{prop:derived-sky} at $F=\sky_\tau$. For the converse, vanishing on skyscrapers suffices. Indeed, enumerate $\Pos$ as $x_1,\dots,x_n$ so that $x_m\le x_j$ in $\Pos$ implies $m\le j$; such an enumeration exists, obtained by repeatedly deleting an element minimal among those not yet listed. Set $U_p=\{x_{p+1},\dots,x_n\}$, up-closed in $\Pos$ because $x_m\in U_p$ and $x_m\le x_j$ give $j\ge m>p$. Write $F_U$ for the subsheaf with $F_U(x)=F(x)$ for $x\in U$ and $0$ otherwise; it is a subfunctor because $U$ is up-closed, $x\in U$ and $x\le y$ forcing $y\in U$. This exhibits every $F$ as a finite filtration
\[
0=F_{U_n}\subseteq F_{U_{n-1}}\subseteq\cdots\subseteq F_{U_0}=F,
\qquad
F_{U_{p-1}}/F_{U_p}\cong\sky_{x_p}\otimes_kF(x_p).
\]
Because $\operatorname{Bar}(J_a(\sigma),q_2^\ast F)$ is degreewise a direct sum of stalk values of $F$, and homology commutes with direct sums of complexes (coproducts in $\mathrm{Vec}_k$ being exact and an exact functor commuting with homology), Proposition \ref{prop:bar} makes each $H_i\LC_a$ preserve arbitrary direct sums of sheaves; with Proposition \ref{prop:derived-sky} and hypothesis (i) this gives $H_i\LC_a\bigl(\sky_{x_p}\otimes_kF(x_p)\bigr)=0$ for $i\ge1$. Since $\LC_a$ is triangulated (Remark \ref{rem:homotopical}), each short exact sequence $0\to F_{U_p}\to F_{U_{p-1}}\to\sky_{x_p}\otimes_kF(x_p)\to0$ gives a long exact sequence in the homology of $\LC_a$, whose outer terms vanish in degrees $\ge1$, the first by descending induction from $F_{U_n}=0$; so the middle one does too, and (ii) follows. Degree one alone would not do: the dimension-shifting identity that would propagate it is available for $\Lan_{q_1^\Delta}$ and not across the restriction (Remark \ref{rem:dimension-shifting}).

Given (ii), let $0\to F'\to F\to F''\to0$ be exact; $C_a$ is right exact (Corollary \ref{cor:exactness}), and the long exact sequence of the associated distinguished triangle reads
\[
   0=H_1\LC_aF''\to H_0\LC_aF'\to H_0\LC_aF\to H_0\LC_aF''\to0 ,
\]
and $H_0\LC_a=C_a$ (Proposition \ref{prop:bar}), so $C_a$ is exact.

For $\sigma=\tau=v$ with $v$ minimal, the condition $s\le\sigma$ forces $s=v$, so $A_{v,v}=J_a(v)\cong\st_a(v)$ has least element $(v,v)$, and its order complex is contractible, while $B_{v,v}\cong\st_a(v)\setminus\{v\}$. The long exact sequence of the pair gives the formula at a minimal element. For $a=1$, $\st_1(v)\setminus\{v\}$ is the set of elements covering $v$, an antichain (no strict comparable pair $a<b$), so its reduced $H_0$ is $k^{\,c(v)-1}$; on a face poset the covers of a vertex are the $1$-cells containing it.

Inexactness of $C_1$ at such a $v$ needs no derived functor. The subsheaf $k_{\Pos_{>v}}\subseteq k_{\st(v)}$ has cokernel $\sky_v$, and over $J_1(v)\cong\st_1(v)$ the colimits of the three are $k^{\,c(v)}$, $k$ and $0$: the first because the value at $v$ is zero and the covers form an antichain, the second because $\st_1(v)$ has least element $v$. So $C_1$ carries that monomorphism to a map $k^{\,c(v)}\to k$, not injective once $c(v)\ge2$. The chain statement is Proposition \ref{prop:chain}(i), where every $J_a(\sigma)$ has a terminal object.
\end{proof}

\begin{proposition}[Chains: the strict flow, on both sides]\label{prop:chain}
Let $\Pos$ be a chain $x_0<x_1<\cdots<x_n$, which for $n\ge1$ is not a face poset (Appendix \ref{ssec:face-posets}); for $\sigma=x_i$ write $\sigma\up a=x_{\min(i+a,n)}$. Then for every $F$:
\begin{enumerate}
\item[\textup{(i)}] $J_a(\sigma)$ has terminal object $(\sigma,\sigma\up a)$, so $\LC_aF$ is the sheaf $\sigma\mapsto F(\sigma\up a)$ concentrated in degree $0$: the upward height shift.
\item[\textup{(ii)}] $(\sigma\up a)\up b=\sigma\up(a+b)$, so $\LC_a\LC_b=\LC_{a+b}$ and $\{\LC_a\}$ is a strict flow. Dually $\RC^a$ is the downward shift and $\{\RC^a\}$ a strict flow.
\end{enumerate}
\end{proposition}

\begin{proof}
(i) On a chain $s\le\sigma$ and $\len(s,t)\le a$ force $t\le\sigma\up a$, and $(\sigma,\sigma\up a)\in J_a(\sigma)$ dominates every object componentwise and is terminal; a homotopy colimit over an index poset with terminal object is evaluation there, in degree $0$ (Proposition \ref{prop:bar} and Lemma \ref{lem:terminal-bar}), giving $F(\sigma\up a)$, with transitions those of $F$ along $\sigma\up a\le\sigma'\up a$. (ii) The identity $\min(\min(i+a,n)+b,n)=\min(i+a+b,n)$ gives $(\sigma\up a)\up b=\sigma\up(a+b)$; by (i), $\LC_a\LC_bF$ is the sheaf $\sigma\mapsto F(\sigma\up(a+b))=\LC_{a+b}F$. The dual identity $\max(\max(i-a,0)-b,0)=\max(i-a-b,0)$ gives the statement for $C^a$.
\end{proof}

\subsection{The meet functor, finality, and the diamond obstruction}

The single window $\LC_{a+b}F$ is modelled by $\hocolim_{J_{a+b}(\sigma)}q_2^\ast F$, and the iterate by Proposition \ref{prop:groth}. The meet gives an assignment between the two index posets, and where that assignment is total the resulting functor produces a comparison map.

\begin{definition}[The meet assignment and the meet functor]\label{def:phi}
The \emph{meet assignment} at $(\sigma;a,b)$ is the partially defined assignment
\[
   \Phi=\Phi_{\sigma,a,b}\colon(s,t,r,u)\longmapsto(s\meet r,\,u)
\]
on $\Ka_{a,b}(\sigma)$, defined at those quadruples for which the meet $s\meet r$ exists in $\Pos$. The poset is added to the subscript $\Phi_{\Pos,\sigma,a,b}$ where it varies. It is \emph{total} if for every $(s,t,r,u)\in\Ka_{a,b}(\sigma)$ the meet $s\meet r$ exists in $\Pos$ and $\len(s\meet r,u)\le a+b$. The length bound is an independent hypothesis, following neither from additivity of $\len$ nor from the existence of all meets; Example \ref{ex:essential-nonsingleton} is a lattice on which it fails.
A total meet assignment gives rise to a monotone, $u$-preserving functor
\[
   \Phi\colon\Ka_{a,b}(\sigma)\longrightarrow J_{a+b}(\sigma),
\]
called the \emph{meet functor} at $(\sigma;a,b)$.
\end{definition}

\begin{definition}[Semimodularity]\label{def:semimodular}
Let $\Pos$ be a finite graded lattice, so that any two elements $x,z$ have a greatest lower bound $x\meet z$ and a least upper bound $x\vee z$. Then $\Pos$ is
\begin{enumerate}
\item[\textup{(a)}] \emph{lower semimodular} if $\rk(x)+\rk(z)\le\rk(x\meet z)+\rk(x\vee z)$ for all $x,z\in\Pos$;
\item[\textup{(b)}] \emph{upper semimodular} if $\rk(x)+\rk(z)\ge\rk(x\meet z)+\rk(x\vee z)$ for all $x,z\in\Pos$;
\item[\textup{(c)}] \emph{modular} if both hold, that is if equality holds throughout.
\end{enumerate}
The word is always qualified below: neither \textup{(a)} nor \textup{(b)} implies the other, and only \textup{(a)} is used in Proposition \ref{prop:total-families}.
\end{definition}

\begin{proposition}[Two families on which the meet assignment is total]\label{prop:total-families}
Let $a,b\ge0$.
\begin{enumerate}
\item[\textup{(i)}] Let $\Sigma$ be a finite simplicial complex and let $\Pos$ be its set of faces, the empty face included, ordered by inclusion and graded by cardinality $\rk(x)=\lvert x\rvert$, the augmented face poset $\widehat\Pos(\Sigma)$ of Definition \ref{def:order-complex}. Then $\Phi$ is total at every $(\sigma;a,b)$.
\item[\textup{(ii)}] Let $\Pos$ be a lower semimodular lattice (Definition \ref{def:semimodular}). Then $\Phi$ is total at every $(\sigma;a,b)$.
\end{enumerate}
\end{proposition}

\begin{proof}
Let $\kappa=(s,t,r,u)\in\Ka_{a,b}(\sigma)$. In both cases the meet $s\meet r$ exists and satisfies $s\meet r\le s\le\sigma$, so only the bound $\len(s\meet r,u)\le a+b$ is at issue.

\textup{(i)} Meets are intersections of faces, hence faces. Since $s\subseteq t$ and $r\subseteq t$ one has $r\setminus s\subseteq t\setminus s$, so
\[
   \len(s\cap r,u)
   =\lvert u\rvert-\lvert s\cap r\rvert
   =\bigl(\lvert u\rvert-\lvert r\rvert\bigr)+\lvert r\setminus s\rvert
   \le\len(r,u)+\bigl(\lvert t\rvert-\lvert s\rvert\bigr)
   \le b+a .
\]

\textup{(ii)} From $s\le t$ and $r\le t$ we get $s\vee r\le t$, so lower semimodularity gives
\[
   \rk(r)-\rk(s\meet r)\;\le\;\rk(s\vee r)-\rk(s)\;\le\;\rk(t)-\rk(s)=\len(s,t)\le a ,
\]
and adding $\rk(u)-\rk(r)=\len(r,u)\le b$ gives $\len(s\meet r,u)\le a+b$.
\end{proof}

The argument for \textup{(i)} is valid in any finite simplicial poset (Definition \ref{def:simplicial-poset}), not only in an augmented face poset. If $s,r\le t$ then both lie in $\cl(t)$, which is Boolean; transporting along an isomorphism of $\cl(t)$ with a lattice of subsets makes $s\meet r$ an intersection and $\len$ a difference of cardinalities, so $\len(s\meet r,r)\le\len(s,t)\le a$. Together with $s\meet r\le r\le u$ and $\len(r,u)\le b$, additivity gives $\len(s\meet r,u)\le a+b$. An augmented face poset is a simplicial poset that is a meet-semilattice (Appendix \ref{ssec:face-posets}). So totality imposes less than the existence of all meets.

Upper semimodularity does not suffice in \textup{(ii)}. Let $\Pi_4$ be the lattice of set partitions of $\{1,2,3,4\}$ ordered by refinement, graded by $4$ minus the number of blocks, with least element $\hat0=1\vert2\vert3\vert4$ and greatest element $\hat1=1234$; its meet is the common refinement, and it is upper semimodular. Yet $\Phi$ is not total on it: take $\sigma=s=12\vert34$, $t=u=\hat1$ and $r=13\vert24$, so that $(s,t,r,u)\in\Ka_{1,1}(\sigma)$ while $s\meet r=\hat0$ and $\len(\hat0,\hat1)=3>2$.

Neither family exhausts totality, since it holds on every rooted forest, principal down-sets there being chains (Proposition \ref{prop:rooted}), and on the rank-three lattice of Example \ref{ex:essential-nonsingleton}, where $\len\le a+b$ throughout; being a lattice is not by itself enough, as that same example shows.
\begin{definition}[Homotopy and homological finality]\label{def:homotopy-final}
For a functor $U\colon A\to B$, write $b\downarrow U$ for the comma category whose objects are pairs $(a,b\to U(a))$, the \emph{comma fibre} of $U$ over $b$. The functor $U$ is \emph{homotopy final} if $N(b\downarrow U)$ is weakly contractible for every $b\in B$ \cite[Def.~8.5.1]{Riehl}. It is \emph{final} if each $b\downarrow U$ is nonempty and connected, which is what makes $\colim_AU^\ast G\to\colim_BG$ an isomorphism for every diagram $G$ on $B$ \cite[Thm.~IX.3.1]{ML}; homotopy finality implies finality \cite[Lem.~8.5.5]{Riehl}. Thus a total meet functor $\Phi$ is homotopy final if and only if, for every $j=(x,u)\in J_{a+b}(\sigma)$, the order complex of
\[
   j\downarrow\Phi
   =
   \{\kappa\in\Ka_{a,b}(\sigma):j\le \Phi(\kappa)\}
\]
is weakly contractible.

The functor $U$ is \emph{$k$-homologically final} if $N(b\downarrow U)$ is nonempty and
\[
   \widetilde H_q\bigl(N(b\downarrow U);k\bigr)=0
   \qquad(q\ge0)
\]
for every $b\in B$. Homotopy finality implies $k$-homological finality, but not conversely in general. The homological weakening is the one under which Quillen--McCord is classically available for finite posets \cite[Cor.~5.5]{Barmak}; the form needed here is Corollary \ref{cor:final-criterion}.

Because $\meet$ is a greatest lower bound, totality gives the comma fibre of a meet functor the closed form
\begin{equation}\label{eq:comma-fibre}
   j\downarrow\Phi
   \;=\;
   \{(s,t,r,v)\in\Ka_{a,b}(\sigma)\;:\;y\le s,\ y\le r,\ u\le v\}
   \qquad\bigl(j=(y,u)\bigr),
\end{equation}
since $y\le s\meet r$ if and only if $y\le s$ and $y\le r$.
\end{definition}

The converse does hold under a condition on the comma fibres:
if, fibrewise in $j$, the nerves $X=N(j\downarrow\Phi)$ are simply connected and $\widetilde H_\ast(X;k)=0 \Longrightarrow \widetilde H_\ast(X;\mathbb Z)=0$
(e.g. if their reduced integral homology is torsion-free, or if one imposes $\mathbb Z$-homological finality instead), then a first nonzero homotopy group would produce a first nonzero integral homology group, so the classical Hurewicz argument \cite[Thm.~4.32]{Hatcher} tightens $k$-homological finality back to homotopy finality.

\begin{proposition}[The meet comparison]\label{prop:final}
If $\Phi$ is total at $(\sigma;a,b)$ it induces a natural comparison map, the \emph{meet comparison}
$\gamma_\sigma\colon(\LC_a\LC_bF)(\sigma)\to(\LC_{a+b}F)(\sigma)$. If moreover $\Phi$ is $k$-homologically final, then $\gamma_\sigma$ is a quasi-isomorphism for every $F$.
\end{proposition}

\begin{proof}
Let $J=J_{a+b}(\sigma)$ and $K=\Ka_{a,b}(\sigma)$.  Write
\[
   q_2^J\colon J\to\Pos,\qquad q_2^J(x,u)=u,
   \qquad
   p\colon K\to\Pos,\qquad p(s,t,r,u)=u .
\]
The single-window stalk is evaluated on the diagram
\[
   D=(q_2^J)^\ast F\colon J\to\mathrm{Vec}_k,
   \qquad D(x,u)=F(u),
\]
whereas the Grothendieck model for the iterate is evaluated on
\[
   E=p^\ast F\colon K\to\mathrm{Vec}_k,
   \qquad E(s,t,r,u)=F(u).
\]
The sheaf $F$ enters the two stalk models through the following commuting diagram:
\[
\begin{tikzcd}[column sep=huge,row sep=large]
K \arrow[rr,"\Phi"] \arrow[dr,"p"'] \arrow[ddr,bend right=15,"E=p^\ast F"'] &&
J \arrow[dl,"q_2^J"] \arrow[ddl,bend left=15,"D=(q_2^J)^\ast F"] \\
& \Pos \arrow[d,"F"] & \\
& \mathrm{Vec}_k &
\end{tikzcd}
\]
Since $\Phi$ preserves the $u$-coordinate, $q_2^J\circ\Phi=p$, so
\[
   E=p^\ast F=(q_2^J\circ\Phi)^\ast F=\Phi^\ast(q_2^J)^\ast F=\Phi^\ast D .
\]
The functor $\Phi$ together with this identification gives the bar map
\[
   \operatorname{Bar}(K,E)\longrightarrow \operatorname{Bar}(J,D),
\]
denoted $\Phi_\#$. Under Propositions \ref{prop:bar}, \ref{prop:groth}, $\gamma_\sigma$ is the bottom horizontal map in the commutative square
\begin{equation}\label{eq:comparison-bar-square}
\begin{tikzcd}[column sep=huge,row sep=large]
\operatorname{Bar}(K,E) \arrow[r,"\Phi_\#"] \arrow[d,"\simeq"'] &
\operatorname{Bar}(J,D) \arrow[d,"\simeq"] \\
(\LC_a\LC_bF)(\sigma) \arrow[r,"\gamma_\sigma"'] &
(\LC_{a+b}F)(\sigma).
\end{tikzcd}
\end{equation}
Here $\Phi_\#$ carries the summand at $c_0<\cdots<c_n$ to the summand at $\Phi(c_0)<\cdots<\Phi(c_n)$, and to zero when that image is not strict; the identification its $d_0$ face uses is
\[
   E(c_0\le c_1)=D(\Phi(c_0)\le\Phi(c_1)),
\]
which is $E=\Phi^\ast D$.

Assume now that $\Phi$ is $k$-homologically final. Then $\Phi_\#$ is a quasi-isomorphism by Corollary \ref{cor:final-criterion}; so $\gamma_\sigma$ is an isomorphism in $\Db(\mathrm{Vec}_k)$. Homotopy finality is stronger by Definition \ref{def:homotopy-final}, and under it the statement is the $k$-linear counterpart of the Homotopy Finality Theorem \cite[Thm.~8.5.6]{Riehl}, whose poset-indexed form is \cite[Thm.~2.12]{FM}; we used only the homological condition.
\end{proof}

Note that the proposition is stated for a sheaf $F$, whereas Definition \ref{def:tame} asks for an isomorphism of endofunctors of $\Db(\Shv(\Pos;k))$. But testing on sheaves settles the derived statement. Deleting the top term of a bounded complex $X$ leaves a subcomplex $X'$ with that term $X''$, a sheaf placed in a single degree, as quotient, and with it the triangle $X'\to X\to X''\to X'[1]$. Both $\LC_a\LC_b$ and $\LC_{a+b}$ are exact functors of triangulated categories and $\gamma_{a,b}$ is a natural transformation between them, so applying them to that triangle gives
\[
\begin{tikzcd}[column sep=small,row sep=large]
\LC_a\LC_bX'\arrow[r]\arrow[d] &
\LC_a\LC_bX\arrow[r]\arrow[d] &
\LC_a\LC_bX''\arrow[r]\arrow[d] &
\LC_a\LC_bX'[1]\arrow[d]\\
\LC_{a+b}X'\arrow[r] &
\LC_{a+b}X\arrow[r] &
\LC_{a+b}X''\arrow[r] &
\LC_{a+b}X'[1]
\end{tikzcd}
\]
whose vertical maps are the components of $\gamma_{a,b}$, and in which the five lemma makes any one of the first three an isomorphism as soon as the other two are. The objects at which $\gamma_{a,b}$ is an isomorphism include the sheaves and are closed under shifts, so iterating the deletion reaches every object of $\Db(\Shv(\Pos;k))$.

Moreover, neither the bar machinery nor Lemma \ref{lem:bar-filtration}(iii) confines the coefficient to a sheaf, but in every instance below $D=(q_2^J)^\ast F$ is concentrated in internal degree zero and that K\"unneth term reduces.

The meet comparison is not the only natural transformation $\LC_a\LC_b\Rightarrow\LC_{a+b}$, nor is it claimed to dominate the others. Already on a chain, where the two functors agree (Proposition \ref{prop:chain}\textup{(ii)}), scaling by a unit of $k$ gives a second natural isomorphism, one failing the unit identities of Proposition \ref{prop:meet-coherence}\textup{(ii)}. So a $\gamma_\sigma$ that is not a quasi-isomorphism refutes one comparison, and the law $\LC_a\LC_b\simeq\LC_{a+b}$ is refuted only where the two functors are exhibited with different values, as in Propositions \ref{prop:sat-nogo} and \ref{prop:diamond-nogo} and Example \ref{ex:path-zigzag}.

Finality is not the only condition the comparison map can be tested against. The coefficient diagrams available at $\gamma_\sigma$ are the pullbacks $(q_2^J)^\ast F$, constant on each fibre of $q_2^J$, and what they impose is indexed by $\Pos$ rather than by $J$.

\begin{proposition}[What the test sheaves detect]\label{prop:representable-test}
Let $\Phi$ be total at $(\sigma;a,b)$, and keep $J$, $K$, $p$ and $q_2^J$ as in the proof of Proposition \ref{prop:final}. For $u\in\Pos$ put
\[
   W_u=\{(y,v)\in J:\ u\le v\},
   \qquad
   \Phi^{-1}(W_u)=\{\kappa\in K:\ u\le p(\kappa)\},
\]
both up-closed. Then $\gamma_\sigma$ is a quasi-isomorphism for every $F\in\Shv(\Pos;k)$ if and only if, for every $u\in\Pos$, the map
\[
   C_\ast\bigl(\Delta(\Phi^{-1}W_u);k\bigr)\longrightarrow C_\ast\bigl(\Delta(W_u);k\bigr)
\]
induced by $\Phi$ is a quasi-isomorphism; equivalently, if and only if $\gamma_\sigma$ is a quasi-isomorphism on the projective sheaves $k_{\st(u)}$, $u\in\Pos$, of Lemma \ref{lem:enough-projectives}. At $F=k_{\st(u)}$ the square \eqref{eq:comparison-bar-square} identifies $\gamma_\sigma$ with that map of order complexes.
\end{proposition}

\begin{proof}
Apply Lemma \ref{lem:updetect} with $L=\Pos$ and $\varphi=q_2^J$, whose $W_u$ is the displayed set and for which $\Phi^{-1}(W_u)=\{\kappa:u\le q_2^J\Phi(\kappa)\}=\{\kappa:u\le p(\kappa)\}$ by $q_2^J\Phi=p$. The coefficient diagrams occurring in \eqref{eq:comparison-bar-square} are the $D=(q_2^J)^\ast F$, and $(q_2^J)^\ast k_{\st(u)}=k_{W_u}$, so the last sentence is Lemma \ref{lem:support-restriction}\textup{(i)}.
\end{proof}

By Corollary \ref{cor:final-criterion} the condition of Proposition \ref{prop:representable-test} is implied by $k$-homological finality, which asks the same over each principal up-set $\st_J(j)$, $j\in J$, and its preimage. The two are compared through the minimal elements: each $W_u$, being up-closed in a finite poset, is a union of the smaller family,
\[
   W_u=\bigcup\{\,\st_J(j)\;:\;j\ \text{minimal in}\ W_u\,\},
\]
so the test sheaves impose over each such union what finality imposes over each of its members. The two conditions coincide as soon as every $\st_J(j)$ is itself one of the $W_u$, which is what a minimal apex supplies (Theorem \ref{thm:minimal-sharp}). Whether the coarser family forces the finer one in general is the open direction of Conjecture \ref{conj:sharp}.

\begin{proposition}[Coherence of the meet comparisons]\label{prop:meet-coherence}
Assume that, for every $\sigma\in\Pos$ and all $a,b\ge0$, the meet assignment $\Phi_{\sigma,a,b}$ is total, so the meet functor $\mathcal K_{a,b}(\sigma)\to J_{a+b}(\sigma)$. Then the meet comparisons $\gamma_{a,b}$ supplied by Proposition \ref{prop:final}
are compatible with the derived transition maps of Proposition \ref{prop:derived-exist}(iv) and satisfy the unit and associativity identities; namely:
\begin{enumerate}
\item[\textup{(i)}] If $a\le a'$ and $b\le b'$, then the square of natural transformations
\[
\begin{tikzcd}[column sep=huge,row sep=large]
\LC_a\LC_b\arrow[r,"\gamma_{a,b}"]\arrow[d,"\LC_a\eta_{b,b'}"'] &
\LC_{a+b}\arrow[dd,"\eta_{a+b,a'+b'}"]\\
\LC_a\LC_{b'}\arrow[d,"\eta_{a,a'}\LC_{b'}"'] & \\
\LC_{a'}\LC_{b'}\arrow[r,"\gamma_{a',b'}"] & \LC_{a'+b'}
\end{tikzcd}
\]
commutes, that is $\eta_{a+b,a'+b'}\circ\gamma_{a,b} = \gamma_{a',b'}\circ(\eta_{a,a'}\LC_{b'}) \circ(\LC_a\eta_{b,b'}).$
\item[\textup{(ii)}] After the identification $\LC_0\simeq\id$, the two unit triangles
\[
\begin{tikzcd}[column sep=large,row sep=large]
\LC_a\arrow[r,"\LC_a\eta_b"]\arrow[dr,"\eta_{a,a+b}"'] &
\LC_a\LC_b\arrow[d,"\gamma_{a,b}"]\\
& \LC_{a+b}
\end{tikzcd}
\qquad\qquad
\begin{tikzcd}[column sep=large,row sep=large]
\LC_b\arrow[r,"\eta_a\LC_b"]\arrow[dr,"\eta_{b,a+b}"'] &
\LC_a\LC_b\arrow[d,"\gamma_{a,b}"]\\
& \LC_{a+b}
\end{tikzcd}
\]
commute, that is $\gamma_{a,b}\circ\LC_a(\eta_b)=\eta_{a,a+b},\,\gamma_{a,b}\circ(\eta_a\LC_b)=\eta_{b,a+b}.$
\item[\textup{(iii)}] For all $a,b,c$ the natural associativity square
\[
\begin{tikzcd}[column sep=huge,row sep=large]
\LC_a\LC_b\LC_c\arrow[r,"\gamma_{a,b}\LC_c"]\arrow[d,"\LC_a\gamma_{b,c}"'] &
\LC_{a+b}\LC_c\arrow[d,"\gamma_{a+b,c}"]\\
\LC_a\LC_{b+c}\arrow[r,"\gamma_{a,b+c}"] & \LC_{a+b+c}
\end{tikzcd}
\]
commutes, that is $\gamma_{a+b,c}\circ(\gamma_{a,b}\LC_c)=\gamma_{a,b+c}\circ(\LC_a\gamma_{b,c}).$
\end{enumerate}
Parts \textup{(i)}, \textup{(ii)} and \textup{(iii)} are the naturality of $\mu_{a,b}$ in $(a,b)$, the unit coherence and the associativity coherence of Definition \ref{def:flow}, so the choice $\mu_{a,b}=\gamma_{a,b}$ makes $\{\LC_a\}$ a flow indexed by $(\mathbb Z_{\ge0},\le,+)$, strong on any domain where the $\gamma_{a,b}$ are quasi-isomorphisms; that is the coherence used in Section \ref{sec:metric}.
\end{proposition}

\begin{proof}
The proof is on the finite bar models of Propositions \ref{prop:bar} and \ref{prop:groth}. Equality of the relevant index maps, with the same coefficient functor to $\Pos$, gives equality of the induced normalized bar maps.

For parameter naturality, fix $\sigma$ and put
\[
   K=\Ka_{a,b}(\sigma),\quad K''=\Ka_{a,b'}(\sigma),\quad K'=\Ka_{a',b'}(\sigma),\quad
   J=J_{a+b}(\sigma),\quad J'=J_{a'+b'}(\sigma).
\]
Let $i_b\colon K\hookrightarrow K''$, $i_a\colon K''\hookrightarrow K'$ and $i_J\colon J\hookrightarrow J'$ be the inclusions. For $\kappa=(s,t,r,u)\in K$,
\[
   i_J\Phi_{\sigma,a,b}(\kappa)
   =(s\meet r,u)
   =
   \Phi_{\sigma,a',b'}i_ai_b(\kappa).
\]
Thus $i_J\circ\Phi_{\sigma,a,b}=\Phi_{\sigma,a',b'}\circ i_a\circ i_b$ strictly. Writing $p_K,p_{K''},p_{K'}$ for the $u$-coordinate projections from $K,K'',K'$ and $q_J,q_{J'}$ for those from $J,J'$.
Functoriality of the normalized bar construction gives the strictly commuting diagram
\[
\begin{tikzcd}[column sep=huge,row sep=large]
\operatorname{Bar}(K,p_K^\ast F) \arrow[r,"{(\Phi_{\sigma,a,b})_\#}"] \arrow[d,"{(i_b)_\#}"'] &
\operatorname{Bar}(J,q_J^\ast F) \arrow[dd,"{(i_J)_\#}"] & \\
\operatorname{Bar}(K'',p_{K''}^\ast F) \arrow[d,"{(i_a)_\#}"'] & & \\
\operatorname{Bar}(K',p_{K'}^\ast F) \arrow[r,"{(\Phi_{\sigma,a',b'})_\#}"'] & \operatorname{Bar}(J',q_{J'}^\ast F).
\end{tikzcd}
\]
The two left vertical arrows are the bar representatives of the inner transition $b\le b'$ and the outer transition $a\le a'$. The right vertical arrow $(i_J)_\#$ is induced by the direct inclusion $J_{a+b}(\sigma)\hookrightarrow J_{a'+b'}(\sigma)$ and represents the transition $a+b\le a'+b'$. Equality of the two paths from $\operatorname{Bar}(K,p_K^\ast F)$ to $\operatorname{Bar}(J',q_{J'}^\ast F)$ is the displayed identity.

For the unit identities, use the terminal-object model of Lemma \ref{lem:unit} for $\LC_0\simeq\id$. For the first identity, an object $(s,t)\in J_a(\sigma)$ is sent by $\LC_a(\eta_b)$ to $(s,t,t,t)\in\Ka_{a,b}(\sigma)$, and then
\[
   \Phi_{\sigma,a,b}(s,t,t,t)=(s\meet t,t)=(s,t),
\]
which is the inclusion $J_a(\sigma)\subseteq J_{a+b}(\sigma)$ representing $\eta_{a,a+b}$. For the second identity, an object $(r,u)\in J_b(\sigma)$ is sent by $\eta_a\LC_b$ to $(\sigma,\sigma,r,u)\in\Ka_{a,b}(\sigma)$, and
\[
   \Phi_{\sigma,a,b}(\sigma,\sigma,r,u)=(\sigma\meet r,u)=(r,u),
\]
since $r\le\sigma$ in $J_b(\sigma)$; this is the inclusion $J_b(\sigma)\subseteq J_{a+b}(\sigma)$ representing $\eta_{b,a+b}$.

For associativity, use the triple Grothendieck index whose objects are
\[
   (s,t,r,u,v,w),
   \qquad
   (s,t)\in J_a(\sigma),\quad (r,u)\in J_b(t),\quad (v,w)\in J_c(u).
\]
The composite $\gamma_{a+b,c}\circ(\gamma_{a,b}\LC_c)$ sends such an object to
\[
   ((s\meet r)\meet v,w)\in J_{a+b+c}(\sigma),
\]
while $\gamma_{a,b+c}\circ(\LC_a\gamma_{b,c})$ sends it to
\[
   (s\meet(r\meet v),w)\in J_{a+b+c}(\sigma).
\]
The totality hypotheses make both displayed meets exist, and both are the greatest lower bound of $\{s,r,v\}$ in $\Pos$, so they are equal. The coefficient value is $F(w)$ on both sides. The two functors from the triple index to $J_{a+b+c}(\sigma)$ are therefore strictly equal, and the induced bar maps agree on every generator. The equalities are natural in $F$ and in $\sigma$, because all restriction maps are induced by the same inclusions of window-comma posets.
\end{proof}

\begin{proposition}[The strong flow fails at every branching length-two interval]\label{prop:diamond-nogo}
Let $b<t$ in $\Pos$ with $\len(b,t)=2$, write
\[
   (b,t)=\{m_1,\dots,m_r\},
\]
an antichain, and suppose $b$ is minimal in $\Pos$ or, more generally, that no $s<b$ satisfies $\len(s,t)\le 2$. Call $[b,t]$ a \emph{branching length-two interval} when $r\ge2$; the term carries this condition on $b$ throughout. The case $r=2$ is a \emph{diamond interval},
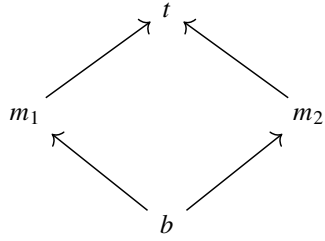
\[
\begin{tikzcd}[row sep=large,column sep=large]
& t & \\
m_1\arrow[ur] & & m_2\arrow[ul]\\
& b\arrow[ul]\arrow[ur] &
\end{tikzcd}
\]
Then for $F=\sky_t$ at the apex $\sigma=b$,
\[
   H_0\bigl((\LC_2\sky_t)(b)\bigr)\cong k,
   \qquad
   H_0\bigl((\LC_1\LC_1\sky_t)(b)\bigr)\cong k^{\,r} .
\]
The exponent is the cardinality of the open interval $(b,t)$; the two dimensions therefore differ as soon as $r\ge2$, and then $\LC_1\LC_1\not\simeq\LC_2$ on $\Pos$, no comparison map being involved. If moreover the meet assignment $\Phi$ is total at $(b;1,1)$, then $\gamma_b$ is defined and, for $r\ge2$, is not a quasi-isomorphism, and $\Phi$ fails $k$-homological finality. The strong flow law then fails on the face poset of every finite regular cell complex of dimension $\ge2$.
\end{proposition}

\begin{proof}
\emph{Single window.} By Proposition \ref{prop:bar}, $(\LC_2\sky_t)(b)\simeq\operatorname{Bar}(J_2(b),D)$ with $D(s,x)=\sky_t(x)$. The hypothesis gives $s=b$ for every $(s,x)\in J_2(b)$, so the only object carrying a nonzero value is $(b,t)$ and $D=\sky_{(b,t)}$; and $\len(b,t)=2$ leaves no object of $J_2(b)$ strictly above it. So Lemma \ref{lem:bar-pair}\textup{(ii)} gives $H_0=\widetilde H_{-1}(\Delta(\varnothing);k)=k$, with higher homology zero.

\emph{Iterate.} By Proposition \ref{prop:groth}, $(\LC_1\LC_1\sky_t)(b)\simeq\hocolim_{\Ka_{1,1}(b)}\sky_t(u)$, whose $H_0$ is the ordinary colimit $\colim_{\Ka_{1,1}(b)}\sky_t(u)$. The coefficient $\sky_t(u)$ is nonzero (equal to $k$) at a quadruple $(s,x,\rho,u)\in\Ka_{1,1}(b)$ if and only if $u=t$, i.e.\ $s=b$, $\len(b,x)\le1$, $\rho\le x$, and $\len(\rho,t)\le1$, forcing $\rho\in(b,t)$ and then $x=\rho$. Thus the nonzero locus is the $r$ objects $\kappa_i=(b,m_i,m_i,t)$, $1\le i\le r$. No two are identified and none is killed: an outgoing arrow $\kappa_i\le(s',x',\rho',u')$ has $s'=b$ and $x'=\rho'=m_i$ by the unit-window constraints, and then $\len(m_i,u')\le1$ together with $t\le u'$ forces $u'=t$, so $\sky_t(u')=k$ with the identity coefficient map; and there is no morphism between $\kappa_i$ and $\kappa_j$ for $i\ne j$, their $x$- and $\rho$-coordinates $m_i,m_j$ being incomparable. The colimit is $k^{\,r}$.

Suppose now in addition that $\Phi$ is total at $(b;1,1)$, so that $\gamma_b$ is defined; then it is not a quasi-isomorphism, by the two displayed dimensions. In finality terms, the comma fibre
\[
   (b,t)\downarrow\Phi=\{\kappa\in\Ka_{1,1}(b):(b,t)\le\Phi(\kappa)\}
\]
consists of the quadruples with $u\ge t$ and $b\le s\meet\rho$; the length constraints computed above force $u=t$ and $\kappa\in\{\kappa_1,\dots,\kappa_r\}$, so the fibre is the $r$-point antichain, drawn here for $r=2$:
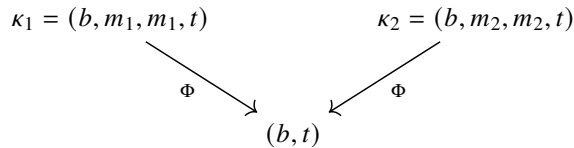
\[
\begin{tikzcd}[row sep=large,column sep=small]
\kappa_1=(b,m_1,m_1,t)\arrow[dr,"\Phi"'] & & \kappa_2=(b,m_2,m_2,t)\arrow[dl,"\Phi"]\\
& (b,t) &
\end{tikzcd}
\]
There is no morphism between distinct $\kappa_i$, so its reduced $H_0$ is $k^{\,r-1}$, nonzero for $r\ge2$, and $\Phi$ fails $k$-homological finality, and with it homotopy finality.

For the last clause, a $2$-cell $t$ and a vertex $v\le t$ span an interval of length two with $v$ minimal, and there $r=2$ by the diamond property of Appendix \ref{ssec:face-posets} (Example \ref{ex:bigon}).
\end{proof}

\begin{remark}[Where the two obstructions coincide]\label{rem:diamond-totality}
The totality hypothesis in the finality clause is not automatic, and on a face poset it fails at the apex where the proposition is applied. Read as a poset in its own right (graded, but not itself a face poset; Appendix \ref{ssec:face-posets}), a diamond interval $\Pos=[b,t]$ has least element $b$, so $s\meet r=b$ exists for every quadruple of $\Ka_{1,1}(b)$ and $\len(b,u)\le2$ throughout: there $\Phi$ is total and the proposition exhibits the total-but-non-final case. The augmented face poset of a simplicial complex supplies the same case at its empty face, without leaving the face-poset setting (Example \ref{ex:simplicial-apex}). Inside a regular cell complex of dimension $\ge2$ the apex is a vertex $v$ of a $2$-cell $t$, and totality fails: choose an edge $e\le t$ with $v\le e$ and let $w\ne v$ be its other vertex, which exists by regularity. Then $(v,e,w,w)\in\Ka_{1,1}(v)$, and $v\meet w$ does not exist, distinct vertices being minimal in a poset of nonempty cells:
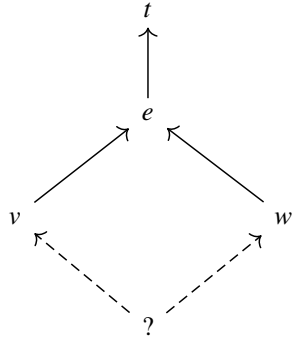
\[
\begin{tikzcd}[row sep=large,column sep=large]
& t & \\
& e\arrow[u] & \\
v\arrow[ur] & & w\arrow[ul]\\
& {?}\arrow[ul,dashed]\arrow[ur,dashed] &
\end{tikzcd}
\]
The dashed element is the lower bound $v\meet w$ that the quadruple $(v,e,w,w)$ requires and that $\Pos$ does not provide. So at such an apex both mechanisms of this section fire simultaneously (the meet is missing \emph{and} the two stalk dimensions differ), and the failure of the strong flow, which is what Proposition \ref{prop:diamond-nogo} asserts unconditionally, is not to be read there as ``totality with non-finality''.
\end{remark}

\begin{proposition}[Adjoint transposition of the flow problem]\label{prop:transposition}
For all $a,b\ge0$ there are adjunctions $\LC_a\LC_b\dashv\RC^b\RC^a$ and $\LC_{a+b}\dashv\RC^{a+b}$ (Proposition \ref{prop:derived-exist}(iii)). A natural isomorphism $\LC_a\LC_b\simeq\LC_{a+b}$ therefore exists if and only if such an isomorphism $\RC^b\RC^a\simeq\RC^{a+b}$ exists: the strong derived-flow law holds for the family $\{\LC_a\}$ if and only if it holds for the family $\{\RC^a\}$.
\end{proposition}

\begin{proof}
Composites of adjoints are adjoint \cite[Prop.~4.3.4]{RiehlCTIC}, and adjoints are unique up to canonical natural isomorphism \cite[Prop.~4.3.1]{RiehlCTIC}.
\end{proof}

So Proposition \ref{prop:diamond-nogo} transposes to $\RC^1\RC^1\not\simeq\RC^2$ on any poset carrying a branching length-two interval; the right adjoint is the matching companion of $C_a$, useful for hom-side computations as in Example \ref{ex:lambda}, but it does not evade the obstruction. The flow does hold on the rooted forests (Proposition \ref{prop:rooted}), chains among them.

\begin{conjecture}[Homological sharpness of finality]\label{conj:sharp}
For every finite graded poset $\Pos$, every $\sigma\in\Pos$, and every $a,b\ge0$ for which the meet assignment $\Phi_{\Pos,\sigma,a,b}$ is total, the comparison map $\gamma_\sigma$ is a quasi-isomorphism for \emph{every} $F\in\Shv(\Pos;k)$ if and only if $\Phi_{\Pos,\sigma,a,b}$ is $k$-homologically final.
\end{conjecture}

The ``if'' direction is Proposition \ref{prop:final}; the converse is open only in part. Theorem \ref{thm:minimal-sharp} proves it at every minimal $\sigma$, whatever the windows, and Corollary \ref{cor:minimal-unit} makes the criterion there purely order-theoretic when $a=b=1$. Away from a minimal apex, Corollary \ref{cor:sharp-essential} proves it wherever the defect of $\Phi$ is essential, which by Corollary \ref{cor:unit-window-essential} it is at every unit first window. What stays open is the range in which some defect is inessential; Remark \ref{rem:detection-scope} delimits it, and the gap recorded after Proposition \ref{prop:representable-test} is the open direction itself.

At a minimal apex that gap closes, $q_2^J$ being injective there, and the test sheaves see the comma fibres one at a time.

\begin{theorem}[Sharpness at a minimal apex]\label{thm:minimal-sharp}
Let $\Pos$ be a finite graded poset, let $\sigma$ be minimal in $\Pos$, let $a,b\ge0$, and suppose the meet assignment $\Phi=\Phi_{\sigma,a,b}$ is total. Then, writing $p(s,t,r,u)=u$ as in Proposition \ref{prop:representable-test}, $\Phi(\kappa)=(\sigma,p(\kappa))$ for every $\kappa\in\Ka_{a,b}(\sigma)$, the comma fibre over $j=(\sigma,u)$ is
\[
   (\sigma,u)\downarrow\Phi\;=\;\{\kappa\in\Ka_{a,b}(\sigma)\;:\;u\le p(\kappa)\},
\]
and the following are equivalent.
\begin{enumerate}
\item[\textup{(i)}] $\gamma_\sigma$ is a quasi-isomorphism for every $F\in\Shv(\Pos;k)$.
\item[\textup{(ii)}] $\gamma_\sigma$ is a quasi-isomorphism for the sheaves $k_{\st(u)}$, $u\in\st_{a+b}(\sigma)$.
\item[\textup{(iii)}] $\Phi$ is $k$-homologically final.
\end{enumerate}
In particular Conjecture \ref{conj:sharp} holds at every minimal $\sigma\in\Pos$, for all $a,b$ at which $\Phi$ is total.
\end{theorem}

\begin{proof}
Minimality of $\sigma$ forces $y=\sigma$ for every $(y,u)\in J:=J_{a+b}(\sigma)$, so $q_2^J$ carries $J$ isomorphically onto $\st_{a+b}(\sigma)$, and it forces $s=\sigma$ for every quadruple of $K:=\Ka_{a,b}(\sigma)$. Where the meet $\sigma\meet r$ exists it is a lower bound of $\sigma$, so it equals $\sigma$; totality therefore gives $\Phi(s,t,r,u)=(\sigma,u)$, that is $\Phi(\kappa)=(\sigma,p(\kappa))$. The closed form \eqref{eq:comma-fibre} of the comma fibre over $(\sigma,u)$ then reduces to the displayed one, the conditions $\sigma\le s$ and $\sigma\le r$ being automatic.

\textup{(i)} implies \textup{(ii)} trivially, and \textup{(iii)} implies \textup{(i)} by Proposition \ref{prop:final}.

Assume \textup{(ii)} and fix $u\in\st_{a+b}(\sigma)$, so that $j=(\sigma,u)\in J$. In the notation of Proposition \ref{prop:representable-test},
\[
   W_u=\{(\sigma,v)\in J:\ u\le v\}=\st_J(j),
\]
the two descriptions agreeing because every object of $J$ has first coordinate $\sigma$; and $\Phi^{-1}(W_u)=\{\kappa:u\le p(\kappa)\}=j\downarrow\Phi$. By Proposition \ref{prop:representable-test} the hypothesis at $F=k_{\st(u)}$ says that
\[
   C_\ast\bigl(\Delta(j\downarrow\Phi);k\bigr)\longrightarrow C_\ast\bigl(\Delta(\st_J(j));k\bigr)
\]
is a quasi-isomorphism. The right-hand complex is quasi-isomorphic to $k$ through its augmentation, $\st_J(j)$ having least element $j$ (Lemma \ref{lem:cone-point}); augmentations being natural for simplicial maps, the displayed map is a quasi-isomorphism if and only if $j\downarrow\Phi$ is nonempty with vanishing reduced $k$-homology. As every object of $J$ is such a $j$, this is \textup{(iii)}.
\end{proof}

\begin{proposition}[At a saturated minimal apex the law itself fails]\label{prop:sat-nogo}
Let $\sigma$ be minimal in $\Pos$, let $a,b\ge0$, suppose the meet assignment $\Phi=\Phi_{\sigma,a,b}$ is total, and let $u\in\Pos$ satisfy $\len(\sigma,u)=a+b$. Then
\[
   (\LC_{a+b}\sky_u)(\sigma)\;\simeq\;k\ \text{in degree }0,
   \qquad
   (\LC_a\LC_b\sky_u)(\sigma)\;\simeq\;C_\ast\bigl(\Delta((\sigma,u)\downarrow\Phi);k\bigr).
\]
If the comma fibre $(\sigma,u)\downarrow\Phi$ is not $k$-acyclic, the two functors take non-isomorphic values at $\sky_u$, and no natural isomorphism $\LC_a\LC_b\cong\LC_{a+b}$ exists, whether or not it is $\gamma_{a,b}$.
\end{proposition}

\begin{proof}
Put $J=J_{a+b}(\sigma)$ and $K=\Ka_{a,b}(\sigma)$. By the first paragraph of the proof of Theorem \ref{thm:minimal-sharp}, minimality of $\sigma$ forces $y=\sigma$ for every $(y,v)\in J$ and $s=\sigma$ for every $(s,t,r,v)\in K$, while totality forces $\sigma\meet r=\sigma$, that is $\sigma\le r$. Additivity of $\len$ then gives $\len(\sigma,v)=\len(\sigma,r)+\len(r,v)\le\len(\sigma,t)+b\le a+b$ at every object of $K$, and every object of $J$ is a $(\sigma,v)$ with $\len(\sigma,v)\le a+b$. Since $\len(\sigma,u)=a+b$, no $v>u$ occurs as a second coordinate in $J$ or as a value of $p$ on $K$.

\emph{Single window.} The diagram $(q_2^J)^\ast\sky_u$ on $J$ is the skyscraper at the object $(\sigma,u)$, and $J_{>(\sigma,u)}=\varnothing$ by the previous paragraph, so Lemma \ref{lem:bar-pair}\textup{(ii)} gives $H_0=\widetilde H_{-1}(\Delta(\varnothing);k)=k$ and nothing above it; Proposition \ref{prop:bar} reads that as the first display.

\emph{Iterate.} The diagram $p^\ast\sky_u$ on $K$ is $k$ on $A=\{\kappa\in K:u\le p(\kappa)\}$, up-closed, and $0$ elsewhere, no $\kappa$ having $p(\kappa)>u$; it is $k_A$ in the notation of Lemma \ref{lem:support-restriction}\textup{(i)}. That part and Proposition \ref{prop:groth} give $(\LC_a\LC_b\sky_u)(\sigma)\simeq C_\ast(\Delta(A);k)$, and $A=\Phi^{-1}(W_u)=(\sigma,u)\downarrow\Phi$ by the proof of Theorem \ref{thm:minimal-sharp}.

A fibre that is not $k$-acyclic leaves $C_\ast(\Delta(A);k)$ not quasi-isomorphic to $k$ in degree $0$: it is $0$ if the fibre is empty, and otherwise carries a nonzero reduced homology class. Stalk functors are exact, so a natural isomorphism of the two functors would make these two objects of $\Db(\mathrm{Vec}_k)$ isomorphic.
\end{proof}

The hypotheses of the proposition and of Proposition \ref{prop:diamond-nogo} are incomparable: the latter asks nothing of totality but treats unit windows only, while totality is what puts $\sigma$ below every $r$ here.

\begin{lemma}[Reduced-chain model for $\operatorname{Def}_\Phi$]\label{lem:defect-reduced}
Fix $\Pos,\sigma,a,b$. Put $J=J_{a+b}(\sigma)$ and $K=\Ka_{a,b}(\sigma)$, assume that the meet assignment $\Phi_{\Pos,\sigma,a,b}$ is total, and write
\[
   \Phi\colon K\to J
\]
for the resulting meet functor. For $j\in J$, set
\[
   R_\Phi(j)=C_\ast\bigl(N(j\downarrow\Phi);k\bigr),
   \qquad
   \operatorname{Def}_\Phi(j)
   =
   \operatorname{Cone}\bigl(R_\Phi(j)\xrightarrow{\varepsilon_j} k\bigr)[-1],
\]
where the map $\varepsilon_j$ is the augmentation. The inclusions $j'\downarrow\Phi\subseteq j\downarrow\Phi$ for $j\le j'$ make $R_\Phi$ and $\operatorname{Def}_\Phi$ contravariant $J$-diagrams; $R_\Phi$ is the \emph{finality-fibre diagram} of $\Phi$ (cf. Proposition \ref{prop:final}). Then:
\begin{enumerate}
\item[\textup{(i)}] Writing $k[-1]$ for the constant contravariant $J$-diagram concentrated in degree $-1$, there is a short exact sequence of contravariant $J$-diagrams of complexes
\[
   0\longrightarrow k[-1]\longrightarrow\operatorname{Def}_\Phi\longrightarrow R_\Phi\longrightarrow 0 ,
\]
degreewise split, whose connecting map is the augmentation $H_0R_\Phi(j)\to k$ up to sign.
\item[\textup{(ii)}] $\operatorname{Def}_\Phi$ is naturally the augmented reduced chain complex
\[
   j\longmapsto \widetilde C_\ast\bigl(N(j\downarrow\Phi);k\bigr).
\]
\end{enumerate}
Here $\widetilde C_\ast$ is the augmented complex, carrying the augmentation term $k$ in degree $-1$; shifts and cones are those fixed in Remark \ref{rem:homotopical}.
This gives
\[
   H_q\operatorname{Def}_\Phi(j)
   \cong
   \widetilde H_q\bigl(N(j\downarrow\Phi);k\bigr)
   \qquad(q\ge0),
\]
and if $j\downarrow\Phi=\varnothing$ then
\[
   H_{-1}\operatorname{Def}_\Phi(j)\cong k.
\]
Thus $\operatorname{Def}_\Phi(j)$ is acyclic if and only if the fibre $j\downarrow\Phi$ is nonempty and $k$-acyclic.
\end{lemma}

\begin{proof}
For $j\le j'$ the defining condition for the comma fibre is contravariant in $j$, so $j'\downarrow\Phi\subseteq j\downarrow\Phi$. Taking nerves and simplicial chains gives the contravariant $J$-diagram $R_\Phi$, and these inclusions commute with the augmentations to $k$, and so give a contravariant $J$-diagram $\operatorname{Def}_\Phi$ as well.

(i) Unwinding the conventions of Remark \ref{rem:homotopical}, $\operatorname{Def}_\Phi(j)_n=R_\Phi(j)_n$ for $n\ge0$ and $\operatorname{Def}_\Phi(j)_{-1}=k$, with differential that of $R_\Phi(j)$ in degrees $\ge1$ and $-\varepsilon_j$ in degree $0$. The summand $k$ in degree $-1$ receives the differential but has nothing below it, so it is a subcomplex, and the quotient is $R_\Phi(j)$ with its own differential. The inclusion and projection are natural in $j$ because the fibre inclusions commute with the augmentations, giving (i); the sequence is degreewise split because every term is a $k$-vector space. Since $H_q(k[-1])=k$ for $q=-1$ and $0$ otherwise, the long exact sequence collapses to $H_q\operatorname{Def}_\Phi(j)\cong H_qR_\Phi(j)$ for $q\ge1$ together with
\[
   0\to H_0\operatorname{Def}_\Phi(j)\to H_0R_\Phi(j)\xrightarrow{\ -\varepsilon_j\ }k\to H_{-1}\operatorname{Def}_\Phi(j)\to 0,
\]
whose middle map is $-\varepsilon_j$ by sign convention, induced by the two-row short exact sequence
\[
\begin{tikzcd}[row sep=1.8em]
0 \arrow[d,"0"] \arrow[r,"0"] & R_\Phi(j)_0\oplus 0 \arrow[d,"(0{,}-\varepsilon_j(x))"] \arrow[r,equals] & R_\Phi(j)_0 \arrow[d,"0"] \\
k \arrow[r,equals] & 0\oplus k \arrow[r] & 0.
\end{tikzcd}
\]

(ii) For a simplicial set $X$, with $C_\ast(X;k)$ normalized as in Lemma \ref{lem:order-nerve}, let $\widetilde C_\ast(X;k)$ be the augmented chain complex
\[
   \cdots\longrightarrow C_1(X;k)\longrightarrow C_0(X;k)\xrightarrow{\ \varepsilon\ }k
\]
with $k$ in degree $-1$. By the computation just made, applied to $\varepsilon\colon C_\ast(X;k)\to k$, the complex $\operatorname{Cone}(\varepsilon)[-1]$ agrees with $\widetilde C_\ast(X;k)$ in every degree and differs from it only in the sign of the degree-$0$ differential; negating the degree-$(-1)$ summand is therefore an isomorphism of complexes, natural in $X$ and so natural for the fibre inclusions above. Taking $X=N(j\downarrow\Phi)$ gives (ii). The homology of $\widetilde C_\ast(X;k)$ is reduced homology in degrees $q\ge0$, and for $X=\varnothing$ the complex is $k$ in degree $-1$; these are the two displays in the statement.
\end{proof}

\begin{lemma}[Only unsaturated pairs can be defective]\label{lem:unsaturated}
Let $\Phi$ be total at $(\sigma;a,b)$ and let $j=(y,u)\in J_{a+b}(\sigma)$. If $\len(y,u)\le b$, then $(y,y,y,u)$ is a least element of $j\downarrow\Phi$, so $N(j\downarrow\Phi)$ is contractible and $\operatorname{Def}_\Phi(j)$ is acyclic. The \emph{defective} locus therefore satisfies
\[
   \{j\in J_{a+b}(\sigma):H_\ast\operatorname{Def}_\Phi(j)\ne0\}
   \;\subseteq\;
   \{(y,u)\;:\;b<\len(y,u)\le a+b\} .
\]
\end{lemma}

\begin{proof}
The quadruple $\kappa_0=(y,y,y,u)$ lies in $\Ka_{a,b}(\sigma)$: $y\le\sigma$ because $j\in J_{a+b}(\sigma)$, $\len(y,y)=0\le a$, $y\le y$, and $\len(y,u)\le b$ by hypothesis. It lies in \eqref{eq:comma-fibre} because $y\le y$ and $u\le u$. For any $(s,t,r,v)$ in \eqref{eq:comma-fibre} one has $y\le s$, $y\le t$ (as $y\le s\le t$), $y\le r$ and $u\le v$, so $\kappa_0\le(s,t,r,v)$ componentwise. So $\widetilde H_\ast\bigl(N(j\downarrow\Phi);k\bigr)=0$ by Lemma \ref{lem:cone-point}, and Lemma \ref{lem:defect-reduced} gives acyclicity. The displayed containment is the contrapositive, together with $\len(y,u)\le a+b$ for $j\in J_{a+b}(\sigma)$.
\end{proof}

\begin{corollary}[The defect needs both windows short]\label{cor:saturated}
Let $\Phi$ be total at $(\sigma;a,b)$. In either of the following cases $\operatorname{Def}_\Phi$ is acyclic at every object, $\Phi$ is homotopy final, and $\gamma_\sigma$ is a quasi-isomorphism for every $F\in\Shv(\Pos;k)$.
\begin{enumerate}
\item[\textup{(i)}] $\len(y,u)\le b$ for every $(y,u)\in J_{a+b}(\sigma)$, as happens whenever $b$ is at least the length of $\Pos$.
\item[\textup{(ii)}] $\Pos$ has a greatest element $\hat1$ and $\len(\sigma,\hat1)\le a$.
\end{enumerate}
\end{corollary}

\begin{proof}
In case \textup{(i)} the locus displayed in Lemma \ref{lem:unsaturated} is empty, so every comma fibre has a least element. In case \textup{(ii)} the quadruple $\kappa_{\hat1}=(\sigma,\hat1,\hat1,\hat1)$ lies in $\Ka_{a,b}(\sigma)$, since $\sigma\le\sigma$, $\len(\sigma,\hat1)\le a$ and $\len(\hat1,\hat1)=0\le b$; it lies in every comma fibre \eqref{eq:comma-fibre}, since $y\le\sigma$, $y\le\hat1$ and $u\le\hat1$; and it dominates $\Ka_{a,b}(\sigma)$ componentwise, every quadruple there having first coordinate $\le\sigma$ and the other three $\le\hat1$. So each fibre has a greatest element instead of a least one. Either way Lemma \ref{lem:cone-point} applies, and Lemma \ref{lem:defect-reduced} with Proposition \ref{prop:final} gives the rest.
\end{proof}

At a unit first window the whole defective locus is saturated: the containment of Lemma \ref{lem:unsaturated} reads $b<\len(y,u)\le b+1$ when $a=1$, leaving $\len(y,u)=b+1$ alone. So at a minimal apex with $a=1$ Proposition \ref{prop:sat-nogo} reaches every defective object, and there the failure of $\gamma_\sigma$ is the failure of the law $\LC_1\LC_b\simeq\LC_{1+b}$.

A minimal apex and two unit windows leave one free rank in the comma fibre, and that rank is the whole of it: the fibre is the open interval between $\sigma$ and $u$, and the defect there measures branching and nothing else.

\begin{proposition}[Unit windows at a minimal apex: the fibre is an open interval]\label{prop:unit-fibre}
Let $\Pos$ be a finite graded poset, let $\sigma$ be minimal in $\Pos$, let $a=b=1$, and suppose $\Phi$ is total at $(\sigma;1,1)$. Every $j\in J_2(\sigma)$ has the form $(\sigma,u)$, and:
\begin{enumerate}
\item[\textup{(i)}] if $\len(\sigma,u)\le1$ then $j\downarrow\Phi$ has the least element $(\sigma,\sigma,\sigma,u)$, so $\operatorname{Def}_\Phi(j)$ is acyclic;
\item[\textup{(ii)}] if $\len(\sigma,u)=2$ then $\kappa\mapsto t$ is an isomorphism of posets
\[
   j\downarrow\Phi\;\xrightarrow{\ \sim\ }\;(\sigma,u),
\]
and the open interval $(\sigma,u)$ is an antichain.
\end{enumerate}
Then $\operatorname{Def}_\Phi$ is concentrated in homological degree $0$,
\[
   H_q\operatorname{Def}_\Phi(j)=0\ \ (q\ge1),
   \qquad
   H_0\operatorname{Def}_\Phi(j)\cong k^{\,\lvert(\sigma,u)\rvert-1},
\]
and $j$ is defective if and only if $\len(\sigma,u)=2$ and the open interval $(\sigma,u)$ has at least two elements.
\end{proposition}

\begin{proof}
Minimality of $\sigma$ forces $y=\sigma$ for every $(y,u)\in J_2(\sigma)$, and $s=\sigma$ for every quadruple of $\Ka_{1,1}(\sigma)$. Part (i) is Lemma \ref{lem:unsaturated} with $b=1$.

For (ii) write $\rho=\rk(\sigma)$, so $\rk(u)=\rho+2$, and let $(\sigma,t,r,v)$ lie in \eqref{eq:comma-fibre}. Membership gives $\len(\sigma,t)\le1$, $r\le t$, $\len(r,v)\le1$ and $u\le v$, and in ranks these leave no slack:
\[
   \rho+1\;\ge\;\rk(t)\;\ge\;\rk(r)\;\ge\;\rk(v)-1\;\ge\;\rk(u)-1\;=\;\rho+1 .
\]
Every inequality is an equality, and two of them carry a relation as well: $r\le t$ with $\rk(r)=\rk(t)$ gives $r=t$, and $u\le v$ with $\rk(v)=\rk(u)$ gives $v=u$. The fibre therefore consists of the quadruples $(\sigma,t,t,u)$, in which $\sigma\le t$ and $t\le u$ cross one rank each, so $\sigma\lessdot t\lessdot u$ and $t\in(\sigma,u)$. Each such $t$ conversely gives one, since $\len(\sigma,t)=\len(t,u)=1$ places $(\sigma,t,t,u)$ in $\Ka_{1,1}(\sigma)$ and in \eqref{eq:comma-fibre}. The correspondence inverts $t\mapsto(\sigma,t,t,u)$ and is monotone both ways, the order on $\Ka_{1,1}(\sigma)$ being componentwise. The interval $(\sigma,u)$ carries the single rank $\rho+1$, so it is an antichain.

An antichain of $m$ elements has discrete order complex, with $\widetilde H_0=k^{m-1}$ and $\widetilde H_q=0$ for $q\ge1$; Lemma \ref{lem:defect-reduced} converts this into the two displays, and the description of the defective locus follows since $m\ge2$ is the condition $\widetilde H_0\ne0$.
\end{proof}

\begin{remark}[What the defect depends on]\label{rem:weight-independent}
The window $\Delta_a^{\Pos}$ is cut out of $\ct$ by the condition $\len\le a$, and the index posets $J_{a+b}(\sigma)$ and $\Ka_{a,b}(\sigma)$, the meet functor $\Phi(s,t,r,u)=(s\meet r,u)$, its comma fibres and with them $\operatorname{Def}_\Phi$ are constructed from the order of $\Pos$ alone. A weight $\Wt$ in the sense of Remark \ref{rem:weighted} is a functor \emph{on} $\Delta_a^{\Pos}$ and does not change it. The finality defect of $\Phi$ is therefore an invariant of $(\Pos,\sigma,a,b)$, and Lemma \ref{lem:unsaturated}, Proposition \ref{prop:unit-fibre} and Corollary \ref{cor:minimal-unit} are statements about that invariant.

What a weight does change is the passage from the invariant to a flow statement. Proposition \ref{prop:final} uses that the two coefficient diagrams agree, $E=\Phi^\ast D$, which holds because both are $u\mapsto F(u)$ and $\Phi$ preserves $u$; for a general weight the single window carries $\Wt_{a+b}(x,u)\otimes F(u)$ and the iterate $\Wt_a(s,t)\otimes\Wt_b(r,u)\otimes F(u)$, and a comparison $\Wt_a\otimes\Wt_b\to\Wt_{a+b}$ along $\Phi$ must be supplied as extra data. So the obstruction is intrinsic while its sufficiency for the flow is not, and the two questions separate cleanly.
\end{remark}
\begin{definition}[The defect module]\label{def:defect-module}
Keep the data of Lemma \ref{lem:defect-reduced} and write $q_2^J\colon J\to\Pos$, $q_2^J(y,u)=u$. Put
\[
   S=\{j\in J: H_\ast\operatorname{Def}_\Phi(j)\ne0\},
\]
the locus bounded in Lemma \ref{lem:unsaturated} and computed in Proposition \ref{prop:unit-fibre},
and for $u\in\Pos$ let
\[
   \Theta_u=(q_2^J)^{-1}(u)
   =\{y\in\Pos: y\le\sigma,\ y\le u,\ \len(y,u)\le a+b\},
\]
the identification being $y\mapsto(y,u)$. As $y$ descends in $\Theta_u$ the defining condition $y\le s\meet r$ of the comma fibre $(y,u)\downarrow\Phi$ weakens, so the fibre \emph{grows}, and the structure map
\[
   \operatorname{Def}_\Phi(y,u)\longrightarrow\operatorname{Def}_\Phi(y',u)
   \qquad(y'\le y)
\]
is the map on augmented reduced chains induced by the inclusion of nerves. The order in $\Theta_u$ and the defect-map direction are opposite:
\[
\begin{tikzcd}[row sep=large,column sep=huge]
(y,u) & (y,u)\downarrow\Phi\arrow[d,hook] & \operatorname{Def}_\Phi(y,u)\arrow[d]\\
(y',u)\arrow[u] & (y',u)\downarrow\Phi & \operatorname{Def}_\Phi(y',u)
\end{tikzcd}
\qquad(y'\le y).
\]
Thus $\operatorname{Def}_\Phi|_{\Theta_u}$ is a diagram on $\Theta_u^{op}$, and for each $q$
\[
   \mathsf P_{u,q}
   :=
   H_q\bigl(\operatorname{Def}_\Phi|_{\Theta_u}\bigr)
   \colon\Theta_u^{op}\longrightarrow\mathrm{Vec}_k
\]
is a persistence module on the finite poset $\Theta_u^{op}$, the \emph{defect module of $\Phi$ at $u$}. Its values are finite-dimensional, the nerves involved being finite; no finiteness of the stalks of any sheaf is at issue.
\end{definition}

\begin{definition}[Essential defect]\label{def:essential-defect}
Keep the notation of Definition \ref{def:defect-module}. Say that $\Phi$ has \emph{essential defect} if for every $u$ maximal in $q_2^J(S)$ some defect class survives the passage to the colimit along the growth of the comma fibre, that is, if the derived colimit of the defect module does not vanish:
\[
   \hocolim_{\Theta_u^{op}}\operatorname{Def}_\Phi|_{\Theta_u}
   \;=\;
   \operatorname{Def}_\Phi|_{\Theta_u}\otimes^{\mathbb L}_{k\Theta_u^{op}}k
   \;\ne\;0
   \quad\text{in }\Db(\mathrm{Vec}_k),
\]
equivalently $\operatorname{Tor}^{k\Theta_u^{op}}_\ast\bigl(\operatorname{Def}_\Phi|_{\Theta_u},k\bigr)\ne0$; the covariant factor stands first, as fixed in Definition \ref{def:incidence-algebra}. The condition is vacuous when $S=\varnothing$, and it holds whenever every fibre $\Theta_u$ is a single point (as at a diamond interval with minimal apex, where $y\le\sigma=b$ forces $y=b$), the derived colimit then being the value itself.

Essentiality does not follow from nonvanishing of the defect module (Remark \ref{rem:bar-cancellation} exhibits diagrams with nonzero homology and vanishing derived colimit), so it is imposed and not deduced. It is a condition on one fibre at a time, for which Lemma \ref{lem:essential-certificates} gives two checkable sufficient forms.
\end{definition}

\begin{lemma}[Two certificates for essentiality]\label{lem:essential-certificates}
Keep the notation of Definition \ref{def:defect-module} and fix $u\in q_2^J(S)$. Write
\[
   T_u=\{y\in\Theta_u: H_\ast\operatorname{Def}_\Phi(y,u)\ne0\}
\]
for the support, a nonempty subset of $\Theta_u$. Either of the following implies that the derived colimit at $u$ is nonzero.
\begin{enumerate}
\item[\textup{(i)}] \emph{(Essential bottom class.)} Let $q_0=\min\{q:\mathsf P_{u,q}\ne0\}$. Then
\[
   H_{q_0}\Bigl(\hocolim_{\Theta_u^{op}}\operatorname{Def}_\Phi|_{\Theta_u}\Bigr)
   \;\cong\;
   \colim_{\Theta_u^{op}}\mathsf P_{u,q_0},
\]
so it suffices that the lowest-degree defect module have nonzero colimit; dually, the values being finite-dimensional, that $\mathsf P_{u,q_0}^{\,\ast}$ admit a nonzero global section, i.e.\ a family of functionals on the $H_{q_0}\operatorname{Def}_\Phi(y,u)$, not all zero, compatible with every fibre enlargement.
\item[\textup{(ii)}] \emph{(Terminal support.)} If $T_u$ is down-closed in $\Theta_u$ and has a least element $y_0$, then
\[
   \hocolim_{\Theta_u^{op}}\operatorname{Def}_\Phi|_{\Theta_u}
   \;\simeq\;
   \operatorname{Def}_\Phi(y_0,u)\;\ne\;0 .
\]
\end{enumerate}
\end{lemma}

\begin{proof}
(i) is Lemma \ref{lem:bottom-degree} applied to $I=\Theta_u^{op}$ and $D=\operatorname{Def}_\Phi|_{\Theta_u}$, whose hypothesis is the minimality of $q_0$; the dual reformulation is the isomorphism $\bigl(\colim\mathsf P\bigr)^\ast\cong\lim\mathsf P^\ast$.

(ii) Since $T_u$ is down-closed in $\Theta_u$, the subposet $T_u^{op}$ is up-closed in $\Theta_u^{op}$, and it contains the homological support of $\operatorname{Def}_\Phi|_{\Theta_u}$ by definition, so Lemma \ref{lem:support-restriction}\textup{(ii)} gives $\hocolim_{\Theta_u^{op}}\operatorname{Def}_\Phi|_{\Theta_u}\simeq\operatorname{Bar}(T_u^{op},\operatorname{Def}_\Phi|_{T_u})$. The least element $y_0$ of $T_u$ is terminal in $T_u^{op}$. Filtering by bar degree (Lemma \ref{lem:bar-filtration} with $R=k$), the row $q$ of the $E^1$-page with its $d^1$ is $\operatorname{Bar}(T_u^{op},H_q\operatorname{Def}_\Phi|_{T_u})$, which Lemma \ref{lem:terminal-bar} collapses onto the terminal object:
\[
   E^2_{p,q}\cong
   \begin{cases}
      H_q\operatorname{Def}_\Phi(y_0,u), & p=0,\\
      0, & p\ge1 .
   \end{cases}
\]
Every $d_r$ with $r\ge2$ then has source or target in a vanishing column, so $E^2=E^\infty$; the filtration is bounded by Lemma \ref{lem:bar-filtration}(ii), and $H_\ast\operatorname{Def}_\Phi(y_0,u)\ne0$ because $y_0\in T_u$.
\end{proof}

A third case is decided on the order of $\Theta_u$ rather than on a colimit. If every element of $T_u$ is minimal in $\Theta_u$, then $T_u$ is down-closed in $\Theta_u$ and is an antichain, so $T_u^{op}$ is up-closed in $\Theta_u^{op}$ and carries the homological support; Lemma \ref{lem:support-restriction}\textup{(ii)} then collapses the derived colimit onto $\operatorname{Bar}(T_u^{op},\operatorname{Def}_\Phi|_{T_u})$, which an antichain reduces to $\bigoplus_{y\in T_u}\operatorname{Def}_\Phi(y,u)$ in bar degree $0$, nonzero because $T_u$ is. This is incomparable to Lemma \ref{lem:essential-certificates}\textup{(ii)}, which asks $T_u$ for a least element, and an antichain of more than one element has none.

Testing against a representable sheaf collapses the global bar complex onto a single fibre of $q_2^J$, and uses no essentiality.

\begin{lemma}[The local defect complex]\label{lem:local-defect}
Keep the data of Lemma \ref{lem:defect-reduced} and the notation of Definition \ref{def:defect-module}, and let $u\in\Pos$ be such that
\[
   H_\ast\operatorname{Def}_\Phi(y,v)=0
   \qquad\text{for every }(y,v)\in J\text{ with }v>u .
\]
Then the representable sheaf $k_{\st(u)}$ computes the derived colimit over the fibre $\Theta_u$ alone:
\[
   B\bigl(\operatorname{Def}_\Phi,J,(q_2^J)^\ast k_{\st(u)}\bigr)
   \;\simeq\;
   \hocolim_{\Theta_u^{op}}\operatorname{Def}_\Phi|_{\Theta_u}
   \;=\;
   \operatorname{Def}_\Phi|_{\Theta_u}\otimes^{\mathbb L}_{k\Theta_u^{op}}k ,
\]
by maps of bar-degree filtered complexes inducing isomorphisms of the associated spectral sequences from $E^1$ on. That spectral sequence is
\[
   E^2_{p,q}\;\cong\;\operatorname{Tor}^{k\Theta_u^{op}}_p\bigl(\mathsf P_{u,q},k\bigr)
   \;\Longrightarrow\;
   \operatorname{Tor}^{k\Theta_u^{op}}_{p+q}\bigl(\operatorname{Def}_\Phi|_{\Theta_u},k\bigr).
\]
\end{lemma}

\begin{proof}
Recall the up-closed $W_u=\{(y,v)\in J:u\le v\}$ of Proposition \ref{prop:representable-test}, so that $\Theta_u\subseteq W_u\subseteq J$. Then $\Theta_u$ is down-closed in $W_u$: if $(y',v')\le(y,u)$ in $J$ and $(y',v')\in W_u$, then $v'\le u$ by the first condition and $u\le v'$ by the second, so $v'=u$. This much is unconditional. The hypothesis enters only here, as the vanishing of $H_\ast\operatorname{Def}_\Phi$ on $W_u\setminus\Theta_u$, whose objects have $q_2^J$-image strictly above $u$.

Now $(q_2^J)^\ast k_{\st(u)}=k_{W_u}$ is the constant diagram $k$ on $W_u$. The summand $\operatorname{Def}_\Phi(j_p)\otimes D(j_0)$ of Lemma \ref{lem:bar-filtration} is therefore nonzero only for $j_0\in W_u$, in which case the whole chain $j_0<\cdots<j_p$ lies in $W_u$; and $D$ is constant there, so the $d_0$ face is a plain deletion. Reversing the order of the bar chains identifies
\[
   B\bigl(\operatorname{Def}_\Phi,J,(q_2^J)^\ast k_{\st(u)}\bigr)
   \;\cong\;
   \operatorname{Bar}\bigl(W_u^{op},\operatorname{Def}_\Phi|_{W_u}\bigr),
\]
the bar sign being corrected degreewise by $(-1)^{\binom{p+1}{2}}$. Indeed reversal carries the face $d_i$ of Definition \ref{def:two-sided-bar} to the face $d'_{p-i}$ of Definition \ref{def:bar-hocolim} (in particular the coefficient-side face $d_0$ to the plain deletion $d'_p$ and the $\operatorname{Def}_\Phi$-side face $d_p$ to the structure-map face $d'_0$), so it converts $\sum_i(-1)^id_i$ into $(-1)^p\sum_m(-1)^md'_m$. A degreewise sign $\epsilon_p$ makes the reversal a chain map if and only if $\epsilon_p=(-1)^p\epsilon_{p-1}$, that is $\epsilon_p=(-1)^{p(p+1)/2}$. The internal-degree signs need no adjustment: $D$ is concentrated in internal degree zero, so $\delta_B$ carries the sign $(-1)^{|r|}$ on the left-hand side, which is the sign $(-1)^m$ of Definition \ref{def:bar-hocolim} on the right.

The subposet $\Theta_u^{op}$ is up-closed in $W_u^{op}$, and by the hypothesis the homological support of $\operatorname{Def}_\Phi|_{W_u}$ is contained in it, so Lemma \ref{lem:support-restriction}\textup{(ii)} gives the displayed equivalence; the identification of the homotopy colimit with the derived tensor product is Lemma \ref{lem:bar-derived-colim}.

For the filtrations: order reversal carries a $p$-chain to a $p$-chain, so the first identification preserves bar degree, its correcting sign being degreewise. The inclusion of Lemma \ref{lem:support-restriction} is by \textup{(i)} the inclusion of the summands indexed by chains lying in $\Theta_u^{op}$, and so is filtered, and the proof of \textup{(ii)} computes the $E^1$-page of the quotient to be $0$; so the two $E^1$-pages agree, and with them every later page and differential. Finally, writing $\operatorname{Bar}(\Theta_u^{op},\operatorname{Def}_\Phi|_{\Theta_u})=B(k,\Theta_u^{op},\operatorname{Def}_\Phi|_{\Theta_u})$ and applying Lemma \ref{lem:bar-filtration}(iii) with $R=k$, row $q$ of the $E^1$-page carries its $d^1$ as $\operatorname{Bar}(\Theta_u^{op},\mathsf P_{u,q})$, whose homology Lemma \ref{lem:incidence-trivial}\textup{(iii)} identifies with $\operatorname{Tor}^{k\Theta_u^{op}}_p(\mathsf P_{u,q},k)$.
\end{proof}

\begin{lemma}[Defect-bar detection]\label{lem:defect-detection}
Keep the data and hypotheses of Lemma \ref{lem:defect-reduced} and the notation of Definition \ref{def:defect-module}. Consider the complexes
\[
   B(\operatorname{Def}_\Phi,J,(q_2^J)^\ast F),
   \qquad F\in\Shv(\Pos;k).
\]
\begin{enumerate}
\item[\textup{(i)}] If $\operatorname{Def}_\Phi(j)$ is acyclic for every $j\in J$, then $B(\operatorname{Def}_\Phi,J,(q_2^J)^\ast F)$ is acyclic for every $F$.
\item[\textup{(ii)}] \emph{(Detection under essentiality.)} Suppose $\Phi$ has essential defect. If $B(\operatorname{Def}_\Phi,J,(q_2^J)^\ast F)$ is acyclic for every $F\in\Shv(\Pos;k)$, then $\operatorname{Def}_\Phi(j)$ is acyclic for every $j\in J$.
\end{enumerate}
\end{lemma}

\begin{proof}
For a coefficient diagram $D=(q_2^J)^\ast F$ put $A=B(\operatorname{Def}_\Phi,J,D)$.
Since $F$ is an ordinary sheaf, $D$ is concentrated in internal degree zero, so
Lemma \ref{lem:bar-filtration}(iii) gives its $E^1$-page in the reduced form there,
with $R=\operatorname{Def}_\Phi$, converging to $H_\ast(A)$ by Lemma \ref{lem:bar-filtration}(iv).

\emph{Proof of \textup{(i)}.} If $\operatorname{Def}_\Phi$ is objectwise acyclic
then every $H_q\operatorname{Def}_\Phi(j_p)$ vanishes, so $E^1=0$ and $A$ is
acyclic; this is the last sentence of Lemma \ref{lem:bar-filtration}. As $F$ was
arbitrary, \textup{(i)} follows.

\emph{Proof of \textup{(ii)}.} Suppose, for contradiction, that $S$ is nonempty. Since $J$ is finite, choose
\[
   u_0\ \text{maximal in}\ q_2^J(S)=\{q_2^J(j):j\in S\}\subseteq\Pos .
\]
Then $H_\ast\operatorname{Def}_\Phi(y,v)=0$ for every $(y,v)\in J$ with $v>u_0$, since otherwise $v$ would lie in $q_2^J(S)$ strictly above $u_0$. So Lemma \ref{lem:local-defect} applies at $u_0$, and testing against the representable sheaf $k_{\st(u_0)}$ gives
\[
   B\bigl(\operatorname{Def}_\Phi,J,(q_2^J)^\ast k_{\st(u_0)}\bigr)
   \;\simeq\;
   \operatorname{Def}_\Phi|_{\Theta_{u_0}}\otimes^{\mathbb L}_{k\Theta_{u_0}^{op}}k .
\]
The right-hand side is nonzero: that is the essentiality hypothesis at $u_0$, available because $u_0$ was chosen maximal in $q_2^J(S)$, and it is the only place the hypothesis is used. This contradicts the assumed acyclicity of the left-hand side, so $S=\varnothing$. Lemma \ref{lem:essential-certificates} supplies the two sufficient forms in which the hypothesis is checked in practice: a lowest-degree defect class with nonzero colimit, or a support that is down-closed with a least element.
\end{proof}

\begin{remark}[Scope of the detection lemma]\label{rem:detection-scope}
What is \emph{not} proved is the unconditional converse, that is Lemma \ref{lem:defect-detection}\textup{(ii)} with the essentiality hypothesis deleted; that statement is equivalent to the open direction of Conjecture \ref{conj:sharp} (see the reduction below) and remains open. Part \textup{(ii)} and Corollary \ref{cor:sharp-essential} hold only under the essentiality hypothesis of Definition \ref{def:essential-defect}.

Essentiality cannot be weakened to objectwise nonvanishing of the defect module, nor to nonvanishing together with injectivity of the structure maps on homology: Remark \ref{rem:bar-cancellation} gives diagrams on small posets defeating both, all of them concentrated in a single homological degree, so that separating degrees is of no help either. What is decisive in Lemma \ref{lem:essential-certificates}(i) is not degree separation but degree \emph{position}, the bottom row of the bar filtration being unreachable by differentials.

Those diagrams are arbitrary diagrams on a finite index poset, and none is realised here as some $\operatorname{Def}_\Phi|_{\Theta_u}$. Nor does any meet functor treated below have inessential defect: at a minimal $\sigma$ every fibre $\Theta_u$ is a singleton; where it is not, Corollary \ref{cor:unit-window-essential} disposes of every unit first window (which covers the one such case computed here, Example \ref{ex:essential-nonsingleton}, where both certificates of Lemma \ref{lem:essential-certificates} apply as well). So the range in which Conjecture \ref{conj:sharp} remains open is not exhibited to be nonempty, and it is narrower than the essentiality hypothesis alone suggests: by Corollary \ref{cor:unit-window-essential} it requires $a\ge2$, and whether a defect module can be inessential at all is itself part of the open problem. The closing paragraph of Remark \ref{rem:bar-cancellation} says what such an instance would have to look like.

What makes detection delicate is that the available test objects are only the pullbacks $(q_2^J)^\ast F$, constant in the first coordinate of $J$, so the representable $J$-diagrams that would isolate $\operatorname{Def}_\Phi$ at a single object are not among them. Over the incidence algebra of $J$ (Definition \ref{def:incidence-algebra}), $B(\operatorname{Def}_\Phi,J,(q_2^J)^\ast F)$ computes the derived form of the pairing $(q_2^J)^\ast F\otimes_{kJ}\operatorname{Def}_\Phi$ of Definition \ref{def:two-sided-bar}, the resolution being the diagram-level bar resolution of Remark \ref{rem:bar-vs-projective}; so the hypothesis of \textup{(ii)} says that $\operatorname{Def}_\Phi$ is Tor-orthogonal to the subcategory of $\Db(kJ)$ generated by the pullbacks $k_{W_u}$ of the representables, and the open question is whether that orthogonality forces acyclicity. Proposition \ref{prop:representable-test} carries the same delimitation without the incidence algebra: what the test sheaves impose is a $k$-homology equivalence over each $W_u$ and its preimage, what finality asks is one over each $\st_J(j)$ and its preimage, and each $W_u$ is the union of the $\st_J(j)$ at its minimal elements.
\end{remark}

\paragraph*{Reduction of Conjecture \ref{conj:sharp} to a detection question.}
For the converse, fix $\Pos,\sigma,a,b$ and assume $\Phi=\Phi_{\Pos,\sigma,a,b}$ is total. Put $J=J_{a+b}(\sigma)$ and $K=\Ka_{a,b}(\sigma)$, and define $R_\Phi(j)$ and $\operatorname{Def}_\Phi(j)$ as in Lemma \ref{lem:defect-reduced}. By Lemma \ref{lem:defect-reduced}, $\Phi$ is $k$-homologically final if and only if every $\operatorname{Def}_\Phi(j)$ is acyclic.

Let $D=(q_2^J)^\ast F$ for $F\in\Shv(\Pos;k)$, and let $\varepsilon\colon R_\Phi\to k$ be the augmentation. By the change-of-base Lemma \ref{lem:change-of-base}, the comparison map $\Phi_\#$ is a quasi-isomorphism if and only if
\[
   \varepsilon_\ast\colon B(R_\Phi,J,D)\longrightarrow B(k,J,D)=\operatorname{Bar}(J,D)
\]
is a quasi-isomorphism. Applying Lemma \ref{lem:bar-cone} to $\varepsilon$ and to the defining shifted cone $\operatorname{Def}_\Phi=\operatorname{Cone}(\varepsilon)[-1]$ of Lemma \ref{lem:defect-reduced} gives an equality of complexes
\[
   \operatorname{Cone}(\varepsilon_\ast)[-1]
   \;=\;
   B(\operatorname{Def}_\Phi,J,D),
\]
with the two-sided bar differential of Definition \ref{def:two-sided-bar} on both sides. What both lemmas ask of the first variable is met here: $\operatorname{Def}_\Phi$ is a contravariant $J$-diagram in $\Chb(\mathrm{Vec}_k)$, its structure maps being induced by inclusions of nerves and so chain maps, which is what Lemma \ref{lem:bar-total-differential} needs for $d_{\mathrm{tot}}^2=0$. Therefore, if $\gamma_\sigma$ is a quasi-isomorphism for every $F$, then
\[
   B(\operatorname{Def}_\Phi,J,(q_2^J)^\ast F)
\]
is acyclic for every sheaf $F$. Lemma \ref{lem:defect-detection}(ii) then makes $\operatorname{Def}_\Phi$ objectwise acyclic and $\Phi$ $k$-homologically final, provided the essentiality hypothesis of Definition \ref{def:essential-defect} holds. This gives the following conditional form of the open direction.

\begin{corollary}[Sharpness under essential defect]\label{cor:sharp-essential}
Let $\Pos,\sigma,a,b$ be such that $\Phi=\Phi_{\Pos,\sigma,a,b}$ is total and has essential defect. Then $\gamma_\sigma$ is a quasi-isomorphism for every $F\in\Shv(\Pos;k)$ if and only if $\Phi$ is $k$-homologically final, so Conjecture \ref{conj:sharp} holds at every such $(\Pos,\sigma,a,b)$. Where $\sigma$ is minimal the hypothesis is automatic, $y\le\sigma$ forcing $y=\sigma$, so that each $\Theta_u$ is empty or $\{\sigma\}$; there the conclusion is Theorem \ref{thm:minimal-sharp}, which needs neither the defect module nor essentiality. What the corollary adds is the range in which $\sigma$ is not minimal.
\end{corollary}

\begin{proof}
``If'' is Proposition \ref{prop:final}. For ``only if'', the display above gives acyclicity of $B(\operatorname{Def}_\Phi,J,(q_2^J)^\ast F)$ for every $F$, so Lemma \ref{lem:defect-detection}(ii) applies and Lemma \ref{lem:defect-reduced} converts objectwise acyclicity of $\operatorname{Def}_\Phi$ into $k$-homological finality. For the minimal case, $\Theta_u=\{y:y\le\sigma,\ y\le u,\ \len(y,u)\le a+b\}$ and minimality of $\sigma$ leaves only $y=\sigma$, which is the singleton-fibre clause of Definition \ref{def:essential-defect}.
\end{proof}

Every step of that reduction is an equivalence, so the \emph{unconditional} detection statement is \emph{equivalent} to the open direction of Conjecture \ref{conj:sharp} and not merely sufficient for it: a counterexample would refute the conjecture. A unit first window removes the essentiality hypothesis as well, the defective locus there meeting each fibre of $q_2^J$ in an antichain.

\begin{corollary}[A unit first window makes essentiality automatic]\label{cor:unit-window-essential}
If $\Phi$ is total at $(\sigma;1,b)$, then $\Phi$ has essential defect. Conjecture \ref{conj:sharp} then holds unconditionally at every $(\Pos,\sigma,1,b)$ at which $\Phi$ is total, whether or not $\sigma$ is minimal.
\end{corollary}

\begin{proof}
If $S=\varnothing$ the condition is vacuous. Otherwise let $u$ be maximal in $q_2^J(S)$ and let $y\in T_u$. Lemma \ref{lem:unsaturated} gives $\len(y,u)>b$, while $(y,u)\in J$ gives $\len(y,u)\le a+b=b+1$, so $\len(y,u)=b+1$. If $y'<y$ with $y'\in\Theta_u$, additivity of $\len$ gives $\len(y',u)=\len(y',y)+\len(y,u)\ge b+2$, which $\Theta_u$ forbids; so every element of $T_u$ is minimal in $\Theta_u$. By the case recorded after Lemma \ref{lem:essential-certificates},
\[
   H_\ast\Bigl(\hocolim_{\Theta_u^{op}}\operatorname{Def}_\Phi|_{\Theta_u}\Bigr)
   \;\cong\;
   \bigoplus_{y\in\Theta_u}H_\ast\operatorname{Def}_\Phi(y,u),
\]
which is nonzero because $u\in q_2^J(S)$. Now apply Corollary \ref{cor:sharp-essential}.
\end{proof}

\begin{corollary}[Finality at a minimal apex with unit windows]\label{cor:minimal-unit}
Let $\Pos$ be a finite graded poset, let $\sigma$ be minimal in $\Pos$, and suppose $\Phi=\Phi_{\sigma,1,1}$ is total. The following are equivalent.
\begin{enumerate}
\item[\textup{(i)}] $\gamma_\sigma$ is a quasi-isomorphism for every $F\in\Shv(\Pos;k)$.
\item[\textup{(ii)}] $\Phi$ is $k$-homologically final.
\item[\textup{(iii)}] Every $u\in\Pos$ with $\len(\sigma,u)=2$ has exactly one element strictly between $\sigma$ and $u$.
\end{enumerate}
\end{corollary}

\begin{proof}
$\textup{(i)}\Leftrightarrow\textup{(ii)}$ is Theorem \ref{thm:minimal-sharp}, $\sigma$ being minimal. For $\textup{(ii)}\Leftrightarrow\textup{(iii)}$, finality asks that every comma fibre be nonempty with vanishing reduced homology. By Proposition \ref{prop:unit-fibre}, a fibre with $\len(\sigma,u)\le1$ has the least element $(\sigma,\sigma,\sigma,u)$, so it is nonempty and acyclic whatever $\Pos$ does; and a fibre with $\len(\sigma,u)=2$ is the open interval $(\sigma,u)$, an antichain, which is nonempty because $\Pos$ is graded and whose reduced homology vanishes if and only if it is a single point.
\end{proof}

So at a minimal apex the comparison $\gamma_\sigma$ is a quasi-isomorphism on the posets whose length-two intervals there have a single interior element and no choice, and where a choice appears the strong flow itself fails: a defective fibre at unit windows sits over a $u$ with $\len(\sigma,u)=2=a+b$, so Proposition \ref{prop:sat-nogo} reaches every one of them, and Proposition \ref{prop:diamond-nogo} counts the interior elements. Proposition \ref{prop:chain} is the case of one, and the diamond property that a finite regular cell complex enjoys by Bj\"orner's criterion is the feature that defeats the flow on its face poset. Condition \textup{(iii)} is vacuous when no $u$ satisfies $\len(\sigma,u)=2$, so the corollary is about height and branching together and not either alone. Its totality hypothesis cannot be dropped: Example \ref{ex:closed-edge} satisfies \textup{(iii)} vacuously and carries a strict flow, but $\Phi$ is not total there, so its flow comes from the terminal-object criterion of that example and not from this corollary.

\begin{example}[Simplicial complexes, read at the empty face]\label{ex:simplicial-apex}
Let $\Sigma$ be a finite simplicial complex and $\Pos=\widehat\Pos(\Sigma)$ its augmented face poset, graded by cardinality, so that $\Phi$ is total at every $(\sigma;a,b)$ by Proposition \ref{prop:total-families}\textup{(i)}. The apex $\hat0$ is minimal and $\len(\hat0,u)=\lvert u\rvert$, so the $u$ with $\len(\hat0,u)=2$ are the edges of $\Sigma$, each carrying its two vertices strictly between. By Corollary \ref{cor:minimal-unit} the comparison $\gamma_{\hat0}$ at $a=b=1$ is a quasi-isomorphism for every $F$ if and only if $\Sigma$ has no edge. Over an edge the comma fibre is the two-vertex antichain $(\hat0,u)$ and $\len(\hat0,u)=2=a+b$, so Proposition \ref{prop:sat-nogo} reaches every one of them: on a simplicial complex of positive dimension the meet functor is total at the empty face, and what fails there is the law $\LC_1\LC_1\simeq\LC_2$ and not only the comparison.
\end{example}

What is settled by the minimal case is where the window-comma poset collapses onto the near-window, $J_{a+b}(\sigma)=\{(\sigma,u):u\in\st_{a+b}(\sigma)\}\cong\st_{a+b}(\sigma)$ (Theorem \ref{thm:minimal-sharp}, and Corollary \ref{cor:sky-stalk}). Where $\sigma$ is not minimal a fibre $\Theta_u$ may have several elements, and essentiality is then a genuine condition on the defect module; but only once $a\ge2$, by Corollary \ref{cor:unit-window-essential}.

\begin{example}[Essentiality at a diamond, computed]\label{ex:diamond-essential}
Let $\Pos=[b,t]$ be a diamond interval read as a poset in its own right, so $b$ is least, $t$ is greatest, and $m_1,m_2$ are the two elements between; it is graded, and not a face poset (Appendix \ref{ssec:face-posets}). Take $\sigma=b$ and both window scales equal to $1$, as in Proposition \ref{prop:diamond-nogo}; $\Phi$ is total at $(b;1,1)$ because $b$ is least (Remark \ref{rem:diamond-totality}).

Since $y\le\sigma=b$ forces $y=b$,
\[
   J_2(b)=\{(b,b),(b,m_1),(b,m_2),(b,t)\},
\]
and every fibre $\Theta_u$ is $\{b\}$ or empty. The comma fibre over $(b,t)$ is computed in the proof of Proposition \ref{prop:diamond-nogo}: the condition $b\le s\meet r$ is vacuous, so it consists of the quadruples with $u_\kappa=t$, namely $\kappa_1=(b,m_1,m_1,t)$ and $\kappa_2=(b,m_2,m_2,t)$, which are incomparable, so
\[
   H_\ast\operatorname{Def}_\Phi(b,t)
   =\widetilde H_\ast\bigl(\{\kappa_1,\kappa_2\};k\bigr)
   =\begin{cases}k,&\ast=0,\\0,&\text{otherwise,}\end{cases}
\]
so $(b,t)\in S$ and $t\in q_2^J(S)$. As $t$ is the greatest element of $\Pos$ it is maximal in $q_2^J(S)$, and $\Theta_t=\{b\}$. The defect module $\mathsf P_{t,0}$ is therefore the one-object persistence module with value $k$, its derived colimit is that value, and $\Phi$ has essential defect (automatically, the fibre being a singleton). Corollary \ref{cor:sharp-essential} then applies, and its conclusion is the one already known here by direct computation: $\Phi$ is not $k$-homologically final and $\gamma_b$ is not a quasi-isomorphism, the test sheaf produced by the proof of Lemma \ref{lem:defect-detection}\textup{(ii)} being $k_{\st(t)}=\sky_t$, which is the witness of Proposition \ref{prop:diamond-nogo}.

Here $\sigma=b$ is minimal and $a=b=1$, so Proposition \ref{prop:unit-fibre} concentrates $\operatorname{Def}_\Phi$ in homological degree $0$; by Lemma \ref{lem:bar-filtration}(iii) the $E^1$-page of $A=B\bigl(\operatorname{Def}_\Phi,J_2(b),(q_2^J)^\ast F\bigr)$ is therefore carried by the row $q=0$, and $H_\ast(A)$ is the homology of that row. Since $H_0\operatorname{Def}_\Phi$ is supported at $(b,t)$ alone, only the chains of $J_2(b)$ ending at $(b,t)$ contribute, and under the identification $J_2(b)\cong\Pos$, $(b,u)\mapsto u$, these are the chains of $\Pos$ with greatest element $t$. The row is therefore
\[
\begin{array}{c|ccc}
   p & 0 & 1 & 2\\\hline
   E^1_{p,0} & F(t) & F(b)\oplus F(m_1)\oplus F(m_2) & F(b)\oplus F(b)
\end{array}
\]
and $E^1_{p,0}=0$ for $p\ge3$, the longest chain of $\Pos$ ending at $t$ having three elements. Its differential is the simplicial one, the face deleting the greatest element of a chain contributing nothing because $H_0\operatorname{Def}_\Phi$ vanishes off $(b,t)$. So the row is the chain complex of the order complex of $\{x\in\Pos:x<t\}$, with $F$ read at the least element of each chain, augmented to $F(t)$; the defect is seen if and only if that augmentation fails to be a quasi-isomorphism.

For $F=\sky_t=k_{\st(t)}$ the row is $(k,0,0)$, so $H_0(A)\cong k$ and the defect is seen. For the constant sheaf $F=k_{\Pos}$ the row is $(k,k^3,k^2)$, the augmented chain complex of the order complex of $\{b,m_1,m_2\}$, which is a cone with apex $b$; it is acyclic and the defect is invisible. So Lemma \ref{lem:defect-detection}\textup{(ii)} cannot argue from acyclicity of $A$ for one chosen $F$, and tests against the representable $k_{\st(u_0)}$ instead.

Raising the apex destroys the defect. Take $\sigma=m_1$ instead, with the same two window scales. Now $(m_1,t)\in J_1(m_1)$, because $\len(m_1,t)=1$, so $\Ka_{1,1}(m_1)$ contains
\[
   \kappa_\top=(m_1,t,t,t),
\]
which lies in the comma fibre over $(b,t)$ and dominates every quadruple in it, all of whose coordinates are $\le(m_1,t,t,t)$. The fibre is a cone and therefore $k$-acyclic, and the same holds at every object of $J_2(m_1)$: $S=\varnothing$, $\Phi$ is $k$-homologically final at $(m_1;1,1)$, and $\gamma_{m_1}$ is a quasi-isomorphism. What was unavailable at the apex $b$ is $\kappa_\top$: there $s\le\sigma=b$ forces $s=b$, and $(b,t)\notin J_1(b)$ because $\len(b,t)=2>1$. The defect at the diamond is therefore an interaction between length and branching and not a property of the interval alone: lengthening the first window instead of raising the apex removes it just as well, since $t$ is greatest and $\len(b,t)=2$, so that Corollary \ref{cor:saturated}\textup{(ii)} applies at the windows $(2,1)$.
\end{example}

\begin{example}[An essential defect on a two-element fibre of $q_2^J$]\label{ex:essential-nonsingleton}
Let $\Pos$ be the graded lattice of rank $3$ obtained by glueing two chains of length $3$ along their endpoints,
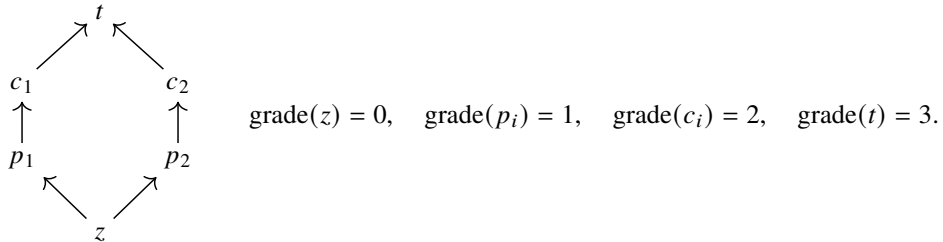
\[
\begin{tikzcd}[row sep=1.6em,column sep=1.6em,cramped]
& t & \\
c_1\arrow[ur] & & c_2\arrow[ul]\\
p_1\arrow[u] & & p_2\arrow[u]\\
& z\arrow[ul]\arrow[ur] &
\end{tikzcd}
\qquad
\operatorname{grade}(z)=0,\quad
\operatorname{grade}(p_i)=1,\quad
\operatorname{grade}(c_i)=2,\quad
\operatorname{grade}(t)=3 .
\]
Every length-two interval of $\Pos$, namely $[z,c_1]$, $[z,c_2]$, $[p_1,t]$, $[p_2,t]$, has one element in its interior, so $\Pos$ carries no branching length-two interval and Proposition \ref{prop:diamond-nogo} does not apply. Likewise, it is graded and not a face poset (Appendix \ref{ssec:face-posets}). It is a lattice, so every meet $s\meet r$ exists, and $\len\le3$ throughout, so at $a+b\ge 3$ the remaining condition $(s\meet r,u)\in J_{a+b}(\sigma)$ is automatic and $\Phi$ is total (the quadruple $(p_1,t,p_2,t)\in \Ka_{1,1}(t)$ shows that it is not total at every $(\sigma;a,b)$).

Take $\sigma=p_1$ and $(a,b)=(1,2)$. Then
\[
   J_3(p_1)=\{(z,y):y\in\Pos\}\cup\{(p_1,p_1),(p_1,c_1),(p_1,t)\},
\]
and the fibre over $t$ is $\Theta_t=\{z,p_1\}.$
The fibre over $(z,t)$ is all of $\Ka_{1,2}(p_1)$ lying over $t$, the conditions $z\le s$ and $z\le r$ being vacuous:
\[
   (z,t)\downarrow\Phi=
   \{(z,p_1,p_1,t),\ (z,p_2,p_2,t),\ (p_1,p_1,p_1,t),\ (p_1,c_1,p_1,t),\ (p_1,c_1,c_1,t)\}.
\]
Because $s\le p_1$ and $\len(s,t')\le1$, the second coordinate of a quadruple here lies in $\{z,p_1,p_2,c_1\}$, and $t'=z$ would force $r=z$ with $\len(z,t)=3>2$. Only $(z,p_2,p_2,t)$ has $t'=p_2$, which is comparable to neither $p_1$ nor $c_1$, so it is isolated and the other four form a chain:
\[
\begin{tikzcd}[row sep=1.4em,column sep=1.2em,cramped]
& (p_1,c_1,c_1,t) & & \\
& (p_1,c_1,p_1,t)\arrow[u] & & (z,p_2,p_2,t)\\
& (p_1,p_1,p_1,t)\arrow[u] & & \\
& (z,p_1,p_1,t)\arrow[u] & &
\end{tikzcd}
\]
The two branches meet in $\Pos$ only at $t$, which the first window does not reach from below $p_1$. So $H_\ast\operatorname{Def}_\Phi(z,t)=\widetilde H_\ast$ of the picture is $k$ in degree $0$ and zero otherwise. A quadruple travelling through $p_2$ has $s\le p_1\meet p_2=z$, so imposing $p_1\le s\meet r$ removes it together with the other object of first coordinate $z$, and
\[
   (p_1,t)\downarrow\Phi=
   \{(p_1,p_1,p_1,t),\ (p_1,c_1,p_1,t),\ (p_1,c_1,c_1,t)\}
\]
is the top of that chain, with greatest element $(p_1,c_1,c_1,t)$, so $\operatorname{Def}_\Phi(p_1,t)$ is acyclic.

Then $S\ni(z,t)$, and $t$, being greatest in $\Pos$, is maximal in $q_2^J(S)$. The defect module on $\Theta_t^{op}$, whose order is $p_1<z$, is
\[
   \mathsf P_{t,0}\colon\quad p_1\mapsto0\ \longrightarrow\ z\mapsto k ,
\]
concentrated at the \emph{later} object of $\Theta_t^{op}$, the second branch entering only at the bottom of $\Theta_t$ with nothing below $z$ to reconnect it. This is the orientation that survives: by Lemma \ref{lem:essential-certificates}\textup{(ii)} the support $T_t=\{z\}$ is down-closed in $\Theta_t$ with least element $z$, so
\[
   \hocolim_{\Theta_t^{op}}\operatorname{Def}_\Phi|_{\Theta_t}\simeq\operatorname{Def}_\Phi(z,t)\ne0 ,
\]
and certificate \textup{(i)} gives the same conclusion, the bottom degree being $q_0=0$ with colimit $k$; Example \ref{ex:defect-resolution} runs the same computation through the explicit resolution of $k$ over $k\Theta_t^{op}$. Reversing the values (a defect at $p_1$ filled in at $z$) would be the first, acyclic, diagram of Remark \ref{rem:bar-cancellation}; that is the configuration essentiality excludes, and it is not the one occurring here.

So $\Phi$ has essential defect at $(\Pos,p_1,1,2)$, and Corollary \ref{cor:sharp-essential} applies although $\sigma$ is not minimal and $\Theta_t$ is not a singleton: Conjecture \ref{conj:sharp} holds at this $(\Pos,\sigma,a,b)$, and since $S\ne\varnothing$ its content there is that $\gamma_{p_1}$ fails to be a quasi-isomorphism. The failure is visible directly, on $F=\sky_t$: the nonzero locus in $J_3(p_1)$ is the chain $(z,t)<(p_1,t)$, whence $H_0\bigl((\LC_3\sky_t)(p_1)\bigr)\cong k$, while the nonzero locus in $\Ka_{1,2}(p_1)$ is the five-object poset displayed above, with two connected components, whence $H_0\bigl((\LC_1\LC_2\sky_t)(p_1)\bigr)\cong k^2$. The dimensions $2$ and $1$ are those of Proposition \ref{prop:diamond-nogo}, produced here with no diamond interval present, by a total meet functor, at a non-minimal apex.
\end{example}

Since the comma fibre $(y,u)\downarrow\Phi$ grows as $y$ descends, one expects a defect to be filled in below, and a class filled in below dies in the colimit: the two-element cancellation of Remark \ref{rem:bar-cancellation}, an inessential defect. Against that stand Examples \ref{ex:diamond-essential} and \ref{ex:essential-nonsingleton}, where the defect survives, in the second on a two-element fibre of $q_2^J$, so not merely because there was nothing below to fill it in. By Lemma \ref{lem:essential-certificates}, inessentiality at $u$ requires both that the lowest-degree defect module have vanishing colimit and that the support $T_u$ fail to be down-closed in $\Theta_u$ with a least element. Remark \ref{rem:bar-cancellation} shows that an abstract diagram can meet both (by a defect filled in at the larger fibre below it, or, on a support whose order complex is not simply connected, by a nontrivial monodromy), so what remains is whether a meet-defect module $\mathsf P_{u,q}$ can. Its structure maps are induced by inclusions of comma fibres inside the fixed poset $\Ka_{a,b}(\sigma)$, and the fibres are closed under intersection, $(y,u)\!\downarrow\!\Phi\ \cap\ (y',u)\!\downarrow\!\Phi=\{\kappa:y,y'\le s\meet r,\ u\le u_\kappa\}$; whether that forces trivial monodromy is a finite question for each $(\Pos,\sigma,a,b)$.

\begin{example}[Interval-chain posets: not sufficient]\label{ex:path-zigzag}
Let $\Pos$ be the face poset of the path
\[
v_0<e_{01}>v_1<e_{12}>v_2 .
\]
Every closed interval of $\Pos$ is a chain. Nevertheless the strong flow fails. Take the projective $F=k_{\st(v_0)}$ and evaluate at $\sigma=v_2$; the poset, each cell carrying the stalk of $F$ there in parentheses, is
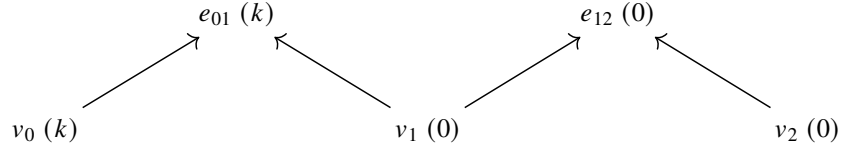
\[
\begin{tikzcd}[row sep=large,column sep=large]
& e_{01}\ (k) & & e_{12}\ (0) & \\
v_0\ (k)\arrow[ur] & & v_1\ (0)\arrow[ul]\arrow[ur] & & v_2\ (0)\arrow[ul]
\end{tikzcd}
\] The single window $J_2(v_2)$ has only the objects $(v_2,v_2)$ and $(v_2,e_{12})$, and $F$ is zero at both $v_2$ and $e_{12}$, so
\[
(\LC_2F)(v_2)=0.
\]
For the iterate, $\Ka_{1,1}(v_2)$ contains the object
\[
q=(v_2,e_{12},v_1,e_{01}),
\]
whose coefficient is $F(e_{01})=k$. This object has no outgoing morphism in the nonzero locus, and the only incoming object has zero coefficient, so it contributes one copy of $k$ to $H_0$:
\[
H_0\bigl((\LC_1\LC_1F)(v_2)\bigr)\cong k .
\]
Thus $\LC_1\LC_1F\not\simeq\LC_2F$. The obstruction is not a diamond; it is the missing meet $v_2\meet v_1$ in the quadruple $(v_2,e_{12},v_1,v_1)$, so the meet assignment is not total.
\end{example}

\begin{definition}[Rooted forest]\label{def:rooted}
$\Pos$ is a \emph{rooted forest} if $\cl(x)$ is a chain for every $x\in\Pos$, and a \emph{rooted tree} if in addition it is connected. In the ancestry reading every element covers at most one element, and $x\le y$ says that $x$ lies on the unique saturated chain from a minimal element up to $y$. Such a poset is automatically graded, $\rk(x)$ being the length of the down-set of $x$, so the standing hypothesis is not strengthened here.
\end{definition}

\begin{proposition}[Rooted forests: the flow at every apex and every pair of windows]\label{prop:rooted}
Let $\Pos$ be a finite graded poset.
\begin{enumerate}
\item[\textup{(i)}] If $\Pos$ is a rooted forest, then for every $\sigma\in\Pos$ and all $a,b\ge0$ the meet assignment $\Phi_{\sigma,a,b}$ is total and the resulting meet functor is homotopy final. Then $\gamma_\sigma$ is a quasi-isomorphism for every $F$, and $\Pos$ is tame in the sense of Definition \ref{def:tame}.
\item[\textup{(ii)}] $\Pos$ is a rooted forest if and only if $\Pos$ is interval-chain and $\Phi_{\sigma,a,b}$ is total for every $\sigma\in\Pos$ and all $a,b$ at least the length of $\Pos$.
\end{enumerate}
\end{proposition}

\begin{proof}
\emph{Totality.} Let $\kappa=(s,t,r,u)\in\Ka_{a,b}(\sigma)$. Then $s\le t$ and $r\le t$, so $s$ and $r$ lie in the down-set of $t$, which is a chain, so they are comparable and $s\meet r=\min(s,r)$ exists. If $r\le s$ then $\len(s\meet r,u)=\len(r,u)\le b$. If $s\le r$ then $s\le r\le t$, so $\len(s,r)\le\len(s,t)\le a$ and, by additivity of $\len$, $\len(s\meet r,u)=\len(s,r)+\len(r,u)\le a+b$. In both cases $s\meet r\le s\le\sigma$, so $\Phi(\kappa)\in J_{a+b}(\sigma)$.

\emph{Finality.} Fix $j=(y,u)\in J_{a+b}(\sigma)$ and put $L=\len(y,u)\le a+b$. The interval $[y,u]$ lies in the down-set of $u$ and so is a chain; let $z$ be its unique element with $\len(y,z)=\max(0,L-b)$. Then $\len(y,z)\le a$, since $L\le a+b$, and $\len(z,u)=\min(L,b)\le b$, so
\[
   \kappa_0=(y,z,z,u)\in\Ka_{a,b}(\sigma),
\]
the condition $y\le\sigma$ holding because $j\in J_{a+b}(\sigma)$, and $\Phi(\kappa_0)=(y\meet z,u)=(y,u)=j$ because $y\le z$; so $\kappa_0$ lies in the comma fibre \eqref{eq:comma-fibre}.

Let $(s,t,r,v)$ be any object of that fibre, so that $y\le s$, $y\le r$ and $u\le v$. From $y\le r\le v$ and additivity,
\[
   \len(y,r)=\len(y,v)-\len(r,v)\;\ge\;\len(y,u)-b\;=\;L-b,
\]
using $\len(y,v)\ge\len(y,u)$ and $\len(r,v)\le b$. Both $z$ and $r$ lie in the down-set of $v$ (the first through $z\le u\le v$, the second through $r\le v$), which is a chain, so they are comparable; both lie above $y$, and $\len(y,z)=\max(0,L-b)\le\len(y,r)$, whence $z\le r$. Therefore $z\le r\le t$, and $\kappa_0\le(s,t,r,v)$ componentwise. So $\kappa_0$ is a least element of the fibre, whose order complex is then contractible (Lemma \ref{lem:cone-point}), and $\Phi$ is homotopy final.

For the remaining assertions of (i), Proposition \ref{prop:final} gives $\gamma_\sigma$ a quasi-isomorphism for every $F\in\Shv(\Pos;k)$ and every $\sigma$, and so on $\Db(\Shv(\Pos;k))$ by the paragraph following it. Totality holding at every $\sigma$ and all $a,b$, Proposition \ref{prop:meet-coherence} supplies \textup{(T2)}--\textup{(T5)} for the choice $\mu_{a,b}=\gamma_{a,b}$, while \textup{(T1)} is Proposition \ref{prop:derived-exist}(iv). So $\Pos$ is tame.

For (ii), if $\Pos$ is a rooted forest then every closed interval $[y,u]$ lies in the down-set of $u$ and so is a chain, and totality is part of (i). Conversely let $\Pos$ be interval-chain and suppose the down-set of some $t$ is not a chain; choose $s,r\le t$ incomparable. Then $s\meet r$ does not exist: a greatest lower bound $w$ would make $[w,t]$ an interval containing the incomparable pair $s,r$, and no chain does. Taking $\sigma=s$ and $a,b$ at least the length of $\Pos$, the quadruple $(s,t,r,t)$ lies in $\Ka_{a,b}(s)$, so $\Phi_{s,a,b}$ is not total.
\end{proof}

Part (ii) is the exact converse of Example \ref{ex:path-zigzag}. What the two-edge path lacks is that the down-set of $e_{12}$ is $\{v_1,v_2,e_{12}\}$, not a chain (which is the missing meet $v_2\meet v_1$ recorded there), and on an interval-chain poset that is the only thing that can go wrong.

Chains are rooted forests, so Proposition \ref{prop:chain} is the case of a single chain; but on a chain $\LC_a$ is concentrated in degree zero, and on a rooted forest with a branch it is not. On the branch of Example \ref{ex:overglue}\textup{(i)}, Proposition \ref{prop:inexact} gives $H_1\LC_1\sky_v(v)\cong k$ and $C_1$ inexact there, so $\LC_1\sky_v$ carries a class in degree one; the flow holds all the same, so exactness and the flow are independent conditions. And the height is unconstrained: on the rooted tree $\{r<x<X,\ r<y<Y\}$ one has $\len(r,X)=2$, so $\Delta_1^{\Pos}\ne\Delta_2^{\Pos}$ and the positive windows do not all coincide, while $H_1\LC_1\sky_r(r)\cong k$ at the branching root. Finally, this criterion and the terminal-object criterion of Example \ref{ex:closed-edge} are incomparable: the closed edge $\{u,v<e\}$ has down-set $\{u,v,e\}$ at $e$, so it is not a rooted forest and $\Phi$ is not even total on it, while the branch $\{v<e_1,v<e_2\}$ is a rooted tree whose $J_1(v)=\{(v,v),(v,e_1),(v,e_2)\}$ has two incomparable maximal objects and no terminal object. Both contain the chains.

\begin{problem}[The domain of the meet comparison]\label{prob:finality-domain}
Classify the posets $\Pos$ for which, for every $\sigma,a,b$, the meet assignment $\Phi_{\sigma,a,b}$ is total and the resulting meet functor is $k$-homologically final; equivalently, by Remark \ref{rem:weight-independent}, determine the vanishing locus of a defect built from the order of $\Pos$ alone. Proposition \ref{prop:rooted} answers it positively on the rooted forests; Corollary \ref{cor:minimal-unit} settles the case of a minimal $\sigma$ with $a=b=1$ by a condition on length-two intervals; and Lemma \ref{lem:unsaturated} confines the defect to $\len(y,u)>b$ at every $(a,b)$, so that by Corollary \ref{cor:saturated} it is gone once the second window is as long as $\Pos$. That the law itself fails, and not only the meet comparison, is known where the two functors are exhibited with different values (Propositions \ref{prop:sat-nogo}, \ref{prop:diamond-nogo}, Example \ref{ex:path-zigzag}). How far the degree-zero concentration of Proposition \ref{prop:unit-fibre} extends is itself part of the question. Coherence does not: wherever $\Phi$ is total, Proposition \ref{prop:meet-coherence} already supplies compatibility with the transition maps and the full associativity coherence. The answer would be a sufficient condition for tameness (Definition \ref{def:tame}), not a characterisation of it (Example \ref{ex:closed-edge}).
\end{problem}

\section{The interleaving distance}\label{sec:metric}

Where the flow is strong it defines a distance, and the window bounds what that distance can measure (Lemma \ref{lem:constant-flow}). Throughout, $\eta_{a,b}\colon\LC_a\Rightarrow\LC_b$ for $a\le b$ are the derived transition maps of Proposition \ref{prop:derived-exist}(iv), and $\eta_a:=\eta_{0,a}\colon\id\Rightarrow\LC_a$.

\subsection{The tame domain and the distance}

\begin{definition}[Tame poset]\label{def:tame}
$\Pos$ is \emph{tame} if the derived transition maps $\eta_{a,b}$ together with a choice of \emph{flow isomorphisms}
\[
   \mu_{a,b}\colon \LC_a\LC_b\xrightarrow{\sim}\LC_{a+b}
   \qquad(a,b\ge0)
\]
make $a\mapsto\LC_a$ a strong monoidal functor
\[
   (\mathbb Z_{\ge0},\le,+)\longrightarrow
   \operatorname{End}(\Db(\Shv(\Pos;k))),
\]
the source being the poset $(\mathbb Z_{\ge0},\le)$ with monoidal structure $+$. In the terminology of Definition \ref{def:flow} this asks the family $\{\LC_a\}$ to be a strong flow on $\Db(\Shv(\Pos;k))$; \emph{strong} is the term of Definition \ref{def:flow}, and such a family is in particular a flow in the sense of \cite[Def.~2.3]{dSMS}. Unpacked, this is the following data and identities, which are the only properties of the flow that the distance below uses:
\begin{enumerate}
\item[\textup{(T1)}] $\eta_{a,a}=\id$ and $\eta_{b,c}\circ\eta_{a,b}=\eta_{a,c}$ for $a\le b\le c$ (Proposition \ref{prop:derived-exist}(iv));
\item[\textup{(T2)}] each $\mu_{a,b}$ is a natural isomorphism in the object: $\LC_{a+b}(f)\circ\mu_{a,b,X}=\mu_{a,b,Y}\circ\LC_a\LC_b(f)$ for $f\colon X\to Y$;
\item[\textup{(T3)}] $\mu$ is natural in the parameters: for $a\le a'$ and $b\le b'$,
\[
   \eta_{a+b,a'+b'}\circ\mu_{a,b}
   =\mu_{a',b'}\circ(\eta_{a,a'}\LC_{b'})\circ(\LC_a\eta_{b,b'});
\]
\item[\textup{(T4)}] associativity: $\mu_{a+b,c}\circ(\mu_{a,b}\LC_c)=\mu_{a,b+c}\circ(\LC_a\mu_{b,c})$;
\item[\textup{(T5)}] unitality: $\LC_0\simeq\id$, and under this identification $\mu_{0,b}$ and $\mu_{a,0}$ are identities.
\end{enumerate}
Tameness is a property of $\Pos$; the $\mu_{a,b}$ are witnessing data, whereas the transition maps $\eta_{a,b}$ are the canonical maps of Proposition \ref{prop:derived-exist}(iv). By Proposition \ref{prop:choice} the resulting distance does not depend on the choice of $\mu_{a,b}$, so it is attached to $\Pos$ and not to a witnessing family.
\end{definition}

From \textup{(T3)} and \textup{(T5)} one gets the two forms of the unit compatibility used repeatedly,
\[
   \mu_{a,b}\circ\LC_a(\eta_b)=\eta_{a,a+b},
   \qquad
   \mu_{a,b}\circ(\eta_a\LC_b)=\eta_{b,a+b} ,
\]
by taking $(a,0)\le(a,b)$ and $(0,b)\le(a,b)$ respectively in \textup{(T3)}, with $\LC_0\simeq\id$ and $\mu_{a,0}=\mu_{0,b}=\id$ from \textup{(T5)}. For the meet comparisons $\mu_{a,b}=\gamma_{a,b}$, items \textup{(T2)}--\textup{(T5)} are the content of Proposition \ref{prop:meet-coherence}; what tameness adds is that the $\gamma_{a,b}$, or some other choice of $\mu_{a,b}$, be invertible.

Every chain is tame: under the strict identifications of Proposition \ref{prop:chain} the transition maps are the evident restrictions $F(\sigma\up a)\to F(\sigma\up b)$, the $\mu_{a,b}$ are identities, and \textup{(T1)}--\textup{(T5)} hold because every map in sight is a restriction map of $F$ and any two composites of restrictions between the same two stalks agree. So is the closed edge (Example \ref{ex:closed-edge}), and so is every rooted forest (Proposition \ref{prop:rooted}), on which $\LC_a$ need not be concentrated in degree zero. Being interval-chain is not enough (Example \ref{ex:path-zigzag}), and a poset carrying a branching length-two interval (Proposition \ref{prop:diamond-nogo}) is not tame; in particular no face poset of a finite regular cell complex of dimension $\ge2$ is. Fix a tame $\Pos$ throughout this section.

\begin{example}[A tame poset that is not a chain]\label{ex:closed-edge}
Let $\Pos=\{u,v<e\}$ be the face poset of a closed $1$-cell: two vertices $u\ne v$ and one edge $e$ with $u,v<e$. Then $\len(s,t)\le1$ throughout, so $\Delta_a^{\Pos}=\Delta_1^{\Pos}=\ct$ and $C_a=C_1$ for every $a\ge1$. The poset and the largest of its window-comma posets are
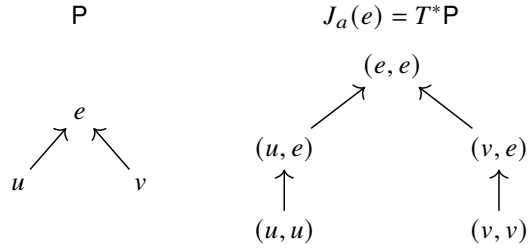
\[
\begin{array}{c@{\qquad\qquad}c}
\Pos & J_a(e)=\ct\\[0.8em]
\begin{tikzcd}[row sep=1.6em,column sep=1.2em,cramped]
& e & \\
u\arrow[ur] & & v\arrow[ul]
\end{tikzcd}
&
\begin{tikzcd}[row sep=1.6em,column sep=1.2em,cramped]
& (e,e) & \\
(u,e)\arrow[ur] & & (v,e)\arrow[ul]\\
(u,u)\arrow[u] & & (v,v)\arrow[u]
\end{tikzcd}
\end{array}
\]
Each window-comma poset has a terminal object:
\[
   J_a(u)=\{(u,u),(u,e)\},\quad
   J_a(v)=\{(v,v),(v,e)\},\quad
   J_a(e)=\ct ,
\]
with terminal objects $(u,e)$, $(v,e)$ and $(e,e)$ respectively. By Proposition \ref{prop:bar} and Lemma \ref{lem:terminal-bar}, for $a\ge1$ the object $\LC_aF$ is the sheaf constant equal to $F(e)$, concentrated in degree $0$, with identity corestrictions; the corestrictions are identities because the inclusion $J_a(u)\subseteq J_a(e)$ carries the class of $(u,e)$ to the class of $(u,e)$, which is identified in the colimit with the terminal class $(e,e)$ along $(u,e)\le(e,e)$, and the coefficient map is $F(e\le e)$. In particular $C_1$ is exact here, being evaluation at $e$ followed by the constant-sheaf functor, consistently with Proposition \ref{prop:inexact}: each vertex lies on one edge, so $\deg v-1=0$.

For $a,b\ge1$,
\[
   (\LC_a\LC_bF)(\sigma)=(\LC_bF)(e)=F(e)=(\LC_{a+b}F)(\sigma),
\]
and the identification is an equality of functors, not merely an isomorphism; with $\LC_0=\id$ this makes $\{\LC_a\}$ a strict flow, so $\mu_{a,b}=\id$. The identities \textup{(T1)}--\textup{(T5)} then reduce to identities among the transition maps alone, and these hold because $\eta_{a,b}=\id$ for $1\le a\le b$ (the two windows being equal, so that $\tau_{a,b}=\id$ by the initiality in Proposition \ref{prop:derived-exist}(iv)), while the components of $\eta_{0,a}$ at a stalk $\sigma$ are the corestrictions $F(\sigma)\to F(e)$ of $F$, and are the identity at $\sigma=e$. Thus $\Pos$ is tame, and it is not a chain, $u$ and $v$ being incomparable.

Two features bear on what follows. First, the meet assignment is \emph{not} total on $\Pos$: at $\sigma=e$ the quadruple $(u,e,v,v)$ lies in $\Ka_{1,1}(e)$ and $u\meet v$ does not exist. So a missing meet obstructs the meet comparison $\gamma$, not tameness: a flow isomorphism can exist where the meet comparison does not, and Problem \ref{prob:finality-domain} is therefore a sufficient-condition question, not a characterisation of the tame class. Second, the mechanism generalises: if every $J_a(\sigma)$ has a terminal object $(\sigma,\sigma\up a)$ and $(\sigma\up a)\up b=\sigma\up(a+b)$, the same argument gives a strict flow, and $\Pos$ is tame. Proposition \ref{prop:chain} is the chain case of this criterion and the present example is the two-vertex case; Proposition \ref{prop:rooted} supplies a criterion incomparable with it.
\end{example}

\begin{definition}[Interleaving distance]\label{def:interleaving}
For $F,G\in\Db(\Shv(\Pos;k))$ and $a\ge0$, an \emph{$a$-interleaving} is a pair of morphisms $\varphi\colon F\to\LC_aG$ and $\psi\colon G\to\LC_aF$ (an interleaving in the sense of \cite[Def.~2.6]{dSMS} for the flow $\{\LC_a\}$, and of \cite{BdSS} for the translations $\eta_{a,b}$) such that the two triangles
\[
\begin{array}{c}
\begin{tikzcd}[column sep=huge,row sep=large]
F\arrow[r,"\varphi"]\arrow[dr,"\eta_{2a,F}"'] &
\LC_aG\arrow[r,"\LC_a\psi"] &
\LC_a\LC_aF\arrow[dl,"\mu_{a,a,F}"]\\
& \LC_{2a}F &
\end{tikzcd}
\\[1.2em]
\begin{tikzcd}[column sep=huge,row sep=large]
G\arrow[r,"\psi"]\arrow[dr,"\eta_{2a,G}"'] &
\LC_aF\arrow[r,"\LC_a\varphi"] &
\LC_a\LC_aG\arrow[dl,"\mu_{a,a,G}"]\\
& \LC_{2a}G &
\end{tikzcd}
\end{array}
\]
commute, that is
\[
   \mu_{a,a,F}\circ\LC_a(\psi)\circ\varphi=\eta_{2a,F}\colon F\to\LC_{2a}F,
   \qquad
   \mu_{a,a,G}\circ\LC_a(\varphi)\circ\psi=\eta_{2a,G}\colon G\to\LC_{2a}G,
\]
the flow isomorphism $\mu_{a,a}\colon\LC_a\LC_a\to\LC_{2a}$ being displayed here and suppressed below where no computation depends on it. Set
\[
   d(F,G)=\inf\{\,a\in\mathbb Z_{\ge0}: F\text{ and }G\text{ are }a\text{-interleaved}\,\}
   \;\in\;\mathbb Z_{\ge0}\cup\{\infty\},
\]
the value being $\infty$ when no such $a$ exists. The parameter monoid is $\mathbb Z_{\ge0}$, so the value set consists of the non-negative integers together with $\infty$; by Lemma \ref{lem:upward} the set of admissible $a$ is upward closed, so a finite infimum is a minimum and is attained.
\end{definition}

In general an interleaving distance depends on the flow and not only on the category \cite[after Thm.~2.7]{dSMS}. Here it does not.

\begin{proposition}[Independence of the witnessing flow data]\label{prop:choice}
Let $\Pos$ be tame and let $(\mu_{a,b})$ and $(\mu'_{a,b})$ be two families of flow isomorphisms witnessing tameness in the sense of Definition \ref{def:tame}, both taken with respect to the canonical transition maps $\eta_{a,b}$ of Proposition \ref{prop:derived-exist}(iv). Then for all $F,G\in\Db(\Shv(\Pos;k))$ and all $a\ge0$, a pair $(\varphi,\psi)$ is an $a$-interleaving with respect to $(\mu_{a,b})$ if and only if it is an $a$-interleaving with respect to $(\mu'_{a,b})$. The two families therefore define the same $d$, and $d$ depends only on $\Pos$.
\end{proposition}

\begin{proof}
Put $\theta_{a,b}=\mu'_{a,b}\circ\mu_{a,b}^{-1}$, a natural automorphism of $\LC_{a+b}$; it is natural in the object because both $\mu_{a,b}$ and $\mu'_{a,b}$ are, by \textup{(T2)}. Both families satisfy the first unit identity displayed after Definition \ref{def:tame}, which is a consequence of \textup{(T3)} and \textup{(T5)} alone:
\[
   \mu_{a,b}\circ\LC_a(\eta_b)=\eta_{a,a+b}=\mu'_{a,b}\circ\LC_a(\eta_b).
\]
Composing the left-hand equality with $\mu_{a,b}^{-1}$ and substituting into the right-hand one gives
\[
   \theta_{a,b}\circ\eta_{a,a+b}
   =\mu'_{a,b}\circ\mu_{a,b}^{-1}\circ\mu_{a,b}\circ\LC_a(\eta_b)
   =\mu'_{a,b}\circ\LC_a(\eta_b)
   =\eta_{a,a+b} .
\]
Take $b=a$ and evaluate at $F$. By \textup{(T1)}, $\eta_{2a,F}=\eta_{a,2a,F}\circ\eta_{a,F}$, so
\[
   \theta_{a,a,F}\circ\eta_{2a,F}
   =\theta_{a,a,F}\circ\eta_{a,2a,F}\circ\eta_{a,F}
   =\eta_{a,2a,F}\circ\eta_{a,F}
   =\eta_{2a,F} .
\]
Thus the passage from one witnessing family to the other is represented by
\[
\begin{tikzcd}[column sep=huge,row sep=large]
F
  \arrow[r,"\LC_a(\psi)\circ\varphi"]
  \arrow[dr,"\eta_{2a,F}"']
  \arrow[ddr,bend right=18,"\eta_{2a,F}"']
&
\LC_a\LC_aF
  \arrow[d,"\mu_{a,a,F}",swap]
  \arrow[dd,bend left=32,"\mu'_{a,a,F}"]
\\
&
\LC_{2a}F
  \arrow[d,"\theta_{a,a,F}",swap]
\\
&
\LC_{2a}F ,
\end{tikzcd}
\]
where $\mu'_{a,a,F}=\theta_{a,a,F}\circ\mu_{a,a,F}$ and
$\theta_{a,a,F}\circ\eta_{2a,F}=\eta_{2a,F}$.
For any $\varphi\colon F\to\LC_aG$ and $\psi\colon G\to\LC_aF$,
\[
   \mu'_{a,a,F}\circ\LC_a(\psi)\circ\varphi
   =\theta_{a,a,F}\circ\bigl(\mu_{a,a,F}\circ\LC_a(\psi)\circ\varphi\bigr),
\]
and since $\theta_{a,a,F}$ fixes $\eta_{2a,F}$, the first interleaving identity holds for $\mu'$ as soon as it holds for $\mu$. The same computation with $F$ and $G$ exchanged gives the second identity, and the converse follows by exchanging the roles of the two families, $\theta^{-1}_{a,b}=\mu_{a,b}\circ(\mu'_{a,b})^{-1}$ satisfying the same relation. The two families therefore admit the same $a$-interleavings for every $a$, and the same $d$.
\end{proof}

Only \textup{(T1)}, \textup{(T2)}, \textup{(T3)} and \textup{(T5)} enter; associativity \textup{(T4)}, which the triangle inequality needs, is not used. The $\eta_{a,b}$ being fixed in advance, \textup{(T3)} and \textup{(T5)} already pin down a witnessing family's effect on the two composites of Definition \ref{def:interleaving}. The residual freedom $\theta_{a,b}$ need not be the identity, but it acts trivially on the transition maps, which is all the interleaving conditions see.

\begin{lemma}[Upward closure]\label{lem:upward}
Let $\Pos$ be tame, let $F,G\in\Db(\Shv(\Pos;k))$ be $a$-interleaved, and let $a\le a'$. Then $F$ and $G$ are $a'$-interleaved.
\end{lemma}

\begin{proof}
Let $(\varphi,\psi)$ be an $a$-interleaving and set
\[
   \varphi'=\eta_{a,a',G}\circ\varphi\colon F\to\LC_{a'}G,
   \qquad
   \psi'=\eta_{a,a',F}\circ\psi\colon G\to\LC_{a'}F .
\]
Then
\[
   \LC_{a'}(\psi')\circ\varphi'
   =\LC_{a'}(\eta_{a,a',F})\circ\LC_{a'}(\psi)\circ\eta_{a,a',G}\circ\varphi .
\]
Naturality of the transformation $\eta_{a,a'}$ at $\psi\colon G\to\LC_aF$ gives $\LC_{a'}(\psi)\circ\eta_{a,a',G}=\eta_{a,a',\LC_aF}\circ\LC_a(\psi)$, so, using the $a$-interleaving identity in the form $\LC_a(\psi)\circ\varphi=\mu_{a,a,F}^{-1}\circ\eta_{2a,F}$,
\[
   \LC_{a'}(\psi')\circ\varphi'
   =\LC_{a'}(\eta_{a,a',F})\circ\eta_{a,a',\LC_aF}\circ\mu_{a,a,F}^{-1}\circ\eta_{2a,F} .
\]
It remains to identify $\mu_{a',a',F}\circ\LC_{a'}(\eta_{a,a',F})\circ\eta_{a,a',\LC_aF}\circ\mu_{a,a,F}^{-1}$ with $\eta_{2a,2a',F}$. Apply \textup{(T3)} twice. In the second parameter, with $(a',a)\le(a',a')$,
\[
   \mu_{a',a'}\circ(\LC_{a'}\eta_{a,a'})=\eta_{a+a',2a'}\circ\mu_{a',a};
\]
in the first, with $(a,a)\le(a',a)$,
\[
   \mu_{a',a}\circ(\eta_{a,a'}\LC_a)=\eta_{2a,a+a'}\circ\mu_{a,a}.
\]
Since $\eta_{a,a',\LC_aF}$ is the component at $F$ of the whiskering $\eta_{a,a'}\LC_a$, composing the two displays gives
\[
   \mu_{a',a',F}\circ\LC_{a'}(\eta_{a,a',F})\circ\eta_{a,a',\LC_aF}
   =\eta_{a+a',2a',F}\circ\eta_{2a,a+a',F}\circ\mu_{a,a,F}
   =\eta_{2a,2a',F}\circ\mu_{a,a,F}
\]
by \textup{(T1)}. So $\mu_{a',a',F}\circ\LC_{a'}(\psi')\circ\varphi'=\eta_{2a,2a',F}\circ\eta_{2a,F}=\eta_{2a',F}$, again by \textup{(T1)}. The second identity follows by exchanging the roles of $(F,\varphi)$ and $(G,\psi)$.
\end{proof}

\begin{proposition}[Basic properties]\label{prop:metric}
Let $\Pos$ be tame. Then the distance $d$ of Definition \ref{def:interleaving} is an extended pseudometric on $\Db(\Shv(\Pos;k))$, and for all objects $F,G$ and every $a\ge0$
\[
   d(F,\LC_aF)\le a,
   \qquad
   d(\LC_aF,\LC_aG)\le d(F,G) .
\]
Its values lie in $\mathbb Z_{\ge0}\cup\{\infty\}$, and a finite value is attained.
\end{proposition}

\begin{proof}
Items \textup{(T1)}--\textup{(T5)} of Definition \ref{def:tame} are the data of a strong flow, in particular of a flow in the sense of \cite[Def.~2.3, Def.~2.4]{dSMS}, with parameter monoid $\mathbb Z_{\ge0}$ in place of $[0,\infty)$, and Definition \ref{def:interleaving} is the interleaving it induces \cite[Def.~2.6]{dSMS}. The pseudometric axioms are then \cite[Thm.~2.7]{dSMS}, \cite[Thm.~3.21]{BdSS}, whose arguments are formal in the parameter monoid: reflexivity is the $0$-interleaving $(\id,\id)$, symmetry is the symmetry of Definition \ref{def:interleaving}, and the triangle inequality composes an $a$-interleaving $(\varphi,\psi)$ of $F,G$ with a $b$-interleaving $(\varphi',\psi')$ of $G,H$ into
\[
   \mu_{a,b,H}\circ\LC_a(\varphi')\circ\varphi\colon F\to\LC_{a+b}H,
   \qquad
   \mu_{b,a,F}\circ\LC_b(\psi)\circ\psi'\colon H\to\LC_{a+b}F ,
\]
the verification of the two defining identities for this pair being where associativity \textup{(T4)} is consumed. The last two estimates are the same formalities: $\eta_a$ witnesses $d(F,\LC_aF)\le a$ through $\mu_{a,a,F}^{-1}\circ\eta_{2a,F}$ paired with $\id_{\LC_aF}$, and applying $\LC_a$ to an interleaving gives the stability bound.

The value set is the only part of the statement special to this flow. The parameter monoid is $\mathbb Z_{\ge0}$, so $d$ takes values in $\mathbb Z_{\ge0}\cup\{\infty\}$; and by Lemma \ref{lem:upward} the set of admissible parameters is upward closed, so a finite infimum is a minimum. An infinite value therefore records that no parameter is admissible at all, rather than that an infimum is unattained; the three infinite entries of Example \ref{ex:aoki-chain-table} are of the former kind.
\end{proof}

\begin{example}[The distance on a chain, and where it parts from Aoki's]\label{ex:aoki-chain-table}
Let $\Pos=[3]=\{0<1<2<3\}$, with $\rk(i)=i$, so that $\len(i,j)=j-i$; it is a tame poset. For $0\le p\le q\le3$, let $I[p,q]$ be the interval sheaf, equal to $k$ on $p\le i\le q$ with identity structure maps inside the interval and $0$ elsewhere, regarded as a complex concentrated in degree $0$. By Proposition \ref{prop:chain},
\[
   \LC_aF(i)=F(\min(i+a,3)),
\]
so the conditions of Definition \ref{def:interleaving} are scalar equations and $d$ may be evaluated directly. Set beside it, for comparison, Aoki's real-parameter height-interleaving distance $d_h^{\mathrm{Aoki}}$ for the height-difference function $\rho_h(i,j)=h(j)-h(i)=j-i$, where $h(i)=i$ is his height function on $[3]$ and the symbol $\ell$ remains reserved for the height $\len$ of two arguments \cite[Def.~3.7]{Aoki}. Then:
\[
\begin{array}{c|c|c}
(F,G) & d(F,G) & d_h^{\mathrm{Aoki}}(F,G)\\
\hline
(I[0,1],I[1,2]) & 1 & 0\\
(I[0,1],I[2,3]) & \infty & 0\\
(I[0,0],0) & 1 & 0\\
(I[0,2],0) & 2 & 1\\
(I[3,3],0) & \infty & 0\\
(I[0,3],0) & \infty & 1
\end{array}
\]
The first column is the distance of Definition \ref{def:interleaving}, which on this chain takes the values $1$, $2$ and $\infty$. Its infinite entries arise in two ways. In the rows $(I[3,3],0)$ and $(I[0,3],0)$ it is the clamp: $\eta_{2a,F}$ is the identity at $i=3$ for every $a$, so it never vanishes and no interleaving with $0$ exists. In the row $(I[0,1],I[2,3])$ it is a disjointness: $\LC_aI[0,1]$ vanishes at $i=2,3$ for every $a$, so the only map $I[2,3]\to\LC_aI[0,1]$ is zero, while $\eta_{2a,I[2,3]}$ is not. The second column is computed from the latching and matching functors of \cite[Def.~3.4]{Aoki}; none of its six entries is attained, since at $r=0$ both functors are the identity, so that a $0$-height-interleaving is an isomorphism, while in the two rows with value $1$ the parameter $r=1$ still fails and every $r>1$ succeeds. That the admissible set is upward closed is \cite[Lem.~3.9]{Aoki}, the analogue of Lemma \ref{lem:upward}.

The columns differ, and not because one construction is derived and the other is not: on a chain $\LC_a$ is concentrated in degree $0$ by Proposition \ref{prop:chain}, so both entries are ordinary height shifts. The discrepancy is the boundary convention: the flow used here clamps at the terminal element, so $\eta_{2a,F}$ cannot vanish at the top of a nonzero $F$, whereas Aoki's deep upper windows become empty in the top collar. Even on a finite chain the two distances are related but not identical.
\end{example}

The two examples that follow leave the chain, one on each of the two tameness criteria, and both run on saturation: once $a$ reaches the length of a finite $\Pos$, the window is all of $\ct$ and $\LC_a$ stops moving.

Write $c^\ast\colon\Db(\mathrm{Vec}_k)\to\Db(\Shv(\Pos;k))$ for the constant-object functor. Under the presentation of Section \ref{sec:conventions} it is right adjoint to $\colim_{\Pos}$, and it is exact; so the adjunction derives, and
\begin{equation}\label{eq:const-adjunction}
   \Db(\Shv(\Pos;k))\bigl(F,\,c^\ast H\bigr)
   \;\cong\;
   \Db(\mathrm{Vec}_k)\bigl(\hocolim_{\Pos}F,\,H\bigr),
\end{equation}
naturally, the unit being the canonical map $F\to c^\ast\hocolim_{\Pos}F$.

\begin{lemma}[Above the saturation threshold the distance compares derived colimits]\label{lem:constant-flow}
Let $\Pos$ be tame with $\Delta(\Pos)$ contractible, let $m\ge1$, and suppose that $\Delta_a^{\Pos}=\ct$ for every $a\ge m$ and that for such $a$ there are isomorphisms $\LC_aF\cong c^\ast\hocolim_{\Pos}F$, natural in $F$, carrying $\eta_{0,a,F}$ to the unit of \eqref{eq:const-adjunction}. Then $d$ takes no value outside $\{0,1,\ldots,m\}\cup\{\infty\}$, and for $F,G\in\Db(\Shv(\Pos;k))$
\begin{enumerate}
\item[\textup{(i)}] $d(F,G)=0$ if and only if $F\cong G$;
\item[\textup{(ii)}] $d(F,G)\le m$ if and only if $\hocolim_{\Pos}F\cong\hocolim_{\Pos}G$.
\end{enumerate}
\end{lemma}

\begin{proof}
\textup{(i)} A $0$-interleaving is a pair $\varphi\colon F\to\LC_0G=G$ and $\psi\colon G\to F$ with $\psi\varphi=\eta_{0,F}=\id$ and $\varphi\psi=\id$, by \textup{(T1)} and \textup{(T5)}.

For the rest, the windows $\Delta_a^{\Pos}$ coincide for all $a\ge m$, so $\LC_a=\LC_m$ there and $\eta_{a,b}=\id$ for $m\le a\le b$, the transition map $\tau_{a,b}$ being the identity by the initiality in Proposition \ref{prop:derived-exist}\textup{(iv)}. So $\eta_{2a,F}=\eta_{0,m,F}$ for every $a\ge m$, and an $a$-interleaving with $a\ge m$ is literally the data and the identities of an $m$-interleaving; no finite value above $m$ is attained.

\textup{(ii)} Write $H(F)=\hocolim_{\Pos}F$. Since $\Delta(\Pos)$ is contractible, $\hocolim_{\Pos}k\cong k$ by Lemma \ref{lem:cone-point} and Lemma \ref{lem:order-nerve}\textup{(iii)}, so $H(c^\ast H')\cong H'$ for every $H'$; thus $\LC_m\LC_m\cong\LC_{2m}=\LC_m$ and this identification is $\mu_{m,m}$, which by Proposition \ref{prop:choice} is the only choice the distance can see. By \eqref{eq:const-adjunction} a map $\varphi\colon F\to\LC_mG=c^\ast H(G)$ is the same as $f\colon H(F)\to H(G)$, and $\psi\colon G\to\LC_mF$ the same as $g\colon H(G)\to H(F)$; naturality of \eqref{eq:const-adjunction} identifies $\mu_{m,m,F}\circ\LC_m(\psi)\circ\varphi$ with $g\circ f$, while $\eta_{2m,F}=\eta_{0,m,F}$ is the unit and corresponds to $\id_{H(F)}$. So the two triangles of Definition \ref{def:interleaving} say $gf=\id$ and $fg=\id$.
\end{proof}

The chain of Example \ref{ex:aoki-chain-table} is the case $m=3$: there $\LC_aF$ is constant at $F(3)$ for $a\ge3$, and $\hocolim_{[3]}F=F(3)$ because $3$ is terminal. So on $[3]$ one has $d(F,G)=\infty$ if and only if $F(3)\not\cong G(3)$, and that single condition accounts for all three infinite rows of the table.

\begin{example}[The closed edge: the distance is a stalk comparison]\label{ex:edge-distance}
Let $\Pos=\{u,v<e\}$ be as in Example \ref{ex:closed-edge}. There $\len\le1$, so $\Delta_a^{\Pos}=\ct$ for $a\ge1$; $\Delta(\Pos)$ is contractible, $e$ being greatest (Lemma \ref{lem:cone-point}); $\hocolim_{\Pos}F=F(e)$ for the same reason, the terminal object of $\Pos$ making the colimit an evaluation and the derived colimit that same evaluation; and Example \ref{ex:closed-edge} computes $\LC_aF$, for $a\ge1$, to be the constant object at $F(e)$, with $\eta_{0,a,F}$ the corestrictions $F(\sigma)\to F(e)$, which is the unit. Lemma \ref{lem:constant-flow} applies with $m=1$, and
\[
   d(F,G)\le 1
   \quad\Longleftrightarrow\quad
   F(e)\cong G(e)\ \text{in }\Db(\mathrm{Vec}_k) .
\]
So on the closed edge $d$ is the discrete comparison of the stalk at $e$: it is $0$ on isomorphic objects, $1$ on non-isomorphic objects with isomorphic edge stalk, and $\infty$ otherwise. In particular $d(\sky_u,\sky_v)=1$ (the distance does not separate the two vertices) and $d(\sky_u,0)=1$, a nonzero object at finite distance from zero. Nothing here is derived: $C_1$ is exact on $\Pos$ (Example \ref{ex:closed-edge}), and $\hocolim_{\Pos}$ is an evaluation.
\end{example}

\begin{example}[A tree of height two: a distance of $2$]\label{ex:tree-distance}
Let $\Pos=\{r<x<X,\ r<y<Y\}$ with $\rk(r)=0$, $\rk(x)=\rk(y)=1$ and $\rk(X)=\rk(Y)=2$: the rooted tree with one branching root and two branches of length two, the poset already used after Proposition \ref{prop:rooted}. Every $\cl(z)$ is a chain, so $\Pos$ is a rooted forest and is tame, with $\Phi$ total and homotopy final at every apex and every pair of windows (Proposition \ref{prop:rooted}); it is not a chain, and $J_1(r)=\{(r,r),(r,x),(r,y)\}$ has two incomparable maximal objects, so the terminal-object criterion of Example \ref{ex:closed-edge} does not apply. The length of $\Pos$ is $2$, so $\Delta_a^{\Pos}=\ct$ if and only if $a\ge2$, and $\Delta(\Pos)$ is contractible, $r$ being least.

The subposet $\{r<X,\ r<Y\}$ is homotopy final in $\Pos$: the comma poset over $r$ is all of it and has least element $r$, and over each of $x,X,y,Y$ it is a single object, so
\begin{equation}\label{eq:tree-hocolim}
   \hocolim_{\Pos}F
   \;\simeq\;
   \bigl[\,F(r)\longrightarrow F(X)\oplus F(Y)\,\bigr]
\end{equation}
in bar degrees $1$ and $0$, the map being the two composite corestrictions with a sign. For $a\ge2$ the length condition in $J_a(\sigma)$ is vacuous, so $\pi(z,w)=(r,w)$ carries $J_a(\sigma)$ into itself; it is monotone, idempotent, below the identity, does not move $w$ (so the coefficient diagram is $\pi^\ast$ of the one on its image $J_a(r)\cong\Pos$), and it is homotopy final, the comma poset over $(r,w_0)$ having $(r,w_0)$ itself as least element. So $\LC_aF(\sigma)\simeq\hocolim_{\Pos}F$, the corestrictions being sections of these equivalences, and $\eta_{0,a}$ is the unit. Lemma \ref{lem:constant-flow} therefore applies with $m=2$:
\[
   d(F,G)\le 2
   \quad\Longleftrightarrow\quad
   \hocolim_{\Pos}F\cong\hocolim_{\Pos}G\ \text{in }\Db(\mathrm{Vec}_k).
\]

Take the skyscrapers $\sky_X$ and $\sky_Y$. Both leaves are maximal, so $\st(X)=\{X\}$ and $\sky_X=k_{\st(X)}$ is the representable projective at $X$ (Lemma \ref{lem:enough-projectives}), and likewise at $Y$. By \eqref{eq:tree-hocolim},
\[
   \hocolim_{\Pos}\sky_X\;\simeq\;[\,0\to k\oplus0\,]\;\cong\;k\;\cong\;\hocolim_{\Pos}\sky_Y ,
\]
so $d(\sky_X,\sky_Y)\le2$; and $\sky_X\not\cong\sky_Y$, so the distance is not $0$.

It is not $1$ either. The second coordinates occurring in $J_1(X)$ are those of $\bigcup_{z\le X}\st_1(z)=\{r,x,y\}\cup\{x,X\}\cup\{X\}$, which does not contain $Y$; so the coefficient diagram $w\mapsto\sky_Y(w)$ on $J_1(X)$ is identically zero and $(\LC_1\sky_Y)(X)=0$. Since $\sky_X$ is the projective $k_{\st(X)}$, evaluation gives $\Db(\Shv(\Pos;k))\bigl(\sky_X,H\bigr)\cong H_0\bigl(H(X)\bigr)$ for every $H$, so the only morphism $\varphi\colon\sky_X\to\LC_1\sky_Y$ is zero. The first triangle of Definition \ref{def:interleaving} would then force $\eta_{2,\sky_X}=0$; but $\eta_{2,\sky_X}$ is the unit $\sky_X\to c^\ast\hocolim_{\Pos}\sky_X$, which at the stalk $X$ is the identity of $k$. So no $1$-interleaving exists and
\[
   d(\sky_X,\sky_Y)=2 .
\]

At window scale $1$ the leaf $Y$ is not among the elements any window-comma poset over $X$ reaches; at scale $2$ it is, and both skyscrapers have become the constant object $k$. The distance is the number of ranks the window must climb before the two branches see each other through the root, and here that is the length of $\Pos$. It is not a height shift on a chain: on this poset $\LC_a$ is not concentrated in degree zero, since $J_1(r)$ gives $(\LC_1\sky_r)(r)\simeq[\,k\to0\,]\cong k[1]$, consistently with $H_1\LC_1\sky_r(r)\cong k$ (Proposition \ref{prop:inexact}). The pair realising the distance $2$ is nevertheless an ordinary one: both $\hocolim_{\Pos}\sky_X$ and $\hocolim_{\Pos}\sky_Y$ sit in degree zero.
\end{example}

\begin{remark}[Domain of definition]\label{rem:metric-scope}
The construction of $d$ uses the flow law (Definition \ref{def:interleaving}), so the distance lives only where the flow lives: on the tame posets, delimited at the start of this section by Example \ref{ex:path-zigzag} and Proposition \ref{prop:diamond-nogo}.

The two sufficient criteria there, terminal window-comma objects (Proposition \ref{prop:chain}, Example \ref{ex:closed-edge}) and rooted forests (Proposition \ref{prop:rooted}), are incomparable, and they differ in what they deliver: under the first $\LC_a$ is concentrated in degree zero and the distance compares clamped height shifts, under the second it need not be: on a rooted forest with a branch it is not, so there the distance is defined on genuinely derived objects (Proposition \ref{prop:inexact}, Example \ref{ex:overglue}\textup{(i)}), as Example \ref{ex:tree-distance} records. A disjoint union of tame posets is tame, the index posets $J_a(\sigma)$ and $\Ka_{a,b}(\sigma)$ lying in one component. At a minimal apex with unit windows the surviving condition is that every length-two interval have exactly one interior element (Corollary \ref{cor:minimal-unit}).

Nor can tameness be arranged by restriction, $C_a$ not being local. For a full subposet $Q\subseteq\Pos$ write $C_a^Q$ for the convolution of Definition \ref{def:Ca} formed in $Q$; the full inclusion $J_a^Q(\sigma)\subseteq J_a(\sigma)$ induces $C_a^Q(F|_Q)\to(C_aF)|_Q$, an isomorphism for every $F$ if and only if that inclusion is final. Closure of $Q$ under $\le$ and $\ge$ is sufficient, not necessary: on the closed edge $Q=\{e\}$ is compatible for every $a$, the two index posets sharing the terminal object $(e,e)$, and is not closed. Branching leaves no such accident (on $\{v<e_1,\ v<e_2\}$ no proper full subposet is compatible, although $\{e_1,e_2\}$ is itself tame), so a tame subposet carries no information about the ambient flow, and the distance admits no patching over a decomposition.
\end{remark}

\appendix

\section{Posets, simplicial complexes, and simplicial sets}\label{app:combinatorics}

The conventions are those of Wachs \cite[\S1.1]{Wachs}.

\subsection{Chains, and the Alexandrov topology}\label{ssec:chains-alexandrov}

A \emph{chain} of a finite poset $X$ is a totally ordered subset $C$; written $x_0<\cdots<x_n$ it is a \emph{strict $n$-chain}. It is \emph{saturated} if no element of $X$ lies strictly between two consecutive entries, and \emph{maximal} if no element of $X\setminus C$ is comparable with all of $C$.

The \emph{Alexandrov} topology on $X$, under which Curry's equivalence is read in Section \ref{sec:conventions}, has the up-sets for its open sets. So the smallest open set containing $p$ is the principal up-set $\st(p)$ of Section \ref{sec:conventions}, the closed sets are the down-sets, and \emph{up-closed} and \emph{down-closed} always refer to this topology.

A sheaf on this topology has for stalk at $p$ its sections over $\st(p)$, the value $F(p)$ of the corresponding functor, and $F(U)=\lim_{p\in U}F(p)$ over an arbitrary open $U$. It suffices to verify the sheaf axiom on the sets $\st(p)$, since in any cover $\st(p)=\bigcup_iU_i$ one of the $U_i$ already contains $p$, and with it $\st(p)$.

\subsection{Order complexes and nerves}\label{ssec:order-nerve}

\begin{definition}[Order complex, nerve, face poset]\label{def:order-complex}
Let $X$ be a finite poset and $\Sigma$ a finite abstract simplicial complex. The \emph{order complex} $\Delta(X)$ is the simplicial complex with vertex set $X$ whose faces are the strict chains of $X$ \cite[\S1.1]{Wachs}. The \emph{nerve} $N(X)$ is the nerve of $X$: its $n$-simplices are the weak chains $x_0\le\cdots\le x_n$, with $d_i$ deleting $x_i$ and $s_i$ repeating it. The \emph{face poset} $\Pos(\Sigma)$ is the set of nonempty faces ordered by inclusion, and the \emph{augmented face poset} $\widehat\Pos(\Sigma)=\Pos(\Sigma)\sqcup\{\hat0\}$ adjoins the empty face as a new least element $\hat0$. The same hat carries the same meaning on the face poset $\Pos(K)$ of a finite regular cell complex (Section \ref{sec:conventions}).
\end{definition}

\begin{lemma}[Order complex against nerve]\label{lem:order-nerve}
Let $X$ be a finite poset. On a simplicial set, $C_\ast(-;k)$ means the normalized chain complex of the free simplicial $k$-vector space (Definition \ref{def:normalized-complex}).
\begin{enumerate}
\item[\textup{(i)}] The nondegenerate simplices of $N(X)$ are the strict chains of $X$, that is the faces of $\Delta(X)$; each is totally ordered by the order of $X$, so $\Delta(X)$ carries a canonical vertex ordering, and the choice that the embedding of simplicial complexes into simplicial sets forces in general is absent here.
\item[\textup{(ii)}] The normalized complex of Definition \ref{def:normalized-complex}, applied to the free simplicial $k$-vector space on $N(X)$, is the simplicial chain complex of the order complex:
\[
   \mathcal N\bigl(k[N(X)]\bigr)=C_\ast(\Delta(X);k),
\]
compatibly with the augmentations.
\item[\textup{(iii)}] On the constant diagram that complex is $\operatorname{Bar}(X,k)$ of Definition \ref{def:bar-hocolim}, so bar degree $n$ is indexed by the $n$-faces of $\Delta(X)$.
\item[\textup{(iv)}] $\lvert N(X)\rvert\cong\lvert\Delta(X)\rvert$, and $\widetilde H_\ast(N(X);k)\cong\widetilde H_\ast(\Delta(X);k)$.
\end{enumerate}
A homotopical or homological attribute of a finite poset may therefore be read on either model.
\end{lemma}

\begin{proof}
(i) A weak chain is degenerate if and only if two consecutive entries agree, so the nondegenerate simplices are the strict chains, and these are the faces of $\Delta(X)$.

(ii) and (iii) In degree $n$ the degenerate subspace of $k[N(X)]$ is spanned by the degenerate simplices, so the quotient of Definition \ref{def:normalized-complex} is free on the strict $n$-chains with differential $\sum_i(-1)^id_i$, the map $d_i$ deleting the $i$-th entry: the simplicial boundary of $\Delta(X)$ in that vertex order, and equally $\operatorname{Bar}(X,k)$.

(iv) Every face of a strict chain is a strict chain and distinct subsets of a chain are distinct simplices, so $N(X)$ is the simplicial set of the ordered complex $\Delta(X)$ and the two realizations agree; the homology statement follows from (ii).
\end{proof}

\begin{lemma}[Cone points]\label{lem:cone-point}
Let $X$ be a finite poset with a least or a greatest element. Then $\Delta(X)$ is a cone and therefore contractible, so $\widetilde H_\ast(N(X);k)=0$ and the augmented simplicial chain complex of $N(X)$ is acyclic.
\end{lemma}

\begin{proof}
If $x_0$ is least then $C\cup\{x_0\}$ is a chain for every chain $C$, so $\Delta(X)$ is the cone with apex $x_0$ over the order complex of $X\setminus\{x_0\}$; equivalently, $X$ has an initial object and its nerve is contractible \cite[Lem.~8.5.3]{Riehl}. A greatest element is a least element of $X^{op}$, and $\Delta(X)=\Delta(X^{op})$. The two homological consequences are Lemma \ref{lem:order-nerve}(ii) and (iv), the augmented complex of $N(X)$ being $\widetilde C_\ast(\Delta(X);k)$.
\end{proof}

Lemma \ref{lem:terminal-bar} is the sharper statement needed when the coefficients are not constant.

The letter $\Delta$ has four roles, fixed by what it is applied to: the window $\Delta_a^{\Pos}$, the order complex $\Delta(X)$, the standard simplex $\Delta^n$ and the simplex category $\mathbf\Delta$. A finite simplicial complex is written $\Sigma$. The letter $N$ has one role, the nerve of Definition \ref{def:order-complex}, applied always to a poset or a small category; the normalization of a simplicial object, which many sources also write $N$, is $\mathcal N$ throughout (Definition \ref{def:normalized-complex}). Part \textup{(ii)} above composes the two.

\subsection{Face posets, subdivision, and simplicial posets}\label{ssec:face-posets}

The order complex $\Delta(-)$ and the face poset $\Pos(-)$ are functors between finite posets with monotone maps and finite (abstract) simplicial complexes with simplicial maps, and neither composite is the identity. One way, $\Delta(\Pos(\Sigma))$ is the barycentric subdivision of $\Sigma$, and $\lvert\Sigma\rvert\cong\lvert\Delta(\Pos(\Sigma))\rvert$ \cite[\S1.1]{Wachs}. The other, $\Pos\Delta(X)$ replaces a chain order by an inclusion order (for instance $\{a<b\}\overset{\Pos\Delta}{\mapsto}\big\{\{a\}<\{a,b\}>\{b\}\big\}$), and so forgets the vertex order.

Bj\"orner's criterion \cite[Def.~2.1, Prop.~3.1]{Bjorner} characterises the face posets of regular CW complexes, augmented by a least element $\hat0$, as the \emph{CW posets}, those with more than one element in which the open interval $(\hat0,x)$ has order complex homeomorphic to a sphere for every $x\ne\hat0$; regularity supplies that sphere, so a cell of positive dimension has at least two proper faces, and a vertex $v$ of a $2$-cell $t$ lies on exactly two of the edges decomposing the circle $\partial t$, making $(v,t)$ a two-element set: the \emph{diamond property}, Bj\"orner's \emph{thinness} \cite[\S4]{Bjorner}, used at Proposition \ref{prop:diamond-nogo}. Chains of length $\ge1$, a diamond interval read as a poset in its own right (Remark \ref{rem:diamond-totality}), and the rank-$3$ lattice of Example \ref{ex:essential-nonsingleton} are graded and are not face posets.

\begin{example}[The smallest regular cell complex with a $2$-cell]\label{ex:bigon}
Let $K$ be the closed disk presented with two vertices $u,v$, two edges $e_1,e_2$, each with the endpoints $u$ and $v$, and one $2$-cell $t$ attached along the circle $e_1\cup e_2$. Each closed cell is an embedded ball, so $K$ is a finite regular cell complex; it is not a simplicial complex, the two edges sharing both endpoints. Its face poset augmented by $\hat0$ is
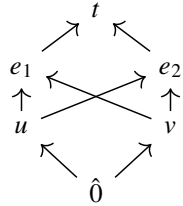
\[
\begin{tikzcd}[row sep=small,column sep=small]
& t & \\
e_1\arrow[ur] & & e_2\arrow[ul] \\
u\arrow[u]\arrow[urr] & & v\arrow[ull]\arrow[u] \\
& \hat0\arrow[ul]\arrow[ur] &
\end{tikzcd}
\]
with $\rk=\dim+1$ on $\Pos(K)$ and $\rk(\hat0)=0$. The criterion is read off level by level. The intervals $(\hat0,u)$ and $(\hat0,v)$ are empty, the $(-1)$-sphere. Each $(\hat0,e_i)=\{u,v\}$ is a two-element antichain, whose order complex is a pair of points, $S^0$. The strict chains of $(\hat0,t)=\{u,v,e_1,e_2\}$ are the four singletons and the four pairs $u<e_i$, $v<e_i$, so $\Delta\bigl((\hat0,t)\bigr)$ is the cycle
\[
   u<e_1>v<e_2>u ,
\]
a subdivision of the circle $\partial t$ and homeomorphic to $S^1$.

Five nonempty cells is the least a regular cell complex of dimension $\ge2$ can have: the boundary of a $2$-cell is a regular cell complex homeomorphic to a circle, and one vertex with one edge does not qualify, the interval below that edge being a single point rather than $S^0$, so two vertices and two edges are already needed to bound $t$. The interval $(u,t)=\{e_1,e_2\}$ carries the diamond property, and $K$ is therefore the smallest complex to which the last clause of Proposition \ref{prop:diamond-nogo} applies. Both vertices are minimal in $\Pos(K)$, so $u\meet v$ does not exist there.
\end{example}

\begin{definition}[Simplicial poset]\label{def:simplicial-poset}
A finite poset with least element $\hat0$ is a \emph{simplicial poset} if every principal down-set is a \emph{Boolean} lattice, that is isomorphic to the lattice of subsets of a finite set. This is Bj\"orner's \emph{poset of Boolean type} \cite[\S2.3]{Bjorner}.
\end{definition}

Every $\widehat\Pos(\Sigma)$ is a simplicial poset, and is the general case under a meet hypothesis: a simplicial poset is an augmented face poset of a simplicial complex if and only if it is a meet-semilattice, and in any case it is a CW poset \cite[\S2.3]{Bjorner}.

The tame posets of Section \ref{sec:metric} lie outside this class, a principal down-set that is a chain being Boolean only up to two elements.

\section{Deformations, incidence algebras, and bar complexes}\label{app:bar}

\subsection{Left deformations and dimension shifting}

\begin{remark}[Deformations: the two descriptions of a left derived functor]\label{rem:deformation}
Let $T$ be a right exact additive functor on a category of sheaves over a finite poset; the instances used are $\Lan_{q_1^\Delta}$ on $\Shv(\Delta_a^{\Pos};k)$ and the colimit functors $\colim_I$ of Lemma \ref{lem:bar-derived-colim}. Two descriptions of its left derived functor are used. Concretely, $\mathbb LT(X)=T(QX)$ for a projective resolution $QX\to X$. Abstractly, the total left derived functor is a right Kan extension along the localization $\gamma\colon\Chb(\mathcal A)\to\Db(\mathcal A)$ \cite[Def.~2.1.17]{Riehl}. The bridge is a left deformation \cite[Def.~2.2.1, Def.~2.2.4]{Riehl}: a functorial projective replacement $q\colon Q\Rightarrow\id$ on whose image the functor to be derived is homotopical. In $\mathcal A\simeq\mathrm{Mod}\text{-}k\Pos$, standard projective resolutions \cite[Ch.~10]{Weibel} and Lemma \ref{lem:enough-projectives} supply such $Q$; finite global dimension of $k\Pos$, at most the length of $\Pos$ by \cite[Prop.~2.6]{Ladkani} over $\mathrm{Vec}_k$, bounds representatives in $\Db(\mathcal A)$. The same holds over any finite poset, in particular over $\Delta_a^{\Pos}$ and over the index posets of this appendix.

An additive functor preserves chain homotopy equivalences, and a quasi-isomorphism between bounded-below complexes of projectives is a chain homotopy equivalence \cite[Lem.~10.4.6, Cor.~10.4.7]{Weibel}, dual to the injective form stated there. Such a $T$ is therefore left deformable, $TQ$ gives the classical left derived functor in Riehl's deformation formalism \cite[\S2.3, Rem.~2.3.2, Thm.~2.2.8]{Riehl}, and the resulting total derived functor is pointwise \cite[Prop.~2.2.13]{Riehl}. Right exactness supplies none of that; it is instead why deriving on the left is the correction that fits, $H_0\LC_aF\cong C_aF$ for a sheaf $F$ (Proposition \ref{prop:bar}) \cite[Rem.~2.3.1]{Riehl}, and Proposition \ref{prop:inexact} is why the correction is not vacuous. The adjunctions of Proposition \ref{prop:derived-exist}\textup{(iii)} are obtained the same way, the two restrictions being exact \cite[Thm.~2.2.11]{Riehl}.

No Quillen-adjunction assertion is made.\footnote{Such a formulation would require $C^a$ to preserve epimorphisms: on the face poset of $\Delta^2$ one has $(C^1G)(T)=G(e_{01})\times_{G(T)}G(e_{02})\times_{G(T)}G(e_{12})$, and the epimorphism $k_{\Pos}\twoheadrightarrow k_{\Pos\setminus\{T\}}$ is carried to the diagonal $k\to k^3$.} Nor do we assert that $\LC_a\LC_b$ derives $C_aC_b$; a composite of total left derived functors need not derive the composite \cite[Rem.~2.2.10]{Riehl}. Proposition \ref{prop:groth} supplies the single-homotopy-colimit model used instead.
\end{remark}

\begin{remark}[Dimension shifting, and the degrees the criterion runs over]\label{rem:dimension-shifting}
The left derived functors of a right-exact additive functor $T$ on a category with enough projectives satisfy the standard dimension-shifting identity; the instance in play is $T=\Lan_{q_1^\Delta}$ on $\Shv(\Delta_a^{\Pos};k)$. Explicitly, for a short exact sequence $0\to M'\to P\to M\to0$ with $P$ projective,
\[
   L_iTM\cong L_{i-1}TM'\quad(i\ge2),
   \qquad
   L_1TM=\ker\bigl(TM'\to TP\bigr)
\]
\cite[Ex.~2.4.3]{Weibel}; this is the long exact sequence of the homological $\delta$-functor $\{L_iT\}$ \cite[\S2.4]{Weibel} together with $L_iTP=0$ for $P$ projective and $i\ge1$.

The hypothesis in force is that $L_1T$ vanish on \emph{every} object, then the induction is immediate: if $L_{i-1}T=0$ identically for some $i\ge2$, then $L_iTM\cong L_{i-1}TM'=0$ for arbitrary $M$. The identity does not transfer across the restriction in $\LC_a=\mathbb L\Lan_{q_1^\Delta}\circ(q_2^\Delta)^\ast$: the objects $(q_2^\Delta)^\ast F$ do not exhaust $\Shv(\Delta_a^{\Pos};k)$, and $(q_2^\Delta)^\ast$ carries projectives to objects that need not be $\Lan_{q_1^\Delta}$-acyclic (Remark \ref{rem:not-derived}). That is why the criterion of Proposition \ref{prop:inexact} runs over all degrees $i\ge1$ and not over degree one alone.
Only the existence of a projective surjection onto each object is used (Lemma \ref{lem:enough-projectives}).
\end{remark}

\subsection{The incidence algebra and its trivial module}

\begin{definition}[Incidence algebra of a finite poset]\label{def:incidence-algebra}
Let $I$ be a finite poset; see \cite[\S2]{Ladkani} for the presentation used here. The \emph{incidence algebra} $kI$ is the $k$-vector space with basis $\{e_{xy}:x\le y\text{ in }I\}$ and product
\[
   e_{ab}\cdot e_{cd}=
   \begin{cases}
      e_{ad}, & b=c,\\
      0, & \text{otherwise;}
   \end{cases}
\]
reading $e_{xy}$ as the unique arrow $x\to y$, this is composition of composable arrows, $e_{xy}\cdot e_{yz}=e_{xz}$. It is associative, and $1=\sum_{x\in I}e_{xx}$ is a unit, the $e_{xx}$ being orthogonal idempotents. Thus $\dim_kkI$ is the number of pairs $x\le y$ in $I$, and $kI=\bigoplus_{x}e_{xx}kI=\bigoplus_x kIe_{xx}$.

The two-sided bar complex of Definition \ref{def:two-sided-bar} pairs a contravariant $R$ with a covariant $D$, that is a left module with a right module (Lemma \ref{lem:module-diagram}), so the tensor product they admit is $D\otimes_{kI}R$ and not the reverse.
Accordingly $B(R,I,D)$ computes $D\otimes^{\mathbb L}_{kI}R$, and the derived colimit of a covariant $D$ is $D\otimes^{\mathbb L}_{kI}k$.

Let $\mathfrak r=\operatorname{span}\{e_{xy}:x<y\}$, a two-sided ideal, and let $m$ be the \emph{length} of $I$, the largest number of strict steps in a chain of $I$. A nonzero product of $p$ basis elements indexed by strict relations is $e_{xy}$ for some chain of $p$ strict steps from $x$ to $y$, so $\mathfrak r^{\,m+1}=0$. The quotient $kI/\mathfrak r$ has the images of the $e_{xx}$ as orthogonal idempotents with $e_{xx}e_{yy}=0$ for $x\ne y$, and so is $\prod_{x\in I}k$, which is semisimple. So $\mathfrak r$ is the Jacobson radical, and $kI$ is a finite-dimensional split basic algebra.
\end{definition}

\begin{lemma}[Modules are diagrams]\label{lem:module-diagram}
Let $I$ be a finite poset. For a right $kI$-module $M$ put $M_x=Me_{xx}$, with $M_x\to M_y$ right multiplication by $e_{xy}$ for $x\le y$; for a covariant $D\colon I\to\mathrm{Vec}_k$ put $\widetilde D=\bigoplus_{x\in I}D(x)$, with $e_{xy}$ acting as the composite of the projection to $D(x)$, the structure map $D(x\le y)$ and the inclusion of $D(y)$. These assignments are mutually inverse equivalences
\[
\mathrm{Mod}\text{-}kI\;\simeq\;\mathrm{Fun}(I,\mathrm{Vec}_k),
\]
and dually, with $M_x=e_{xx}M$, left $kI$-modules are the contravariant functors.
\end{lemma}

\begin{proof}
Let $M$ be a right $kI$-module. Since $1=\sum_{x\in I}e_{xx}$ is a finite sum of orthogonal idempotents, every $m\in M$ satisfies $m=m1=\sum_xme_{xx}$, and if $\sum_xm_x=0$ with $m_x=m_xe_{xx}$ then right multiplication by $e_{yy}$ gives $m_y=0$; so $M=\bigoplus_xM_x$, a finite direct sum, whatever the dimensions of the $M_x$.
From $e_{xx}e_{xy}=e_{xy}=e_{xy}e_{yy}$ right multiplication by $e_{xy}$ carries $M_x$ into $M_y$; it is the identity of $M_x$ when $x=y$, and $e_{xy}e_{yz}=e_{xz}$ makes the maps composable, so $M$ determines a covariant diagram. Conversely $\widetilde D$ is a right $kI$-module: the displayed action of $e_{xy}$ on $e_{yz}$ composes to that of $e_{xz}$ by functoriality of $D$ and vanishes on the other summands, and $\sum_xe_{xx}$ acts as the identity.

Evaluating the first on $\widetilde D$ gives $\widetilde De_{xx}=D(x)$ with $D(x\le y)$ as structure maps, and applying the second to the diagram of $M$ gives $\bigoplus_xM_x=M$ by the splitting above, both identifications being natural. A homomorphism $\varphi\colon M\to N$ has $\varphi(Me_{xx})\subseteq Ne_{xx}$ and commutes with right multiplication by $e_{xy}$, so it is a natural transformation of the associated diagrams, and conversely. The dual statement is this one applied to $kI^{op}=k(I^{op})$.
\end{proof}

Sheaves are covariant, so $\Shv(\Pos;k)\simeq\mathrm{Mod}\text{-}k\Pos$ at $I=\Pos$, which is the identification of Remark \ref{rem:module-presentation}.

\begin{lemma}[The trivial module and its bar resolution]\label{lem:incidence-trivial}
Let $I$ be a finite poset of length $m$, write $k_I$ for the \emph{trivial} left $kI$-module,
and set $I_{\le x}=\{y\in I:y\le x\}$. Then:
\begin{enumerate}
\item[\textup{(i)}] as a left $kI$-module, $kIe_{xx}$ is the constant contravariant diagram on $I_{\le x}$: it has $(kIe_{xx})(y)=k\cdot e_{yx}$ for $y\le x$ and $0$ otherwise, with identity structure maps. In particular $\dim_kkIe_{xx}=|I_{\le x}|$, and $kIe_{xx}$ is projective, being a direct summand of $kI$;
\item[\textup{(ii)}] put
\[
   P_p=\bigoplus_{y_0<\cdots<y_p\ \text{in}\ I}kIe_{y_0y_0},
   \qquad
   \partial=\sum_{i=0}^p(-1)^id_i,
\]
where for $i\ge1$ the face $d_i$ deletes $y_i$ and is the identity on the displayed summand, and $d_0$ deletes $y_0$ and is right multiplication by $e_{y_0y_1}\in e_{y_0y_0}kIe_{y_1y_1}$. With the augmentation $P_0=\bigoplus_ykIe_{yy}\to k_I$ that is the identity on each value, $P_\bullet\to k_I$ is a projective resolution, and $P_p=0$ for $p>m$;
\item[\textup{(iii)}] for a covariant $I$-diagram $M$ with values in $\Chb(\mathrm{Vec}_k)$, that is a right $kI$-module, there is an equality of complexes
\[
   M\otimes_{kI}P_\bullet=\operatorname{Bar}(I,M),
\]
the right-hand side as in Definition \ref{def:bar-hocolim}; so $\operatorname{Tor}^{kI}_p(M,k)\cong H_p\bigl(\operatorname{Bar}(I,M)\bigr)$, which for $M$ concentrated in internal degree zero is $\operatorname{Tor}$ in the ordinary sense and in general the hyper-$\operatorname{Tor}$ of the complex. The identification of this graded object with $\mathbb L_\ast\!\colim_IM$ is Lemma \ref{lem:bar-derived-colim}\textup{(ii)}.
\end{enumerate}
The resolution in \textup{(ii)} is classical, indexed by the chains of $I$ as the order complex is \cite[Prop.~1.3]{Cibils}; it is reproved here in the normalization and sign conventions of Definition \ref{def:bar-hocolim}, which is what part \textup{(iii)} needs.
\end{lemma}

\begin{proof}
(i) By the product rule $e_{ab}e_{xx}$ is $e_{ax}$ when $b=x$ and $0$ otherwise, so $kIe_{xx}=\operatorname{span}\{e_{ax}:a\le x\}$; and $e_{yy}e_{ax}$ is $e_{yx}$ when $a=y$ and $0$ otherwise, so $(kIe_{xx})(y)=e_{yy}kIe_{xx}$ is $k\cdot e_{yx}$ for $y\le x$ and $0$ otherwise. For $y'\le y\le x$ the structure map is left multiplication by $e_{y'y}$, and $e_{y'y}e_{yx}=e_{y'x}$, so it is the identity of $k$ on basis vectors.

(ii) Each $P_p$ is a direct sum of summands of $kI$ and so projective, and vanishes for $p>m$ because $I$ has no strict chain with more than $m$ steps. Exactness may be checked after evaluating at each $z\in I$, the evaluation $M\mapsto M(z)=e_{zz}M$ of a left module, in the notation of Lemma \ref{lem:module-diagram}, being exact. By (i), $(kIe_{y_0y_0})(z)$ is $k$ when $z\le y_0$ and $0$ otherwise, so
\[
   P_p(z)=\bigoplus_{z\le y_0<\cdots<y_p}k
\]
is the group of simplicial $p$-chains of the nerve of the up-set $I_{\ge z}=\{y\in I:y\ge z\}$. Under this identification $d_i$ deletes $y_i$ for $i\ge1$, and $d_0$ sends the generator $e_{zy_0}$ to $e_{zy_0}e_{y_0y_1}=e_{zy_1}$, so it too is the face deleting $y_0$; the augmentation is the usual one. Thus $P_\bullet(z)\to k$ is the augmented simplicial chain complex of $N(I_{\ge z})$, which is acyclic by Lemma \ref{lem:cone-point}, $I_{\ge z}$ having least element $z$. Thus $P_\bullet\to k_I$ is exact at every object, and so exact.

(iii) Since $M\otimes_{kI}kIe_{xx}=Me_{xx}=M_x$,
\[
   M\otimes_{kI}P_p=\bigoplus_{y_0<\cdots<y_p}M_{y_0},
\]
which is $\operatorname{Bar}_p(I,M)$. Under this identification $d_i$ for $i\ge1$ is the identity on the coefficient, and $d_0$ is $\id\otimes\,e_{y_0y_1}$, that is the structure map $M_{y_0}\to M_{y_1}$; these are the faces of Definition \ref{def:bar-hocolim}, and the internal differential of $M$ enters with the sign recorded there. The $\operatorname{Tor}$ identification is then (ii).
\end{proof}

\subsection{Colimits, normalized complexes, and algebraic realization}

\begin{lemma}[Colimits by coproducts and coequalizers, after Mac Lane]\label{lem:colim-coeq}
Let $I$ be a small category and $D\colon I\to\mathcal C$ a functor into a category with coequalizers of pairs of arrows and coproducts indexed by $\operatorname{obj}I$ and by $\operatorname{arr}I$. Write $\iota$ for the coproduct injections and let
\[
   f,g\colon\coprod_{u\colon j\to k}D(j)\longrightarrow\coprod_{j\in\operatorname{obj}I}D(j),
   \qquad
   f\circ\iota_u=\iota_j,
   \quad
   g\circ\iota_u=\iota_k\circ D(u).
\]
Then $\colim_ID$ exists and is their coequalizer $c$, with colimiting cocone $\lambda_j=c\circ\iota_j$. This is the dual of \cite[Thm.~V.2.1, Thm.~V.2.2]{ML}.
\end{lemma}

If $\mathcal C$ is additive the coequalizer is $\operatorname{coker}(g-f)$; if $I$ is a poset the left coproduct is indexed by the relations $j\le k$, the identities among them contributing nothing. Both specialisations are used at $I=J_a(\sigma)$ in the proof of Proposition \ref{prop:bar}, where the bar complex is normalized and so has its degree one indexed by the strict relations alone.

\begin{definition}[Normalized complex]\label{def:normalized-complex}
We follow \cite[\S8.3]{Weibel}. Let $\mathcal{A}$ be an abelian category and let $X_\bullet$ be a simplicial object of $\mathcal{A}$. Its \emph{Moore complex}, or unnormalized chain complex, is the object $\mathcal M(X)$ of $\Ch_{\ge0}(\mathcal{A})$ with $\mathcal M_n(X)=X_n$ and simplicial differential
\[
   \partial=\sum_{i=0}^n(-1)^i d_i ;
\]
the simplicial identities give $\partial^2=0$. Let
\[
   S_n(X)=\sum_{i=0}^{n-1}\operatorname{im}\bigl(s_i\colon X_{n-1}\to X_n\bigr)
\]
be the degenerate subobject. The simplicial identities make $S_\bullet(X)$ a subcomplex, i.e.
\[
\partial s_j=\sum^{n-2}_{i=0,i<j}(-1)^is_{j-1}d_i+\sum^n_{i=2,i>j+1}(-1)^is_j d_{i-1}:X_{n-1}\to S_{n-1}(X).
\]
 The \emph{normalized complex} is
\[
   \mathcal N_n(X)=\mathcal M_n(X)/S_n(X),
\]
again an object of $\Ch_{\ge0}(\mathcal{A})$, with the induced differential. Equivalently,
\[
   \mathcal N_n(X)\cong \bigcap_{i=0}^{n-1}\ker\bigl(d_i\colon X_n\to X_{n-1}\bigr),
\]
with simplicial differential $(-1)^n d_n$ under the decomposition $\mathcal M_\bullet(X)=\mathcal N_\bullet(X)\oplus S_\bullet(X)$ \cite[\S8.3, Lemma 7]{Weibel}. Weibel writes $N$ for this functor; it is written $\mathcal N$ here because $N$ is the nerve of Definition \ref{def:order-complex}, and the two meet in Lemma \ref{lem:order-nerve}\textup{(ii)}. A map $f\colon X_\bullet\to Y_\bullet$ of simplicial objects commutes with the faces, so it is a chain map $\mathcal M(X)\to\mathcal M(Y)$, and it commutes with the degeneracies, so it carries $S_\bullet(X)$ into $S_\bullet(Y)$. Both $\mathcal M$ and $\mathcal N$ are therefore functorial, the map induced on $\mathcal N$ being the one between the quotients.

Two instances occur below. For $\mathcal{A}=\mathrm{Vec}_k$ this is the normalized chain complex of a simplicial vector space, applied in Lemma \ref{lem:algebraic-realization} to the free simplicial vector space on the standard simplices. For $\mathcal{A}=\Chb(\mathrm{Vec}_k)$ an object of $\Ch_{\ge0}(\mathcal{A})$ is a bicomplex, the simplicial direction commuting with the internal differential of the entries; the \emph{normalized total complex}
\[
\bigl(\operatorname{Tot}\mathcal N(X_\bullet)\bigr)_n:=\bigoplus_{p+q=n}\mathcal N_p(X)_q,
\]
with $p$ the simplicial degree and $q$ the internal one, is its total complex, formed with the Koszul sign fixed in Definition \ref{def:bar-hocolim}, where quotienting by degenerate simplices is represented by the displayed direct sum over strict chains $i_0<\cdots<i_n$.
\end{definition}

\begin{lemma}[Algebraic realization of a simplicial object]\label{lem:algebraic-realization}
Write $\mathbf\Delta$ for the simplex category (finite nonempty ordinals and monotone maps), as in Appendix \ref{app:combinatorics}. Let $\mathcal N_\ast(\Delta^n;k)$ denote the normalized simplicial chain complex of the standard $n$-simplex, i.e.\ $\mathcal N$ applied to the free simplicial $k$-vector space $k[\mathbf\Delta(-,[n])]$; it is a cosimplicial object of $\Ch_{\ge0}(\mathrm{Vec}_k)$. Let $X_\bullet$ be a simplicial object of $\Chb(\mathrm{Vec}_k)$. Then there is an isomorphism of chain complexes, natural in $X_\bullet$,
\[
   \int^{[n]\in\mathbf\Delta}\mathcal N_\ast(\Delta^n;k)\otimes_kX_n
   \;\cong\;
   \operatorname{Tot}\mathcal N(X_\bullet),
\]
the coend being formed in $\Ch(\mathrm{Vec}_k)$, $\otimes_k$ the tensor product of complexes with its Koszul sign, and the right-hand side the normalized total complex of Definition \ref{def:normalized-complex}. In particular, for $X_\bullet=\srep(D)$ with $D\colon I\to\Chb(\mathrm{Vec}_k)$ on a finite poset $I$, the left-hand side is $\operatorname{Bar}(I,D)$.

The simplicially enriched counterpart of the displayed formula is \cite[Thm.~6.6.1]{Riehl}, and the identification of realization with totalization under Dold--Kan is \cite[\S2.1]{Arakawa}; the proof below is what fixes the sign convention of Definition \ref{def:bar-hocolim}.
\end{lemma}

\begin{proof}
Both sides are computed one internal degree at a time. Colimits in $\Ch(\mathrm{Vec}_k)$ are formed degreewise, and for each fixed $n$ the complex $\mathcal N_\ast(\Delta^n;k)\otimes_kX_n$ has degree-$d$ term $\bigoplus_{p+q=d}\mathcal N_p(\Delta^n;k)\otimes_k(X_n)_q$. Since a coend is a colimit and commutes with the direct sum, the degree-$d$ term of the left-hand side is
\[
   \bigoplus_{p+q=d}\ \int^{[n]\in\mathbf\Delta}\mathcal N_p(\Delta^n;k)\otimes_k(X_n)_q ,
\]
so it suffices to prove the isomorphism for a simplicial $k$-vector space $Y_\bullet$ (here $Y_\bullet=(X_\bullet)_q$ for fixed $q$), in the form $\int^{[n]}\mathcal N_\ast(\Delta^n;k)\otimes_kY_n\cong \mathcal N(Y_\bullet)$.

For that, the co-Yoneda (density) formula \cite[Thm.~6.5.7]{RiehlCTIC}, in its $\mathrm{Vec}_k$-enriched form, presents $Y_\bullet$ as the coend
\[
   Y_\bullet\;\cong\;\int^{[n]\in\mathbf\Delta}k[\mathbf\Delta(-,[n])]\otimes_kY_n
\]
in simplicial $k$-vector spaces. The normalization functor $\mathcal N$ is one half of the Dold--Kan equivalence $\mathrm{sVec}_k\simeq\Ch_{\ge0}(\mathrm{Vec}_k)$ \cite[\S8.4]{Weibel}, and so preserves all colimits, in particular coends and the tensorings by $k$-vector spaces appearing in them. Applying $\mathcal N$ to the display and using $\mathcal N(k[\mathbf\Delta(-,[n])])=\mathcal N_\ast(\Delta^n;k)$ gives the claim.

Reassembling over $p+q=d$ recovers the total complex, and the sign attached to the two directions is the tensor-product sign of $\mathcal N_\ast(\Delta^n;k)\otimes_kX_n$, which is the sign $(-1)^m$ recorded in Definition \ref{def:bar-hocolim}. Finally, $\mathcal N(\srep(D))$ in bar degree $n$ is the sum over strict chains $i_0<\cdots<i_n$ of $D(i_0)$, which is $\operatorname{Bar}(I,D)$ by definition.
\end{proof}

\subsection{The two-sided bar complex}

\begin{definition}[Two-sided bar differential]\label{def:two-sided-bar}
Let $J$ be a finite poset, let $R\colon J^{op}\to\Chb(\mathrm{Vec}_k)$ be a contravariant $J$-diagram, and let $D\colon J\to\Chb(\mathrm{Vec}_k)$ be a covariant one; over the incidence algebra these are respectively a left and a right $kJ$-module (Definition \ref{def:incidence-algebra}), so the pairing they admit is $D\otimes_{kJ}R$. The two-sided bar complex
\[
   B(R,J,D)
\]
is the normalized total complex of the simplicial object with
\[
   B_n(R,J,D)=
   \bigoplus_{j_0\le\cdots\le j_n}R(j_n)\otimes D(j_0).
\]
Equivalently, in normalized bar degree $p$,
\[
   B_p(R,J,D)=
   \bigoplus_{j_0<\cdots<j_p}R(j_p)\otimes D(j_0).
\]
If $R=k$, the constant contravariant $J$-diagram with value $k$ in internal degree zero and identity structure maps, the factor $R(j_p)$ contributes nothing and what remains is the one-sided complex of Definition \ref{def:bar-hocolim}:
\[
   \operatorname{Bar}(J,D)=B(k,J,D).
\]
The two-argument $\operatorname{Bar}$ is the abbreviation for a constant first variable, and is the notation used wherever the object computed is $\mathbb L\!\colim_JD$; the three-argument $B$ is used where the first variable varies, and only there.
For $p>0$ and a homogeneous generator $r\otimes x\in R(j_p)_{|r|}\otimes D(j_0)_{|x|}$ indexed by $j_0<\cdots<j_p$, the bar faces are
\[
\begin{aligned}
   d_0(r\otimes x)
      &=r\otimes D(j_0\le j_1)x & \text{indexed by } j_1<\cdots<j_p,\\
   d_i(r\otimes x)
      &=r\otimes x \quad (0<i<p) & \text{indexed by } j_0<\cdots<\widehat{j_i}<\cdots<j_p,\\
   d_p(r\otimes x)
      &=R(j_{p-1}\le j_p)r\otimes x & \text{indexed by } j_0<\cdots<j_{p-1}.
\end{aligned}
\]
The bar differential is
\[
   \partial_{\mathrm{bar}}=\sum_{i=0}^p(-1)^i d_i .
\]
In total degree we use
\[
   d_{\mathrm{tot}}
   =
   d_R\otimes 1+(-1)^{|r|}1\otimes d_D
   +(-1)^{|r|+|x|}\partial_{\mathrm{bar}}
   \;=\;\delta_R+\delta_D+\delta_B,
\]
naming its three parts
\[
\begin{aligned}
   \delta_R(r\otimes x)&=d_Rr\otimes x,\\
   \delta_D(r\otimes x)&=(-1)^{|r|}r\otimes d_Dx,\\
   \delta_B(r\otimes x)&=(-1)^{|r|+|x|}\partial_{\mathrm{bar}}(r\otimes x).
\end{aligned}
\]
For $p=0$ the bar part is absent. This is the ordinary total-complex differential of the multicomplex underlying the simplicial bar object, whose three directions are the internal differentials of $R$ and of $D$ together with the simplicial bar direction; for the totalization convention see McCleary \cite[\S2.4, pp.~47--48]{McCleary}, and for the normalized alternating bar differential Weibel \cite[\S6.5, pp.~177--178]{Weibel}. That $d_{\mathrm{tot}}^2=0$ is the next lemma.
\end{definition}

\begin{lemma}[The total bar differential squares to zero]\label{lem:bar-total-differential}
In the notation of Definition \ref{def:two-sided-bar}, $d_{\mathrm{tot}}^2=0$ on $B(R,J,D)$.
\end{lemma}

\begin{proof}
It suffices that
\[
   \delta_R^2=\delta_D^2=\delta_B^2=0
\]
and that the three pairs anticommute:
\[
   \delta_R\delta_D+\delta_D\delta_R=0,\qquad
   \delta_R\delta_B+\delta_B\delta_R=0,\qquad
   \delta_D\delta_B+\delta_B\delta_D=0.
\]
The first relation is the ordinary tensor-product sign rule:
\[
   \delta_R\delta_D(r\otimes x)
      =(-1)^{|r|}d_Rr\otimes d_Dx,\qquad
   \delta_D\delta_R(r\otimes x)
      =(-1)^{|r|-1}d_Rr\otimes d_Dx .
\]
For the two bar interactions, fix the strict chain
$\mathbf j=(j_0<\cdots<j_p)$ and write
$\partial_i\mathbf j=(j^{(i)}_0<\cdots<j^{(i)}_{p-1})$ for the chain obtained
by deleting $j_i$. Let
\[
   \alpha_i^{\mathbf j}\colon R(j_p)\otimes D(j_0)\longrightarrow
   R(j^{(i)}_{p-1})\otimes D(j^{(i)}_0)\hookrightarrow B_{p-1}(R,J,D)
\]
denote the $i$-th summand map from the $\mathbf j$-summand to the
$\partial_i\mathbf j$-summand:
\[
\begin{aligned}
   \alpha_0^{\mathbf j}&=1\otimes D(j_0\le j_1),\\
   \alpha_i^{\mathbf j}&=1\otimes1\qquad(0<i<p),\\
   \alpha_p^{\mathbf j}&=R(j_{p-1}\le j_p)\otimes1.
\end{aligned}
\]
Thus
$\partial_{\mathrm{bar}}|_{R(j_p)\otimes D(j_0)}
 =\sum_i(-1)^i\alpha_i^{\mathbf j}$, with the direct-sum inclusion of each
target summand understood, and the maps $\alpha_i^{\mathbf j}$ do not affect
the internal chain grading of the tensor factors.

Each $\alpha_i^{\mathbf j}$ commutes with both internal differentials. Only one
face acts non-trivially on each tensor factor (the last on $R$, the first on
$D$), and there the statements used are the chain-map identities
\[
   d_{R(j_{p-1})}\bigl(R(j_{p-1}\le j_p)r\bigr)
      =R(j_{p-1}\le j_p)d_{R(j_p)}r,
   \qquad
   d_{D(j_1)}\bigl(D(j_0\le j_1)x\bigr)
      =D(j_0\le j_1)d_{D(j_0)}x ;
\]
for every other $i$ the factor in question is left unchanged by
$\alpha_i^{\mathbf j}$ and the comparison is trivial, so
\[
\begin{aligned}
   \delta_R\delta_B(r\otimes x)
      &=(-1)^{|r|+|x|}\sum_i(-1)^i\alpha_i^{\mathbf j}(d_{R(j_p)}r\otimes x),\\
   \delta_B\delta_R(r\otimes x)
      &=(-1)^{|r|-1+|x|}\sum_i(-1)^i\alpha_i^{\mathbf j}(d_{R(j_p)}r\otimes x),\\
   \delta_D\delta_B(r\otimes x)
      &=(-1)^{|x|}\sum_i(-1)^i\alpha_i^{\mathbf j}(r\otimes d_{D(j_0)}x),\\
   \delta_B\delta_D(r\otimes x)
      &=(-1)^{|x|-1}\sum_i(-1)^i\alpha_i^{\mathbf j}(r\otimes d_{D(j_0)}x),
\end{aligned}
\]
and each pair cancels, because applying the internal differential lowers the
degree carried in the sign of $\delta_B$ by one before $\delta_B$ is applied.
Finally $\delta_B$ carries the sign $(-1)^{|r|+|x|}$, which the
degree-preserving maps $\alpha_i^{\mathbf j}$ leave inert, so
$\delta_B^2=\partial_{\mathrm{bar}}^2$; and
$\partial_{\mathrm{bar}}^2=0$ is the simplicial identity
$d_id_j=d_{j-1}d_i$ for $i<j-1$ together with the passage to the normalized
quotient \cite[Ex.~6.5.1(1)--(3)]{Weibel}, stated there for the bar
resolution of a group over $\mathbb ZG$ and formal in the face maps.
\end{proof}

\begin{lemma}[Shifted cones in the first bar variable]\label{lem:bar-cone}
Let $J$ be a finite poset, let $D\colon J\to\Chb(\mathrm{Vec}_k)$ be a covariant $J$-diagram, and let $f\colon R\to S$ be a map of contravariant $J$-diagrams in $\Chb(\mathrm{Vec}_k)$, with induced map $B(f,J,D)\colon B(R,J,D)\to B(S,J,D)$. Then the identification of underlying graded objects is an \emph{equality} of complexes
\[
   B\bigl(\operatorname{Cone}(f)[-1],J,D\bigr)
   =
   \operatorname{Cone}\bigl(B(f,J,D)\bigr)[-1] ,
\]
in the conventions of Remark \ref{rem:homotopical} and Definition \ref{def:two-sided-bar}. No sign twist is needed.
\end{lemma}

\begin{proof}
Write $T=\operatorname{Cone}(f)[-1]$. By the shift and cone conventions of Remark \ref{rem:homotopical},
\[
   T(j)_n=R(j)_n\oplus S(j)_{n+1},
   \qquad
   d_T(r,s)=\bigl(d_Rr,\ -f(r)-d_Ss\bigr).
\]
Summing the summands $T(j_p)\otimes D(j_0)$ over strict chains and over bar degrees therefore gives
\[
   B(T,J,D)_n=B(R,J,D)_n\oplus B(S,J,D)_{n+1}
   =\Bigl(\operatorname{Cone}\bigl(B(f,J,D)\bigr)[-1]\Bigr)_n ,
\]
which is the identification of graded objects. Only the differentials remain, and the single point to watch is that an element of $S$-degree $n+1$ has $T$-degree $n$, so the sign $(-1)^{|r|+|x|}$ of Definition \ref{def:two-sided-bar} is computed with $n$ on that summand.

Let $r\otimes x$ have $r\in R(j_p)$, whose $T$-degree equals its $R$-degree. The three parts of $d_{\mathrm{tot}}$ on $B(T,J,D)$ are
\[
   d_Tr\otimes x=d_Rr\otimes x-f(r)\otimes x,
   \qquad
   (-1)^{|r|}r\otimes d_Dx,
   \qquad
   (-1)^{|r|+|x|}\partial_{\mathrm{bar}}(r\otimes x).
\]
The component in the $R$-summand, together with the last two terms, is $d_{B(R,J,D)}(r\otimes x)$; the component in the $S$-summand is $-f(r)\otimes x=-B(f,J,D)(r\otimes x)$.

Let now $s\otimes x$ have $s\in S(j_p)$ of $S$-degree $n+1$ and so of $T$-degree $n$. The three parts read
\[
   -d_Ss\otimes x,
   \qquad
   (-1)^{n}s\otimes d_Dx,
   \qquad
   (-1)^{n+|x|}\partial_{\mathrm{bar}}(s\otimes x),
\]
whereas the corresponding parts of $d_{B(S,J,D)}(s\otimes x)$ are $d_Ss\otimes x$, $(-1)^{n+1}s\otimes d_Dx$ and $(-1)^{n+1+|x|}\partial_{\mathrm{bar}}(s\otimes x)$. Each of the three is negated, so the $S$-summand carries $-d_{B(S,J,D)}$. Altogether $d_{\mathrm{tot}}$ acts on $B(T,J,D)$ by
\[
   (a,b)\longmapsto\bigl(d_{B(R,J,D)}a,\ -B(f,J,D)(a)-d_{B(S,J,D)}b\bigr),
\]
which is the differential of $\operatorname{Cone}(B(f,J,D))[-1]$ in the conventions of Remark \ref{rem:homotopical}.
\end{proof}

\subsection{Terminal objects, change of base, and homological support}

The next two lemmas pass between the bar complex over $\Ka_{a,b}(\sigma)$ and a two-sided bar complex over $J_{a+b}(\sigma)$. Both are stated for a diagram $D$ of vector spaces, that is, concentrated in internal degree zero.

\begin{lemma}[Bar complexes over a poset with terminal object]\label{lem:terminal-bar}
Let $I$ be a finite poset with terminal object $t$ and let $D\colon I\to\mathrm{Vec}_k$. Then the augmentation
\[
   \mathrm{aug}\colon\operatorname{Bar}(I,D)\longrightarrow D(t),
   \qquad
   [i_0]\otimes d\longmapsto D(i_0\le t)d,
\]
zero in bar degrees $\ge1$, is a chain homotopy equivalence, with homotopy inverse the inclusion $\iota$ of the summand indexed by $t$ in bar degree $0$. In particular $\hocolim_ID\simeq D(t)$, concentrated in degree $0$.

The homotopy $s$ constructed below is the extra degeneracy of \cite[Lem.~4.5.1]{Riehl}; \cite[Cor.~2.5]{FM} is the analogue for a diagram of finite posets indexed by a poset with a maximum, where the conclusion is a collapse onto the value there.
\end{lemma}

\begin{proof}
The chain-map condition for $\mathrm{aug}$ reduces to the identity $D(i_1\le t)\,D(i_0\le i_1)=D(i_0\le t)$ applied to the two faces in bar degree $1$, and $\mathrm{aug}\circ\iota=\id_{D(t)}$ is immediate. Define $s\colon\operatorname{Bar}_n(I,D)\to\operatorname{Bar}_{n+1}(I,D)$ by
\[
   s\bigl([i_0<\cdots<i_n]\otimes d\bigr)
   =
   \begin{cases}
      (-1)^{n+1}[i_0<\cdots<i_n<t]\otimes d, & i_n\ne t,\\[2pt]
      0, & i_n=t,
   \end{cases}
\]
which is defined because $i\le t$ for every $i\in I$. We claim $\partial s+s\partial=\id-\iota\circ\mathrm{aug}$.

Let $n\ge1$ and $w=[i_0<\cdots<i_n]\otimes d$, so that $\iota\,\mathrm{aug}(w)=0$. If $i_n\ne t$, the face of $[i_0<\cdots<i_n<t]$ deleting the final entry $t$ carries the sign $(-1)^{n+1}$, which cancels the prefactor of $s$, so
\[
   \partial s(w)
   =(-1)^{n+1}\sum_{m=0}^{n}(-1)^m[i_0<\cdots\widehat{i_m}\cdots<i_n<t]\otimes d\;+\;w .
\]
Each chain appearing in $\partial w$ still has final entry $<t$: for $m<n$ it is $i_n$, and for $m=n$ it is $i_{n-1}<i_n\le t$. So $s$ is given on all of them by the first branch, in bar degree $n-1$, and
\[
   s\partial(w)=(-1)^{n}\sum_{m=0}^{n}(-1)^m[i_0<\cdots\widehat{i_m}\cdots<i_n<t]\otimes d .
\]
The two sums cancel and $\partial s(w)+s\partial(w)=w$. If instead $i_n=t$, then $s(w)=0$, and in $\partial w$ every face with $m<n$ again ends in $t$ and is killed by $s$, while the face $m=n$ contributes $(-1)^n\cdot(-1)^{n}[i_0<\cdots<i_{n-1}<t]\otimes d=w$; so $s\partial(w)=w$ again.

For $n=0$ we have $\partial w=0$. If $w=[i_0]\otimes d$ with $i_0\ne t$, then
\[
   \partial s(w)=-\bigl([t]\otimes D(i_0\le t)d-[i_0]\otimes d\bigr)=w-\iota\,\mathrm{aug}(w),
\]
and if $i_0=t$ then $s(w)=0$ while $w-\iota\,\mathrm{aug}(w)=0$ as well.
\end{proof}

\begin{lemma}[Change of base along a monotone map]\label{lem:change-of-base}
Let $\Phi\colon K\to J$ be a monotone map of finite posets, let $D\colon J\to\mathrm{Vec}_k$, and let
\[
   \Phi_\#\colon\operatorname{Bar}(K,\Phi^\ast D)\longrightarrow\operatorname{Bar}(J,D)
\]
be the induced bar map of Definition \ref{def:bar-hocolim}. Let $R_\Phi(j)=C_\ast\bigl(N(j\downarrow\Phi);k\bigr)$ be the finality-fibre diagram of Lemma \ref{lem:defect-reduced}, with augmentation $\varepsilon\colon R_\Phi\to k$, that is $\varepsilon_j[c]=1$ on each basis element $[c]\in R_\Phi(j)_0$, indexed by the $c\in K$ with $j\le\Phi(c)$, and $\varepsilon_j=0$ in positive internal degree. Write
\[
   \varepsilon_\ast=B(\varepsilon,J,D)\colon B(R_\Phi,J,D)\longrightarrow B(k,J,D)=\operatorname{Bar}(J,D).
\]
Then:
\begin{enumerate}
\item[\textup{(i)}] $B(R_\Phi,J,D)$ is the total complex of the double complex
\[
   W_{p,q}=\bigoplus_{\substack{j_0<\cdots<j_p\ \text{in}\ J\\ c_0<\cdots<c_q\ \text{in}\ K\\ j_p\le\Phi(c_0)}}D(j_0),
\]
with $p$ the bar degree in $J$ and $q$ the internal degree of $R_\Phi$, that is the chain degree in $K$. Its $p$-differential is $(-1)^q$ times the bar differential of $J$, and its $q$-differential is the simplicial differential of the nerve. A generator of the summand indexed by $j_0<\cdots<j_p$ and $c_0<\cdots<c_q$ is written
\[
   [j_0<\cdots<j_p\mid c_0<\cdots<c_q]\otimes x,
   \qquad j_p\le\Phi(c_0),\quad x\in D(j_0);
\]
it is the generator $r\otimes x\in R_\Phi(j_p)_q\otimes D(j_0)_0$ of Definition \ref{def:two-sided-bar}, with $r=[c_0<\cdots<c_q]$ the corresponding $q$-chain of $N(j_p\downarrow\Phi)$.
\item[\textup{(ii)}] Define the collapse map of total complexes $\pi\colon B(R_\Phi,J,D)\longrightarrow\operatorname{Bar}(K,\Phi^\ast D)$ by
\[
   \pi\bigl([j_0\mid c_0<\cdots<c_q]\otimes x\bigr) = \begin{cases}
   [c_0<\cdots<c_q]\otimes D\bigl(j_0\le\Phi(c_0)\bigr)x \quad \text{if }p=0,\\
   0 \quad \text{otherwise},
   \end{cases}
\]
where the target $\operatorname{Bar}(K,\Phi^\ast D)$ is graded by the chain degree in $K$ alone, carrying no $J$-bar direction; moreover the total degree is the chain degree $q$ of the image.
Then $\pi$ is a quasi-isomorphism.
\item[\textup{(iii)}] $\Phi_\#\circ\pi$ and $\varepsilon_\ast$ are chain homotopic.
\end{enumerate}
That is, the triangle
\[
\begin{tikzcd}[column sep=huge,row sep=large]
B(R_\Phi,J,D)\arrow[r,"\pi\;(\simeq)"]\arrow[dr,"\varepsilon_\ast"'] &
\operatorname{Bar}(K,\Phi^\ast D)\arrow[d,"\Phi_\#"]\\
& \operatorname{Bar}(J,D)
\end{tikzcd}
\]
commutes up to chain homotopy, with $\pi$ a quasi-isomorphism. Thus $\Phi_\#$ is a quasi-isomorphism if and only if $\varepsilon_\ast$ is.
\end{lemma}

\begin{proof}
(i) Because $D$ is concentrated in internal degree zero as a complex, the summand of $B_p(R_\Phi,J,D)$ indexed by $j_0<\cdots<j_p$ has internal degree $q$ part $\bigl(R_\Phi(j_p)\otimes D(j_0)\bigr)_q\cong R_\Phi(j_p)_q\otimes D(j_0)$, and $R_\Phi(j_p)_q$ is by definition free on the strict $q$-chains of $N(j_p\downarrow\Phi)$. A chain $c_0<\cdots<c_q$ of $K$ suffices $c_i\in j_p\downarrow\Phi$ exactly when $j_p\le\Phi(c_0)$, since $\Phi$ is monotone. This shows that $B_p(R_\Phi,J,D)_q=W_{p,q}$, with the generators as named. The differentials are those of Definition \ref{def:two-sided-bar} at $|r|=q$ and $|x|=0$: the internal part $\delta_D$ vanishes, $\delta_R$ is the nerve differential
\[
   \delta_R\bigl([j_0<\cdots<j_p\mid c_0<\cdots<c_q]\otimes x\bigr)
   =\sum_{i=0}^{q}(-1)^i\,[j_0<\cdots<j_p\mid c_0<\cdots\widehat{c_i}\cdots<c_q]\otimes x ,
\]
and $\delta_B=(-1)^q\partial_{\mathrm{bar}}$ is the bar differential of $J$ up to that sign, which is the assertion.

Both $\pi$ and $\varepsilon_\ast$ are chain maps. On $W_{p,q}$ the map $\varepsilon_\ast$ is $\varepsilon_{j_p}\otimes 1$, so it annihilates the summands with $q>0$ and sends $[j_0<\cdots<j_p\mid c_0]\otimes x$ to $[j_0<\cdots<j_p]\otimes x$. Against $\delta_R$ both composites vanish, the target carrying the zero internal differential while $\varepsilon$ annihilates nerve boundaries:
\[
   \varepsilon_\ast\delta_R\bigl([j_0<\cdots<j_p\mid c_0<\cdots<c_q]\otimes x\bigr)
   =\sum_{i=0}^{q}(-1)^i\,\varepsilon\bigl([c_0<\cdots\widehat{c_i}\cdots<c_q]\bigr)\otimes x
   =0 ,
\]
where every term being zero for $q\ge2$,
and the two terms cancelling as $1\otimes x-1\otimes x$ for $q=1$; for $q=0$ the differential $\delta_R$ is itself zero. Against $\delta_B$, both composites vanish for $q>0$, and on the surviving summands $q=0$ the sign $(-1)^{|r|+|x|}$ of Definition \ref{def:two-sided-bar} is trivial, so $\varepsilon_\ast\delta_B=\delta_B\varepsilon_\ast$.
For $\pi$, both composites vanish on the summands with $p\ge2$; on $p=0$ and $q\ge1$, we have
\[
\begin{aligned}
d_{\mathrm{tot}}\pi(r\otimes x) &= [c_1<\cdots<c_q]\otimes D(j_0\le \Phi(c_1))x \\
&+ \sum_{i=1}^{q}(-1)^i\,[c_0<\cdots\widehat{c_i}\cdots<c_q]\otimes D(j_0\le \Phi(c_0))x \\
&= \pi d_{\mathrm{tot}}(r\otimes x),
\end{aligned}
\]
the last equality because $\partial_{\mathrm{bar}}$ is absent in bar degree $0$; both sides vanish for $q=0$. On $p=1$ both sides vanish as well, since $\pi$ annihilates the image of $\delta_R$ while $(-1)^q\pi\partial_{\mathrm{bar}}(r\otimes x)$ cancels in the form $[c]\otimes D(\cdot)x-[c]\otimes D(\cdot)x$.

(ii) Let $\widetilde W$ be the augmented double complex with $\widetilde W_{-1,q}=\operatorname{Bar}_q(K,\Phi^\ast D)$, $\widetilde W_{p,q}=W_{p,q}$ for $p\ge0$, and $\pi$ as the augmentation. Filter $\operatorname{Tot}\widetilde W$ by the nerve degree $q$. The filtration is bounded, since $J$ and $K$ are finite. For fixed $q$ the associated graded piece is the direct sum, over strict $q$-chains $c_0<\cdots<c_q$ of $K$, of
\[
   \cdots\longrightarrow\operatorname{Bar}_1(J_{\le\Phi(c_0)},D)\longrightarrow\operatorname{Bar}_0(J_{\le\Phi(c_0)},D)\longrightarrow D(\Phi(c_0))\longrightarrow0,
   \qquad
   J_{\le u}=\{j\in J: j\le u\},
\]
with differential $(-1)^q$ times the one displayed: indeed the constraint $j_p\le\Phi(c_0)$ together with $j_0<\cdots<j_p$ says that the chain lies in $J_{\le\Phi(c_0)}$. Each such complex is acyclic, because $J_{\le\Phi(c_0)}$ has terminal object $\Phi(c_0)$ and Lemma \ref{lem:terminal-bar} makes its augmentation a chain homotopy equivalence. So $E^1=0$, and by convergence of the bounded filtration \cite[\S2.2, Thm.~2.6]{McCleary} the complex $\operatorname{Tot}\widetilde W$ is acyclic. Since $\operatorname{Tot}\widetilde W$ is the shifted cone of $\pi$, this shows (ii).

(iii) Write $\widehat{\operatorname{Bar}}$ and $\widehat W$ for the unnormalized objects, indexed by weakly increasing chains; the normalized objects are their quotients by the subcomplexes spanned by chains with a repeated entry. All of $\pi$, $\varepsilon_\ast$, $\Phi_\#$ are defined on the unnormalized objects and carry repetitions to repetitions, so it suffices to produce the homotopy upstairs. For a generator
\[
   \widehat w=[j_0\le\cdots\le j_p\mid c_0\le\cdots\le c_q]\otimes x,
   \qquad j_p\le\Phi(c_0),\quad x\in D(j_0),
\]
set
\[
   h(\widehat w)=(-1)^{p(q+1)}\,[j_0\le\cdots\le j_p\le\Phi(c_0)\le\cdots\le\Phi(c_q)]\otimes x
   \;\in\;\widehat{\operatorname{Bar}}_{p+q+1}(J,D),
\]
which is a weakly increasing chain because $j_p\le\Phi(c_0)$. We claim
\[
   \partial h+h\,d_{\mathrm{tot}}=\Phi_\#\circ\pi-\varepsilon_\ast .
\]
In $\partial h(\widehat w)$ the faces $m=0,\dots,p$ delete $j_m$ and the faces $m=p+1+l$, $l=0,\dots,q$, delete $\Phi(c_l)$. Deleting $c_l$ before concatenating deletes $\Phi(c_l)$ after concatenating, and likewise for $j_m$, so these faces match the terms of $h(d_{\mathrm{tot}}\widehat w)$ term by term, with two exceptions noted below. The signs cancel: for the $j$-terms,
\[
   (-1)^{p(q+1)}+(-1)^{q}(-1)^{(p-1)(q+1)}=(-1)^{p(q+1)}+(-1)^{p(q+1)-1}=0,
\]
the second summand being the sign carried by $h$ in bidegree $(p-1,q)$ against the factor $(-1)^q$ of $\delta_B$; and for the $c$-terms,
\[
   (-1)^{p(q+1)}(-1)^{p+1}+(-1)^{pq}=(-1)^{pq+1}+(-1)^{pq}=0 .
\]
The two exceptions are the endpoints. When $p=0$ the face $m=0$ deletes $j_0$ and applies $D(j_0\le\Phi(c_0))$, giving $\Phi_\#\pi(\widehat w)$, and $\delta_B$ is absent in bar degree $0$, so nothing cancels it. When $q=0$ the face $m=p+1$ deletes the final entry $\Phi(c_0)$, giving $(-1)^{p}(-1)^{p+1}\varepsilon_\ast(\widehat w)=-\varepsilon_\ast(\widehat w)$, and $\delta_R$ vanishes on nerve degree $0$, so nothing cancels that either. This proves the claim, and the identity descends to the normalized quotients.

The final assertion follows from (ii) and (iii): $\Phi_\#$ is a quasi-isomorphism if and only if $\Phi_\#\pi$ is, and $\Phi_\#\pi$ induces the same map on homology as $\varepsilon_\ast$.
\end{proof}

\begin{lemma}[Restriction to the homological support]\label{lem:support-restriction}
Let $I$ be a finite poset, let $G\colon I\to\Chb(\mathrm{Vec}_k)$, and let $A\subseteq I$ be an \emph{up-closed} full subposet.
\begin{enumerate}
\item[\textup{(i)}] The summands of $\operatorname{Bar}(I,G)$ whose bottom index lies in $A$ are those indexed by chains contained in $A$, and no others; they span a subcomplex, namely $\operatorname{Bar}(A,G|_A)$, and the quotient is spanned by the summands whose bottom index lies in $I\setminus A$. If $G(i)=0$ for every $i\in I\setminus A$, the inclusion
\[
   \operatorname{Bar}(A,G|_A)\longrightarrow\operatorname{Bar}(I,G)
\]
is an isomorphism of complexes. Write $k_A\colon I\to\mathrm{Vec}_k$ for the diagram constant $k$ on $A$ and $0$ elsewhere, a diagram because an ascending chain starting in $A$ stays in $A$. Then
\[
   \operatorname{Bar}(I,k_A)\;=\;C_\ast\bigl(\Delta(A);k\bigr),
\]
and for a monotone $\Psi\colon I'\to I$ one has $\Psi^\ast k_A=k_{\Psi^{-1}A}$ with $\Psi^{-1}A$ up-closed in $I'$, the bar map $\Psi_\#$ being the map $C_\ast(\Delta(\Psi^{-1}A);k)\to C_\ast(\Delta(A);k)$ induced by $\Psi$ on order complexes.
\item[\textup{(ii)}] If $H_\ast G(i)=0$ for every $i\in I\setminus A$, that inclusion is a quasi-isomorphism.
\end{enumerate}
The coefficient of $\operatorname{Bar}(I,G)$ sits at the \emph{bottom} index of a chain (Definition \ref{def:bar-hocolim}), so up-closedness of $A$ is what makes the corresponding summands a subcomplex; both parts fail for down-closed $A$.
\end{lemma}

\begin{proof}
(i) Write $\operatorname{Bar}_p(I,G)=\bigoplus_{i_0<\cdots<i_p}G(i_0)$. If $i_0\in A$ then every $i_j\ge i_0$ lies in $A$, so the summands with $i_0\in A$ are those indexed by chains contained in $A$, and they span $\operatorname{Bar}(A,G|_A)$. This span is a subcomplex: the face $d_0$ deletes $i_0$ and leaves the bottom index $i_1\in A$, and the faces $d_j$ with $j>0$ retain $i_0$. The complementary summands, those with $i_0\in I\setminus A$, therefore span the quotient; and if $G$ vanishes on $I\setminus A$ they are zero, so the inclusion is an isomorphism. For $G=k_A$ the coefficients are constant and the identification with $C_\ast(\Delta(A);k)$ is Lemma \ref{lem:order-nerve}\textup{(iii)}. The last sentence is immediate: $\Psi_\#$ sends a chain to its image, zero if degenerate, with identity coefficient maps.

(ii) Let $Q$ be the quotient, spanned by the summands with $i_0\in I\setminus A$ by (i). Filter $Q$ by bar degree. All faces lower the bar degree, so the induced differential on the graded pieces is the internal differential of the coefficients, and
\[
   E^1_{p,q}
   \cong
   \bigoplus_{\substack{i_0<\cdots<i_p\\ i_0\in I\setminus A}}H_qG(i_0)
   =0
\]
by hypothesis, the coefficient of a summand depending only on $i_0$. The filtration is bounded because $I$ is finite and $G$ takes bounded values, so $Q$ is acyclic \cite[\S2.2, Thm.~2.6]{McCleary} and the inclusion is a quasi-isomorphism.

For the final sentence, take $I=\{i_0<i_1\}$ with $A=\{i_0\}$ down-closed, $G(i_0)=k$ and $G(i_1)=0$. Then $H_\ast G$ vanishes off $A$, but $\operatorname{Bar}(I,G)$ has $G(i_0)$ in bar degrees $0$ and $1$ with the face deleting $i_1$ an isomorphism between them, so it is acyclic, while $\operatorname{Bar}(A,G|_A)=G(i_0)\ne0$.
\end{proof}

\begin{lemma}[Pairs of up-sets, and skyscrapers]\label{lem:bar-pair}
Let $I$ be a finite poset and let $B\subseteq A$ be up-closed subposets. Write $k_{A/B}$ for the diagram equal to $k$ on $A\setminus B$ and $0$ elsewhere, with identity structure maps between nonzero values, and $\sky_x$ for the skyscraper diagram at $x\in I$, equal to $k$ at $x$ and $0$ elsewhere.
\begin{enumerate}
\item[\textup{(i)}] $\operatorname{Bar}(I,k_{A/B})\cong C_\ast(\Delta(A);k)/C_\ast(\Delta(B);k)$, so
\[
   H_n\operatorname{Bar}(I,k_{A/B})\cong H_n\bigl(\Delta(A),\Delta(B);k\bigr).
\]
\item[\textup{(ii)}] For every $x\in I$,
\[
   H_n\operatorname{Bar}(I,\sky_x)\;\cong\;\widetilde H_{n-1}\bigl(\Delta(I_{>x});k\bigr),
   \qquad
   I_{>x}=\{i\in I: i>x\},
\]
with $\widetilde H_{-1}(\Delta(\varnothing);k)=k$.
\end{enumerate}
\end{lemma}

\begin{proof}
(i) The assignment $k_{A/B}$ is a diagram: $A$ being up-closed, no element of $A\setminus B$ maps to one outside $A$, and $B$ is up-closed, so the nonzero values sit at $A\setminus B$ with a possible map to $0$ along $A\setminus B\to B$. It is the cokernel of $k_B\hookrightarrow k_A$, and $\operatorname{Bar}(I,-)$ carries a pointwise short exact sequence of diagrams to a short exact sequence of complexes, being degreewise a direct sum of values; Lemma \ref{lem:support-restriction}\textup{(i)} identifies the two subcomplexes. (ii) Take $A=\st_I(x)$ and $B=I_{>x}$, so that $A\setminus B=\{x\}$ and $k_{A/B}=\sky_x$. The order complex $\Delta(A)$ is contractible, $x$ being least (Lemma \ref{lem:cone-point}), so the long exact sequence of the pair gives $H_n(\Delta(A),\Delta(B))\cong\widetilde H_{n-1}(\Delta(B))$ in every degree.
\end{proof}

\subsection{The bar-degree filtration}

\begin{lemma}[The bar-degree filtration and its spectral sequence]\label{lem:bar-filtration}
Let $J$, $R$, $D$ be as in Definition \ref{def:two-sided-bar} and put $A=B(R,J,D)$. Write
\[
   B_{p,q}
   =
   \bigoplus_{j_0<\cdots<j_p}\bigl(R(j_p)\otimes D(j_0)\bigr)_q,
   \qquad
   A_n=\bigoplus_{p+q=n}B_{p,q},
\]
so that $q$ is the total internal degree of the tensor factors, and filter $A$ by bar degree:
\[
   F_pA_n=\bigoplus_{\substack{0\le i\le p\\ i+q=n}}B_{i,q},
   \qquad F_{-1}A=0 .
\]
This is increasing homological notation; the decreasing reindexing is $\mathcal F^sA=F_{-s}A$, and $\operatorname{gr}^F_pA=F_pA/F_{p-1}A$. Then:
\begin{enumerate}
\item[\textup{(i)}] $d_{\mathrm{tot}}$ preserves $F_\bullet$, because $\delta_R$ and $\delta_D$ preserve bar degree and $\delta_B$ lowers it by one; and
$\operatorname{gr}^F_pA_{p+q}\cong B_{p,q}$, with induced differential $\delta_R+\delta_D$.
\item[\textup{(ii)}] The filtration is bounded in each total degree: $J$ is finite, so $p$ is bounded, and $R,D$ take values in $\Chb(\mathrm{Vec}_k)$.
\item[\textup{(iii)}] Over the field $k$, K\"unneth \cite[\S3.6]{Weibel} gives
\[
   E^1_{p,q}
   \cong
   \bigoplus_{j_0<\cdots<j_p}\ \bigoplus_{q'+q''=q}
   H_{q'}R(j_p)\otimes H_{q''}D(j_0),
\]
with $d^1$ induced by $\partial_{\mathrm{bar}}$. If $D$ is concentrated in internal degree zero (as for $D=(q_2^J)^\ast F$ with $F$ an ordinary sheaf), this reduces to
\[
   E^1_{p,q}\cong\bigoplus_{j_0<\cdots<j_p}H_qR(j_p)\otimes D(j_0).
\]
\item[\textup{(iv)}] The spectral sequence converges:
$E^\infty_{p,q}\cong F_pH_{p+q}(A)/F_{p-1}H_{p+q}(A)$, where $F_pH_n(A)=\operatorname{im}\bigl(H_n(F_pA)\to H_n(A)\bigr)$.
\end{enumerate}
In particular, if $H_\ast R(j)=0$ for every $j\in J$, then $E^1=0$ and $A$ is acyclic.
\end{lemma}

\begin{proof}
(i) and (ii) are immediate from Definition \ref{def:two-sided-bar} and the displayed indexing. Given them, (iii) and (iv) are the filtered-complex spectral sequence of a bounded filtration, in the homological reindexing of the statement \cite[\S2.2, Thm.~2.6]{McCleary}: $E^1_{p,q}=H_{p+q}(\operatorname{gr}^F_pA)=H_q(B_{p,\bullet})$, and $B_{p,\bullet}$ is a finite direct sum of tensor products of complexes of $k$-vector spaces, so K\"unneth applies summandwise. The last sentence follows since then every $E^1_{p,q}$ vanishes.
\end{proof}

\begin{lemma}[Objectwise invariance]\label{lem:bar-invariance}
Let $I$ be a finite poset and let $\alpha\colon G\to G'$ be a map of covariant diagrams $I\to\Chb(\mathrm{Vec}_k)$ which is a quasi-isomorphism at every object of $I$. Then the induced bar map
\[
   \alpha_\#\colon\operatorname{Bar}(I,G)\longrightarrow\operatorname{Bar}(I,G')
\]
of Definition \ref{def:bar-hocolim}, taken along $\id_I$, is a quasi-isomorphism.
\end{lemma}

\begin{proof}
Apply Lemma \ref{lem:bar-filtration} with $J=I$ and $R=k$, so that $B(k,I,G)=\operatorname{Bar}(I,G)$ (Definition \ref{def:two-sided-bar}), and likewise for $G'$. Part \textup{(iii)} then reads
\[
   E^1_{p,q}\cong\bigoplus_{i_0<\cdots<i_p}H_qG(i_0),
\]
the factor $H_{q'}R(i_p)$ contributing $k$ in $q'=0$ and nothing in any other degree, and likewise for $G'$. Being taken along $\id_I$, the map $\alpha_\#$ preserves bar degree and so the filtration, and the map it induces on $E^1$ is $\bigoplus_{i_0<\cdots<i_p}H_q\alpha_{i_0}$, an isomorphism by hypothesis. A filtered map induces a morphism of spectral sequences and $E^{r+1}=H(E^r,d^r)$, so it is an isomorphism on every page $E^r$, $r\ge1$; both filtrations are bounded by part \textup{(ii)}, $I$ being finite and the values bounded, so $\alpha_\#$ is a quasi-isomorphism \cite[\S2.2, Thm.~2.6]{McCleary}.
\end{proof}

\begin{lemma}[Bottom degree of a derived colimit]\label{lem:bottom-degree}
Let $I$ be a finite poset, let $D\colon I\to\Chb(\mathrm{Vec}_k)$, and suppose there is $q_0\in\mathbb Z$ with
\[
   H_qD(i)=0\qquad\text{for every }i\in I\text{ and every }q<q_0 .
\]
Then
\[
   H_{q_0}\bigl(\operatorname{Bar}(I,D)\bigr)\;\cong\;\colim_{I}H_{q_0}D ,
\]
the isomorphism being the edge homomorphism of the bar-degree filtration. In particular $\operatorname{Bar}(I,D)$ is not acyclic as soon as the ordinary colimit of the bottom homology diagram is nonzero. Neither a terminal object in $I$ nor any condition on the shape of the support of $H_\ast D$ is required.
\end{lemma}

\begin{proof}
Filter $\operatorname{Bar}(I,D)=B(k,I,D)$ by bar degree, Lemma \ref{lem:bar-filtration} with $R=k$ the constant contravariant $I$-diagram in internal degree zero. By Lemma \ref{lem:bar-filtration}(iii) the row $q$ of the $E^1$-page, together with $d^1$, is the bar complex $\operatorname{Bar}(I,H_qD)$ of the $\mathrm{Vec}_k$-valued diagram $H_qD$, so
\[
   E^2_{p,q}\;\cong\;H_p\bigl(\operatorname{Bar}(I,H_qD)\bigr)
   \;=\;\mathbb L_p\!\colim_IH_qD ,
   \qquad
   E^2_{0,q}\;\cong\;\colim_IH_qD ,
\]
the last identification because $H_0$ of a bar complex is the ordinary colimit (Definition \ref{def:bar-hocolim}). By hypothesis $E^2_{p,q}=0$ whenever $q<q_0$.

Consider total degree $q_0$. The only $(p,q)$ with $p+q=q_0$, $p\ge0$ and $q\ge q_0$ is $(0,q_0)$, so every other spot in that total degree already vanishes on $E^2$. Moreover $E^2_{0,q_0}$ is a permanent cycle: no differential leaves it, its targets lying in negative bar degree, and none reaches it, since the only differential into it is $d_r\colon E^r_{r,q_0-r+1}\to E^r_{0,q_0}$, whose source lies in the row $q_0-r+1<q_0$ for $r\ge2$. Thus $E^\infty_{0,q_0}=E^2_{0,q_0}$. The filtration is bounded in each total degree by Lemma \ref{lem:bar-filtration}(ii), so it is exhaustive on homology, and by Lemma \ref{lem:bar-filtration}(iv) the induced filtration of $H_{q_0}(\operatorname{Bar}(I,D))$ has $E^\infty_{0,q_0}$ as its only possibly nonzero quotient. Therefore $H_{q_0}(\operatorname{Bar}(I,D))\cong E^\infty_{0,q_0}\cong\colim_IH_{q_0}D$.
\end{proof}

\subsection{What the bar complex resolves}

\begin{remark}[The diagram-level bar resolution]\label{rem:bar-vs-projective}
The mechanism of the lemma below is the usual bar resolution, in normalized form \cite[\S6.5]{Weibel}, transcribed from modules over a group to diagrams over a finite category. The two-sided bar construction with $kI$ itself in the first variable, often denoted $B(I,I,D)\to D$, is the diagram-level bar resolution. What it replaces is the colimit operation over $I$, at a fixed diagram $D$ (the square \eqref{eq:colim-square}), and not the input complex of sheaves, which is what the projective resolution $QX\to X$ of Remark \ref{rem:deformation} replaces.
\end{remark}

\begin{lemma}[The bar complex computes the derived colimit]\label{lem:bar-derived-colim}
Let $I$ be a finite poset of length $m$ and let $D\in\Chb(\Shv(I;k))$ be a bounded complex of covariant $I$-diagrams, that is of right $kI$-modules (Definition \ref{def:incidence-algebra}). Write $k_I$ for the trivial left $kI$-module and $P_\bullet\to k_I$ for the projective resolution of Lemma \ref{lem:incidence-trivial}\textup{(ii)}. Then:
\begin{enumerate}
\item[\textup{(i)}] $\colim_ID=D\otimes_{kI}k_I$, naturally in $D$;
\item[\textup{(ii)}] writing $q\colon QD\to D$ for the projective left deformation of Remark \ref{rem:deformation} at $I$, the two augmentations make
\[
   \colim_I QD
   \;\longleftarrow\;
   QD\otimes_{kI}P_\bullet
   \;\longrightarrow\;
   D\otimes_{kI}P_\bullet
\]
quasi-isomorphisms. Thus $\mathbb L\!\colim_ID\cong\operatorname{Bar}(I,D)$ in $\Db(\mathrm{Vec}_k)$, naturally in $D$, and $\mathbb L_p\!\colim_ID\cong H_p(\operatorname{Bar}(I,D))$, zero for $p>m$.
\end{enumerate}
\end{lemma}

\begin{proof}
(i) By Lemma \ref{lem:incidence-trivial}\textup{(ii)} the tail $P_1\to P_0\to k_I\to0$ is exact, with $P_1=\bigoplus_{y_0<y_1}kIe_{y_0y_0}$ and $P_0=\bigoplus_ykIe_{yy}$. Since $D\otimes_{kI}kIe_{xx}=D_x$ and $D\otimes_{kI}-$ is right exact, applying it gives the exact sequence
\[
   \bigoplus_{y_0<y_1}D_{y_0}
   \xrightarrow{\ d_0-d_1\ }
   \bigoplus_{y}D_{y}
   \longrightarrow
   D\otimes_{kI}k_I
   \longrightarrow
   0 ,
\]
in which $d_1$ is the identity onto the summand indexed by $y_0$ and $d_0$ is the structure map $D_{y_0}\to D_{y_1}$ into the summand indexed by $y_1$. Those are the two maps of Lemma \ref{lem:colim-coeq} at $I$, and the additive form of the coequalizer recorded after that lemma is this cokernel; the coproduct is indexed by the strict relations alone for the reason recorded there. So $D\otimes_{kI}k_I=\colim_ID$, degreewise in the internal degree.

(ii) The functor $\colim_I$ is additive, so Remark \ref{rem:deformation} applies to it verbatim in place of $T$: it is left deformable, $QD$ is a bounded complex of projectives, and $\mathbb L\!\colim_ID=\colim_I(QD)$, which by (i) is $QD\otimes_{kI}k_I$.

For the left-hand map, the augmentation $P_\bullet\to k_I$ makes $QD\otimes_{kI}(P_\bullet\to k_I)$ a double complex whose totalization is the mapping cone of that map, up to a shift. Each $(QD)_n$ is a projective, and so flat, right $kI$-module, so tensoring it with the exact complex $P_\bullet\to k_I\to0$ leaves it exact; the double complex is bounded, in the resolution direction because $P_p=0$ for $p>m$ and in the internal direction because $QD$ is bounded, so its totalization is acyclic. The map is therefore a quasi-isomorphism.

For the right-hand map, each $P_p$ is a projective left $kI$-module, so $-\otimes_{kI}P_p$ is exact and carries the quasi-isomorphism $q$ to a quasi-isomorphism. Filter both totalizations by the resolution degree $p$. The filtrations are bounded, again because $P_p=0$ for $p>m$, and on $E^1$ the induced map is the isomorphism $H_\ast(QD\otimes_{kI}P_p)\to H_\ast(D\otimes_{kI}P_p)$ in each column; so the map is a quasi-isomorphism.

The left-hand target is $\mathbb L\!\colim_ID$ and the right-hand one is $\operatorname{Bar}(I,D)$ by Lemma \ref{lem:incidence-trivial}\textup{(iii)}. Naturality is that of $q$ and of the augmentation, $Q$ being functorial, and the vanishing for $p>m$ is that of $P_p$.
\end{proof}

\section{Finality and the Grothendieck construction}\label{app:finality}

\subsection{The finality criterion and \texorpdfstring{$k$}{k}-linear coefficients}

Both criteria used in Section \ref{sec:derived} are read off the following statement by choosing the second map: the identity gives the finality criterion, and the coefficient projection $q_2^J$ gives what the available test sheaves detect.

\begin{lemma}[Detection along a monotone map]\label{lem:updetect}
Let $\Phi\colon K\to J$ and $\varphi\colon J\to L$ be monotone maps of finite posets. For $x\in L$ put
\[
   W_x=\varphi^{-1}\bigl(\st_L(x)\bigr)=\{j\in J:\ x\le\varphi(j)\},
\]
up-closed in $J$, so that $\Phi^{-1}(W_x)$ is up-closed in $K$. For $G\colon L\to\mathrm{Vec}_k$ write $D=\varphi^\ast G$ and let
\[
   \Phi_\#\colon\operatorname{Bar}(K,\Phi^\ast D)\longrightarrow\operatorname{Bar}(J,D)
\]
be the induced bar map of Definition \ref{def:bar-hocolim}. The following are equivalent.
\begin{enumerate}
\item[\textup{(i)}] $\Phi_\#$ is a quasi-isomorphism for every $G\colon L\to\mathrm{Vec}_k$.
\item[\textup{(ii)}] For every $x\in L$ the map of order complexes induced by $\Phi$,
\[
   C_\ast\bigl(\Delta(\Phi^{-1}W_x);k\bigr)\longrightarrow C_\ast\bigl(\Delta(W_x);k\bigr),
\]
is a quasi-isomorphism.
\end{enumerate}
\end{lemma}

\begin{proof}
By Lemma \ref{lem:support-restriction}\textup{(i)}, for an up-closed $A\subseteq J$ one has $\operatorname{Bar}(J,k_A)=C_\ast(\Delta(A);k)$ and $\operatorname{Bar}(K,\Phi^\ast k_A)=\operatorname{Bar}(K,k_{\Phi^{-1}A})=C_\ast(\Delta(\Phi^{-1}A);k)$, with $\Phi_\#$ the map induced by $\Phi$ on order complexes. Since $\varphi^\ast k_{\st_L(x)}=k_{W_x}$, taking $G=k_{\st_L(x)}$ in \textup{(i)} gives \textup{(ii)}.

Assume \textup{(ii)}. For an up-closed $U\subseteq L$ write $S(U)$ for the assertion that $\Phi_\#$ is a quasi-isomorphism on $\varphi^\ast k_U$, and prove $S(U)$ for every up-closed $U$ by induction on $\lvert U\rvert$. There is nothing to prove for $U=\varnothing$. Let $U\ne\varnothing$ and let $x$ be minimal in $U$. Then $U\setminus\{x\}$ is up-closed, and so is $L_{>x}=\{z\in L:z>x\}$, which is contained in $U\setminus\{x\}$; both are smaller than $U$, so $S(U\setminus\{x\})$ and $S(L_{>x})$ hold by induction, while $S(\st_L(x))$ is hypothesis \textup{(ii)} at $x$. The two sequences of diagrams on $L$
\[
   0\to k_{L_{>x}}\to k_{\st_L(x)}\to\sky_x\to 0,
   \qquad
   0\to k_{U\setminus\{x\}}\to k_U\to\sky_x\to 0
\]
are pointwise exact. Pullback along $\varphi$ and along $\varphi\Phi$ preserves pointwise exactness, and $\operatorname{Bar}(J,-)$ and $\operatorname{Bar}(K,\Phi^\ast(-))$ are degreewise direct sums of the values of the diagram, so they carry each sequence to a short exact sequence of complexes. Apply the five lemma to the two resulting pairs of long exact sequences, first to the left-hand sequence, which gives that $\Phi_\#$ is a quasi-isomorphism on $\varphi^\ast\sky_x$, and then to the right-hand one, which gives $S(U)$.

For a general $G$, enumerate $L$ as $x_1,\dots,x_n$ so that $x_m\le x_j$ implies $m\le j$, and set $U_p=\{x_{p+1},\dots,x_n\}$, up-closed. Writing $G_U$ for the subdiagram equal to $G$ on an up-closed $U$ and $0$ elsewhere, the $G_{U_p}$ form a finite filtration of $G$ with
\[
   G_{U_{p-1}}/G_{U_p}\;\cong\;\sky_{x_p}\otimes_kG(x_p).
\]
Both bar constructions send $\sky_x\otimes_kV$ to $\operatorname{Bar}(-,\varphi^\ast\sky_x)\otimes_kV$, $k$ being a field, so $\Phi_\#$ is a quasi-isomorphism on each subquotient by the previous paragraph. Descending induction along the filtration, with the five lemma at each step, gives \textup{(i)}.
\end{proof}

\begin{corollary}[Homological finality criterion for a bar map]\label{cor:final-criterion}
Let $\Phi\colon K\to J$ be a monotone map of finite posets. Then
\[
   \Phi_\#\colon\operatorname{Bar}(K,\Phi^\ast D)\longrightarrow\operatorname{Bar}(J,D)
\]
is a quasi-isomorphism for every $D\colon J\to\mathrm{Vec}_k$ if and only if $\Phi$ is $k$-homologically final (Definition \ref{def:homotopy-final}).
\end{corollary}

\begin{proof}
Apply Lemma \ref{lem:updetect} with $L=J$ and $\varphi=\id$, so that every $D$ is of the form $\varphi^\ast G$ and $W_j=\st_J(j)$, whence $\Phi^{-1}(W_j)=j\downarrow\Phi$. The poset $\st_J(j)$ has least element $j$, so its augmentation $C_\ast(\Delta(\st_J(j));k)\to k$ is a quasi-isomorphism (Lemma \ref{lem:cone-point}). Augmentations are natural for simplicial maps, so condition \textup{(ii)} at $j$ holds if and only if $C_\ast(\Delta(j\downarrow\Phi);k)\to k$ is a quasi-isomorphism, that is if and only if $j\downarrow\Phi$ is nonempty with vanishing reduced $k$-homology.
\end{proof}

For $D$ the constant diagram $k$ this is the homological Quillen--McCord theorem for finite posets \cite[Cor.~5.5]{Barmak}, quoted as \cite[Thm.~1.2]{GT}. What the corollary adds is that once the coefficients are allowed to vary, finality is no more than what is needed, and no less.

\begin{lemma}[Grothendieck construction with $k$-linear coefficients]\label{lem:groth-hocolim}
Keep the data of Definition \ref{def:grothendieck}, put $K=\int_IX$, let $G\colon L\to\mathrm{Vec}_k$ and $E=\pi^\ast G$. Then
\begin{enumerate}
\item[\textup{(i)}] $\mathfrak B(i)=\operatorname{Bar}\bigl(X(i),G|_{X(i)}\bigr)$ is a diagram $I\to\Chb(\mathrm{Vec}_k)$, the structure maps being induced by the inclusions $X(i)\subseteq X(i')$ with identity coefficient maps;
\item[\textup{(ii)}] $\operatorname{Bar}(I,\mathfrak B)$ and $\operatorname{Bar}(K,E)$ are quasi-isomorphic, that is
$\hocolim_{i\in I}\hocolim_{X(i)}G|_{X(i)}\simeq\hocolim_KE$, the induced isomorphism in degree zero being the canonical isomorphism of colimits.
\end{enumerate}
\end{lemma}

\begin{proof}
Part (i) is immediate, the coefficient of a chain in $X(i)$ being the same vector space $G(x)$ before and after the inclusion.

For $i\in I$ put $K_{\le i}=\{(i',x)\in K: i'\le i\}=\mathrm{pr}\downarrow i$, a full subposet of $K$, and let
\[
   \iota_i\colon X(i)\to K_{\le i},\ \iota_i(y)=(i,y),
   \qquad
   \rho_i\colon K_{\le i}\to X(i),\ \rho_i(i',x)=x ,
\]
the latter well defined because $X(i')\subseteq X(i)$, monotone, and satisfying $\rho_i\iota_i=\id$.

Since $\iota_i^\ast(E|_{K_{\le i}})=G|_{X(i)}$, $\iota_i$ induces a bar map $\iota_{i\#}\colon\mathfrak B(i)\to\operatorname{Bar}(K_{\le i},E)$. For $c=(i',x)\in K_{\le i}$,
\[
   c\downarrow\iota_i=\{y\in X(i): x\le y\}
\]
is the up-set of $x$ in $X(i)$, nonempty with least element $x$, so $\widetilde H_\ast(N(c\downarrow\iota_i);k)=0$ by Lemma \ref{lem:cone-point} and $\iota_{i\#}$ is a quasi-isomorphism by Corollary \ref{cor:final-criterion}.

Because $E(i',x)=G(x)=E(i,\rho_i(i',x))$, the coefficient maps along $(i',x)\le(i,x)$ are identities, and $\rho_i$ together with them gives an induced bar map (Definition \ref{def:bar-hocolim}) $\rho_{i\#}\colon\operatorname{Bar}(K_{\le i},E)\to\mathfrak B(i)$ with $\rho_{i\#}\iota_{i\#}=\id$; because $\rho_{i\ast}\iota_{i\ast}=\id_\ast$ and $\iota_{i\ast}$ is an isomorphism on homologies, $\rho_{i\#}$ is a quasi-isomorphism. Write $\mathfrak B_K(i)=\operatorname{Bar}(K_{\le i},E)$, a diagram on $I$ whose structure maps are induced by the inclusions $K_{\le i}\subseteq K_{\le i'}$ with identity coefficients. Then $\rho_\#\colon\mathfrak B_K\Rightarrow\mathfrak B$ is a map of $I$-diagrams, both composites for $i\le i'$ being $(i'',x)\mapsto x$, and it is an objectwise quasi-isomorphism, so the induced bar map $\operatorname{Bar}(I,\mathfrak B_K)\to\operatorname{Bar}(I,\mathfrak B)$, written $\tilde\rho_\#$, is a quasi-isomorphism by Lemma \ref{lem:bar-invariance}.

$T:=\operatorname{Bar}(I,\mathfrak B_K)$ is the total complex of
\[
   T_{p,q}
   =
   \bigoplus_{\substack{i_0<\cdots<i_p\ \text{in}\ I\\ c_0<\cdots<c_q\ \text{in}\ K\\ \mathrm{pr}(c_q)\le i_0}}E(c_0).
\]
The $q$-differential is that of $\operatorname{Bar}(K_{\le i_0},E)$ and the $p$-differential acts only on the $i_j$, the structure maps of $\mathfrak B_K$ being the identity on chains and coefficients. Define
\[
   \pi_T\colon T\longrightarrow\operatorname{Bar}(K,E),
   \qquad
   \pi_T\bigl([i_0;c_0<\cdots<c_q]\otimes e\bigr)=[c_0<\cdots<c_q]\otimes e ,
\]
and $\pi_T=0$ on the summands with $p\ge1$; write $\mathbf c$ for the chain $c_0<\cdots<c_q$. Consider $\operatorname{Bar}(K,E)$ as the double complex concentrated in the column $p=0$, of which it is the total complex. $\pi_T$ is a chain map: in bar degree $p=0$ the $I$-bar differential vanishes and the $q$-differential is that of $\operatorname{Bar}(K,E)$, since $\pi_T$ commutes with every $q$-face and every face of a chain in $K_{\le i_0}$ lies again in $K_{\le i_0}$; in bar degree $p\ge 1$, the $q$-differential preserves $p$ and is killed by $\pi_T$ (i.e. $\pi_T(T_{p,q-1})=0$), the $I$-bar differential out of $p\ge2$ lands in bar degree $p-1\ge1$ and is killed as well, and out of $p=1$ it is $[i_1;\mathbf c]\otimes e-[i_0;\mathbf c]\otimes e$ up to a global sign, which $\pi_T$ sends to zero.

We show that $\pi_T$ is indeed a quasi-isomorphism. Filter $T$ by the $K$-bar degree, $F_qT_n=\bigoplus_{j\le q}T_{n-j,j}$, increasing and bounded since $K$ is finite, with the associated graded quotient $F_qT_n/F_{q-1}T_n\cong T_{n-q,q}$, on which the part of the total differential that lowers $q$ induces zero, leaving only the $I$-bar differential. Fix a $K$-chain $\mathbf c$ of length $q$ and put $U_{\mathbf c}=\st_I(\mathrm{pr}(c_q))$. The strict chains $i_0<\cdots<i_p$ with $\mathrm{pr}(c_q)\le i_0$ are the strict chains in $U_{\mathbf c}$, so the summand of $T_{\ast,q}$ indexed by $\mathbf c$ is
\[
T_{\ast,q}(\mathbf c)=\bigoplus_{i_0<\cdots<i_p\ \text{in}\ U_{\mathbf c}}E(c_0)\cong C_\ast(\Delta(U_{\mathbf c});k)\otimes E(c_0),
\]
the sum running over all strict chains of $U_{\mathbf c}$, the length $p$ being the $I$-bar degree. Since $U_{\mathbf c}$ has a least element, Lemma \ref{lem:cone-point} gives $T_{\ast,q}(\mathbf c)$ homology $E(c_0)$ concentrated in $I$-bar degree $0$, the isomorphism being induced by the augmentation, that is by $\pi_T$. Giving $\operatorname{Bar}(K,E)$ the corresponding bounded filtration by $q$, $\pi_T$ preserves filtrations and induces an isomorphism on $E^1$ and, as in the proof of Lemma \ref{lem:bar-invariance}, on every later page, so it is a quasi-isomorphism \cite[\S2.2, Thm.~2.6]{McCleary}. Finally, the zigzag
\[
\operatorname{Bar}(I,\mathfrak B)\xleftarrow{\ \tilde\rho_\#\ } \operatorname{Bar}(I,\mathfrak B_K) \xrightarrow{\ \pi_T\ } \operatorname{Bar}(K, E)
\]
gives (ii). In degree zero, $H_0$ of a bar complex is the ordinary colimit, and the resulting map is the canonical comparison $\colim_I\colim_{X(i)}G\to\colim_KE$, an isomorphism by the universal properties, both sides being initial among families compatible with all structure maps of $E$.
\end{proof}

\section{The defect module}\label{app:incidence}

\subsection{Resolution, and bar cancellation}

\begin{example}[A defect module resolved over its incidence algebra]\label{ex:defect-resolution}
Definition \ref{def:defect-module} presents $\operatorname{Def}_\Phi|_{\Theta_u}$ as a diagram on $\Theta_u^{op}$, and so as a right $k\Theta_u^{op}$-module (Definition \ref{def:incidence-algebra}), and Lemma \ref{lem:incidence-trivial} then turns the essentiality condition of Definition \ref{def:essential-defect} into a computation with an explicit finite resolution. Take the rank-$3$ lattice of Example \ref{ex:essential-nonsingleton} with $\sigma=p_1$, $(a,b)=(1,2)$ and $u=t$, where $\Theta_t=\{z,p_1\}$ and $\Theta_t^{op}$ is the chain $p_1<z$. Then $k\Theta_t^{op}$ has basis $e_{p_1p_1},e_{zz},e_{p_1z}$ and is the algebra of upper triangular $2\times2$ matrices over $k$; its length is $1$, so the resolution of Lemma \ref{lem:incidence-trivial}\textup{(ii)} has two terms:
\[
   0\longrightarrow k\Theta_t^{op}e_{p_1p_1}
   \longrightarrow k\Theta_t^{op}e_{p_1p_1}\oplus k\Theta_t^{op}e_{zz}
   \longrightarrow k\longrightarrow 0 .
\]
Tensoring with $\operatorname{Def}_\Phi|_{\Theta_t}$ evaluates it at the first element of each chain of $\Theta_t^{op}$: in bar degree $0$ that gives $\operatorname{Def}_\Phi(p_1,t)\oplus\operatorname{Def}_\Phi(z,t)$, one summand per element, and in bar degree $1$ it gives $\operatorname{Def}_\Phi(p_1,t)$, from the single chain $p_1<z$. The comma fibre over $(p_1,t)$ has a greatest element, so $\operatorname{Def}_\Phi(p_1,t)$ is acyclic wherever it appears, and only the defect at the bottom of $\Theta_t$ survives: the total complex is quasi-isomorphic to $\operatorname{Def}_\Phi(z,t)$, and the $\operatorname{Tor}$ is $k$ in degree $0$ and vanishes elsewhere, the conclusion drawn there from Lemma \ref{lem:essential-certificates}\textup{(ii)}.
\end{example}

\begin{remark}[Bar cancellation and inessential defect modules]\label{rem:bar-cancellation}
Lemma \ref{lem:bar-filtration} computes the $E^1$-page of the bar-degree filtration summandwise, but not $H_\ast B(R,J,D)$: the bar differential and those above it connect different $J$-indexed summands, so a nonzero class in some $H_\ast R(j)$ may still be cancelled. In the proof of Lemma \ref{lem:defect-detection}\textup{(ii)} the choice of $u_0$ maximal in $q_2^J(S)$ removes the cancellation coming from outside the fibre, and what remains is whether the bar faces and the structure maps of $\operatorname{Def}_\Phi|_{\Theta_{u_0}}$ can cancel all defect classes within it.

That is what essentiality forbids (Definition \ref{def:essential-defect}), and the diagrams below show that it forbids something: each has nonzero homology somewhere, yet vanishing derived colimit, so each would be an inessential defect module if it were realised as some $\operatorname{Def}_\Phi|_{\Theta_{u_0}}$. None of them is so realised here (Remark \ref{rem:detection-scope}); what they establish is that essentiality cannot be replaced by a condition on the values and the structure maps of the module taken alone. Each is concentrated in homological degree $0$.

First, a homotopy colimit of a diagram with nonvanishing homology can vanish, as at the end of the proof of Lemma \ref{lem:support-restriction}: on the two-element \emph{index} poset $i_0<i_1$ with $H_\ast G(i_0)=k$ and $H_\ast G(i_1)=0$, the nonzero value sitting at the coefficient end because a bar chain carries the value of its bottom index, $\operatorname{Bar}(\{i_0<i_1\},G)$ is acyclic and $\colim G=0$. The orientation is the whole content: with the nonzero value at $i_1$ instead, the same two-element index poset gives $H_0=k$. With the values of $H_\ast G$ shown in parentheses:
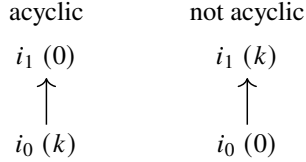
\[
\begin{array}{c@{\qquad\qquad}c}
\text{acyclic} & \text{not acyclic}\\[0.6em]
\begin{tikzcd}[row sep=1.8em,column sep=1.2em,cramped]
i_1\ (0)\\
i_0\ (k)\arrow[u]
\end{tikzcd}
&
\begin{tikzcd}[row sep=1.8em,column sep=1.2em,cramped]
i_1\ (k)\\
i_0\ (0)\arrow[u]
\end{tikzcd}
\end{array}
\]
Read through Definition \ref{def:essential-defect}, where the index is $\Theta_u^{op}$, the acyclic case is a defect present at $i_0$ and filled in at the larger comma fibre below it: an inessential defect, and the reason essentiality is a hypothesis and not a triviality.

Second, requiring the homology to be nonzero at \emph{every} object does not help: for $i_1>i_0<i_2$ with $G(i_0)=k^2$ and $G(i_1)=G(i_2)=k$, the two coordinate projections make $g\mapsto(\tau_1g,-\tau_2g)$ surjective, so the colimit vanishes with every value nonzero.

Third, it does not help to require in addition that every structure map be \emph{injective} on homology, the condition that a defect class, once present, cannot die under enlargement. Let $I$ be the crown $\{a_1,a_2\}<\{b_1,b_2\}$, with $a_i<b_j$ for all four pairs and no other relations, so that $\Delta(I)$ is a $4$-cycle. Fix $\lambda\in k^\ast$, and let $G$ be $k$ at every object, with structure maps
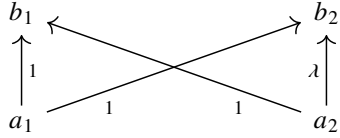
\[
\begin{tikzcd}[row sep=large,column sep=huge]
b_1 & & b_2\\
a_1\arrow[u,"1"']\arrow[urr,"1"',pos=0.2] & & a_2\arrow[u,"\lambda"]\arrow[ull,"1",pos=0.2]
\end{tikzcd}
\]
that is
\[
   G(a_1\le b_1)=G(a_1\le b_2)=G(a_2\le b_1)=1,
   \qquad
   G(a_2\le b_2)=\lambda ,
\]
a rank-one system with monodromy $\lambda$; there are no chains of length $2$, so functoriality is automatic and every structure map is an isomorphism. In the basis of the four length-one chains, $\partial=d_0-d_1\colon\operatorname{Bar}_1(I,G)\to\operatorname{Bar}_0(I,G)$ has matrix
\[
\begin{pmatrix}
   -1 & -1 & 0 & 0\\
    0 & 0 & -1 & -1\\
    1 & 0 & 1 & 0\\
    0 & 1 & 0 & \lambda
\end{pmatrix},
\qquad \det=\lambda-1 ,
\]
rows indexed by $a_1,a_2,b_1,b_2$ and columns by $a_1b_1,a_1b_2,a_2b_1,a_2b_2$. Since $\operatorname{Bar}_2(I,G)=0$, the complex $\operatorname{Bar}(I,G)$ is acyclic if and only if $\lambda\ne1$ in $k$, and then $\colim G=0$ as well, although $G$ is nonzero at every object, all its structure maps are isomorphisms, and its support is down-closed. For $\lambda=-1$ the determinant is $-2$, so the example is acyclic in every characteristic other than $2$, while in characteristic $2$ one has $\lambda=1$ and $H_0\cong H_1\cong k$: the degeneration is not an accident of the field but the trivial-monodromy locus met in characteristic $2$.

Every structure map of $G$ being an isomorphism, $\operatorname{Bar}(I,G)$ is the simplicial chain complex of $\Delta(I)$ with coefficients in the rank-one local system of monodromy $\lambda$ around the $4$-cycle, so the vanishing is that of $H_\ast(S^1;\mathcal L_\lambda)$ for $\lambda\ne1$. What fails is not down-closedness of the support but simple connectivity together with triviality of the monodromy: $I$ has two minimal objects, and the monodromy kills the class that either one alone would carry. A least element rules this out by making the support contractible, which is why Lemma \ref{lem:essential-certificates}(ii) asks for one; Lemma \ref{lem:essential-certificates}(i) rules it out differently, the colimit in the bottom degree vanishing here.

Any proof of the unconditional form of Lemma \ref{lem:defect-detection}\textup{(ii)} must therefore use a property of the defect modules of meet functors, and not a property of arbitrary diagrams on the fibre: some reason why a defect cannot simply be filled in at the larger comma fibre below it, and some reason why a nontrivial local system cannot occur among the comma fibres of a meet.
\end{remark}

\bibliographystyle{alpha}
\bibliography{references}

\end{document}